\documentclass[11pt, reqno]{amsart}
\usepackage{amsmath, amssymb, mathrsfs, amsthm}
\usepackage{bm}
\usepackage{booktabs}
\usepackage{float}
\usepackage{subfig}
\usepackage[colorlinks, linkcolor = blue, anchorcolor = blue, citecolor = blue, CJKbookmarks = True]{hyperref}
\usepackage{geometry}
\usepackage{graphicx}
\usepackage{caption}

\graphicspath{{images/}}
\usepackage{adjustbox}
\usepackage{tabularx}
\usepackage{array}
\usepackage{makecell}
\usepackage{tikz}
\usetikzlibrary{arrows.meta, positioning, shapes.geometric, shadows}
\usetikzlibrary{positioning, arrows.meta, calc}

\usepackage{color, xcolor}
\newcommand{\red}[1]{{\color{red}{#1}}}

\newtheorem{assumption}{Assumption}
\newtheorem{definition}{Definition}[section]
\newtheorem{lemma}[definition]{Lemma}
\newtheorem{theorem}[definition]{Theorem}
\newtheorem{proposition}[definition]{Proposition}
\newtheorem{corollary}[definition]{Corollary}
\newtheorem{remark}[definition]{Remark}

\numberwithin{equation}{section}
\allowdisplaybreaks[4]

\newcommand{\rd}{\mathrm d}

\newcommand{\hE}{\mathbb E}

\newcommand{\hR}{\mathbb R}

\newcommand{\cC}{\mathcal C}

\newcommand{\cN}{\mathcal N}
\newcommand{\cO}{\mathcal O}
\newcommand{\jdc}[1]{{\color{blue}{{#1}}}}

\begin{document}

\title[Splitting AVF method for GLEs]{Splitting AVF method for generalized Langevin equations:\ probability density function and geometric ergodicity}

\author{Xinjie Dai}
\address{School of Mathematics and Statistics, Yunnan University, Kunming 650500, Yunnan, China}
\email{dxj@ynu.edu.cn}

\author{Xingyu Liu}
\address{School of Mathematics and Statistics, Yunnan University, Kunming 650500, Yunnan, China}
\email{lxy@stu.ynu.edu.cn}

\author{Diancong Jin}
\address{School of Mathematics and Statistics, Huazhong University of Science and Technology, Wuhan 430074, China; 
Hubei Key Laboratory of Engineering Modeling and Scientific Computing, Huazhong University of Science and Technology, Wuhan 430074, China}
\email{jindc@hust.edu.cn (Corresponding author)}

\author{Liying Sun}
\address{Academy for Multidisciplinary Studies, Capital Normal University, Beijing 100048, China}
\email{liyingsun@lsec.cc.ac.cn}

\thanks{This work is supported by National Natural Science Foundation of China (Nos.\ 12401547, 12201228, 12471391), Yunnan Fundamental Research Project (No.\ 202501AU070074), and Scientific Research and Innovation Project of Postgraduate Students in the Academic Degree of Yunnan University (No.\ KC-252511570)}

\begin{abstract}
The generalized Langevin equation (GLE) constitutes a fundamental model for describing nonequilibrium dynamics with memory effects. To overcome the numerical challenges arising from superquadratically growing potentials and degenerate noise, we propose and analyze a structure-preserving splitting averaged vector field (AVF) method for a quasi-Markovian GLE. The core advantage of this method lies in its ability to simultaneously preserve the exponential integrability, Malliavin differentiability, and ergodicity of the underlying continuous system. Notably, by leveraging exponential integrability, Malliavin differentiability, and uniform non-degeneracy of the numerical solution, we obtain the existence and smoothness of its probability density function, which converges to that of the exact solution with first-order accuracy. Furthermore, by validating the Lyapunov condition and the minorization condition using a localized technique, we establish the geometric ergodicity of the numerical solution. Finally, numerical experiments are conducted to confirm the theoretical results.
\end{abstract}


\maketitle



\section{Introduction}
\label{sec.Introduction}

Generalized Langevin equations (GLEs) have emerged as a cornerstone in statistical physics, molecular dynamics, and machine learning, owing to their versatility in encompassing memory effects and stochastic forcing within nonequilibrium dynamical systems \cite{GlattNathanHerzogDavidMcKinley2020, LeimkuhlerSachs2019, LimSoonWehrJanLewensteinMaciej2020, Pavliotis2014}. According to Newton's second law, the GLE describing the motion of a particle with position $x(t)$ and momentum $v(t)$ typically takes the form 
\begin{align} \label{eqn:GLE} 
\begin{cases}
\rd v(t) = - \gamma v(t) \rd t - \nabla U(x(t)) \rd t + \sqrt{2\gamma} \rd W_{0}(t) - \int_0^t K(t - s)v(s) \rd s \rd t + F (t) \rd t, \\
\rd x(t) = v(t) \rd t. 
\end{cases}
\end{align}
Here, $t > 0$, $\gamma>0$ denotes the friction coefficient, $W_{0}$ represents a standard Wiener process, and the confining potential $U$ is assumed to exhibit superquadratic growth. The memory effects are captured by a convolution force with kernel $K$, while the stochastic forcing $F$ is a zero-mean stationary Gaussian process; these two terms are intrinsically linked via the fluctuation-dissipation relation 
\begin{align*} 
\hE [F(t)F(s)] = K(|t - s|), \qquad \forall\, t, s \geq 0. 
\end{align*}

In general, the dynamics described by \eqref{eqn:GLE} are non-Markovian due to the presence of memory kernels. Nevertheless, it is well known that for kernels expressed as sums of exponential functions, the Mori--Zwanzig formalism provides a Markovian lifting of \eqref{eqn:GLE} \cite{DuongNguyen2024, DuongShang2022, LeimkuhlerSachs2019}.
Specifically, suppose that the memory kernel is given by 
\begin{align*}
K(t) = \sum_{\ell = 1}^{k} \lambda_{\ell}^2 e^{-\alpha_{\ell} t}, \qquad t \geq 0 
\end{align*}
with positive constants $\lambda_{\ell}$ and $\alpha_{\ell}$ for $\ell = 1, \dots, k$. By invoking the fluctuation-dissipation relation together with Duhamel's principle, \eqref{eqn:GLE} can be reformulated as a higher-dimensional quasi-Markovian system with auxiliary state $z(t) = (z_{1}(t), \cdots, z_{k}(t))^{\top}$ and initial value $z(0) \sim \cN(0, \mathrm{Id}_{k \times k})$, given by 
\begin{align} \label{GLE}
\begin{cases}
\rd v(t) = - \gamma v(t) \rd t - \nabla U(x(t)) \rd t + \sum\limits_{\ell = 1}^{k} \lambda_{\ell} z_{\ell}(t) \rd t + \sqrt{2 \gamma} \rd W_{0}(t), \\ 
\rd z_{\ell}(t) = - \alpha_{\ell} z_{\ell}(t) \rd t - \lambda_{\ell} v(t) \rd t + \sqrt{2 \alpha_{\ell}} \rd W_{\ell}(t), \qquad \ell = 1, \cdots, k, \\ 
\rd x(t) = v(t) \rd t. 
\end{cases}
\end{align}
Here, $W = (W_0, W_1, \cdots, W_k)^\top$ is a $(k+1)$-dimensional Brownian motion defined on a filtered complete probability space $(\Omega, \mathcal{F}, (\mathcal{F}_t)_{t \ge 0}, \mathbb{P})$ satisfying the usual conditions. For the GLE \eqref{GLE}, H\"ormander's condition guarantees that the exact solution admits a smooth probability density function \cite[Chapter 2]{NualartDavid2006}. On the other hand, the GLE \eqref{GLE} is geometrically ergodic with a unique invariant measure $\pi$, characterized by the Gibbs--Boltzmann density function (see \cite{DuongNguyen2024}) 
\begin{align} \label{Gibbs--Boltzmann}
\pi(\rd v, \rd z_1, \cdots, \rd z_k, \rd x) = \frac{1}{Z_k} \exp \big( - H_{0} (v, z_1, \cdots, z_k, x) \big) \rd v \rd z_1 \cdots \rd z_k \rd x, 
\end{align} 
where $Z_k$ is the normalizing constant, and $H_{0}$ is the Hamiltonian defined by 
\begin{align} \label{Hamiltonian0}
H_{0}(v, z, x) = \frac{1}{2}|v|^{2} + \frac{1}{2} \|z\|^{2} + U(x), \qquad v \in \hR,~~z \in \hR^{k},~~x \in \hR.
\end{align}

In this paper, we study discrete-time methods for the GLE \eqref{GLE} that preserve intrinsic properties, focusing on the numerical approximation of probability density functions and invariant measure $\pi$. For clarity, the analysis is carried out for the $\hR^{k+2}$-valued solution $Y = (v, z^{\top}, x)^{\top}$, with the initial value assumed to be deterministic. The results are extendable to $Y \in \hR^{(k+2)\times d}$, where $d \in \mathbb{N}_+$.

\subsection{Probability density function}

The solution to \eqref{GLE} constitutes a family of random variables on the underlying probability space $(\Omega, \mathcal{F}, \mathbb{P})$, whose probability density functions play a key role in the formulation and analysis of the associated Fokker--Planck equation \cite{Pavliotis2014}. Since direct discretizations of the Fokker--Planck equation are often hindered by the curse of dimensionality, it is more prevalent to approximate the probability density function of the underlying stochastic system via its temporal discretization. Moreover, investigating the probability density function of numerical solutions to the GLE, along with its convergence property, not only deepens the understanding of the distributional behavior of stochastic systems but also provides a solid theoretical foundation for applications such as Monte Carlo sampling \cite{PangWangWu2025, WuBrooksVandenEijnden2016}.

Malliavin calculus is a powerful tool for investigating the properties of probability density functions of Wiener functionals, including solutions to stochastic differential and integral equations \cite{HongJinSheng2024, LiWangWang2026, NualartDavid2006}. Cui et al.\ \cite{CuiHongSheng2022} proved the first-order density convergence of the numerical solution for the underdamped Langevin equation using Malliavin calculus. This is the first result to address the density approximation for stochastic differential equations with non-globally Lipschitz coefficients. Dai and Jin \cite{daijin2024} applied Malliavin calculus to establish the convergence rate of the probability density function of the Euler--Maruyama (EM) method for the overdamped GLE under global Lipschitz conditions. Nevertheless, for the GLE with superquadratic growth and degenerate noise, there are currently no available results on the probability density function of numerical solutions, particularly concerning the consistency between the convergence rates of density approximations and those of the corresponding numerical methods.

\subsection{Geometric ergodicity}

Ergodicity characterizes the long-time statistical behavior of stochastic dynamical systems, establishing the unity between time averages and ensemble averages \cite{HongSun2022, HongWang2019}. In the context of Langevin-type equations, geometric ergodicity implies that the distribution of the solution converges exponentially fast to a unique Gibbs--Boltzmann measure as time tends to infinity \cite{DuongNguyen2024, HerzogDavidMattinglyJonathan2019, LeimkuhlerSachs2019, LuYulongMattinglyJonathan2020}. This property provides a theoretical foundation for both the statistical-mechanical interpretation and practical applicability of such models. Accordingly, the design of numerical methods that rigorously preserve ergodicity of the underlying continuous system has become a key step in bridging theoretical analysis with reliable simulations.

To accommodate the diverse dynamical characteristics of stochastic systems, a wide range of numerical methods has been developed, including modified explicit methods, implicit methods, and splitting methods \cite{BaoJianhaiShaoJinghaiYuanChenggui2016, HuangLi2025, LiMaoYin2019, LiuMaoWu2023, LiuLiu2025, PangWangWu2025}. When the memory term in the GLE vanishes, the system reduces to the classical underdamped Langevin equation, for which ergodicity-preserving numerical methods have been extensively studied \cite{AbdulleVilmartZygalakis2015, HongSunWang2017, LeimkuhlerMatthewsStoltz2016, MattinglyStuartHigham2002, MilsteinTretyakov2007, Talay2002}. Recently, a new class of splitting methods proposed by Chen et al.\ \cite{ChenDangHongZhang2025} is capable of preserving both the ergodicity and the exponential integrability of the underlying continuous system. Dai et al.\ \cite{daiJiangWang2025} proposed an explicit splitting SAV method for approximating the invariant measure of the classical underdamped Langevin equation.

Relatively, research on numerical methods for approximating the invariant measure of the GLE \eqref{GLE} remains limited, and we are only aware of \cite{DuongShang2022, LeimkuhlerSachs2022}. Duong and Shang \cite{DuongShang2022} investigated the quasi-Markovian GLE subject to a harmonic oscillator potential and developed a novel splitting scheme named the BAEOEAB method which achieves second-order convergence for approximating the invariant measure. Leimkuhler and Sachs \cite{LeimkuhlerSachs2022} proposed and analyzed a class of symmetric splitting schemes for the quasi-Markovian GLE with a globally Lipschitz continuous $\nabla U$, establishing weak second-order convergence, geometric ergodicity, and central limit theorems.

\subsection{Main difficulty, idea, and contribution} 

Due to the presence of superquadratic potentials and degenerate noise, we encounter several challenges in numerically approximating the probability density functions and the invariant measure $\pi$ of the GLE \eqref{GLE}. Specifically,
\begin{itemize}
\item[(i)] \emph{Probability density function.} 
First, for the GLE \eqref{GLE} with superquadratically growing potentials, the drift coefficient exhibits superlinear growth and fails to satisfy the globally monotone condition. This may result in numerical solution blow-up and long-time sampling distortion, thereby hindering the approximation of probability density functions. Second, the smoothness of the probability density function of the numerical solution depends on its uniform non-degeneracy in the sense of Malliavin calculus, which is difficult to verify under degenerate noise. Third, the convergence analysis of numerical probability density functions relies on the strong or weak convergence rate of the numerical solution, which is also complicated by the superquadratically growing potentials.

\item[(ii)] \emph{Geometric ergodicity.} 
The minorization and Lyapunov conditions are fundamental tools for establishing the geometric ergodicity of numerical approximations for stochastic differential equations. However, in the presence of degenerate noise, direct verification of the minorization condition is often highly non-trivial. A widely adopted sufficient criterion (see \cite[Assumption 2.1]{MattinglyStuartHigham2002}) simplifies this task by requiring the transition kernel to admit a jointly continuous and locally positive density. While this criterion works well for the classical underdamped Langevin equation, where the numerical mapping over the first two steps is bijective with respect to the Brownian increments \cite{ChenDangHongZhang2025,MattinglyStuartHigham2002}, it faces significant challenges for the GLE \eqref{GLE}. To be specific, unlike the classical underdamped Langevin equation, the GLE \eqref{GLE} involves $k$ additional Brownian motions, which disrupts the bijectivity between the numerical solution and the Brownian increments. This lack of bijectivity precludes a direct proof of the joint continuity of the transition kernel density, and thus complicates the verification of the minorization condition. 
\end{itemize}

The main idea for this study is as follows. For the GLE \eqref{GLE} with superquadratic potentials and degenerate noise, we first construct a novel splitting averaged vector field (AVF) method and establish its exponential integrability, which serves as a foundation for the analysis of strong convergence rate under non-globally monotone condition. Next, by investigating the Malliavin differentiability and uniform non-degeneracy of the numerical solution, we show the existence and uniqueness of the numerical probability density function. Building upon these results, we derive the convergence rate of the density. To overcome the difficulties in proving geometric ergodicity, we establish the minorization condition through a localized technique combined with the compactness argument. Further, by verifying the Lyapunov condition, we prove the geometric ergodicity of the numerical solution. A schematic overview of the framework is presented in Figure \ref{fig:avf_theory_framework}.

\begin{figure}[htbp]
\centering
\begin{tikzpicture}[
 node distance=1cm and 1cm,
 >=Stealth, 
 base/.style={
 rectangle, draw, thick, rounded corners=5pt,
 align=center, font=\small\sffamily,
 minimum width=2.8cm, minimum height=1cm,
 drop shadow={opacity=0, shadow xshift=0pt, shadow yshift=0pt} 
 },
 core/.style={base, fill=blue!20, draw=blue!70!black},
 cond/.style={base, fill=blue!5, draw=blue!50!black},
 result/.style={base, fill=orange!20, draw=orange!70!black},
 arrow/.style={thick, draw=gray!75, ->}, 
 line/.style={thick, draw=gray!75} 
]

\node (root) [core] {Splitting AVF method};

\node (expint) [cond, below=1.0cm of root] {Exponential\\integrability};
\node (lyap) [cond, right=1.0cm of expint] {Lyapunov\\condition};
\node (minor) [cond, right=1.0cm of lyap] {Minorization\\condition};

\node (uniform) [cond, below=1.0cm of expint] {Uniform\\non-degeneracy};
\node (malliavin)[cond, left=1.0cm of uniform] {Malliavin\\differentiability};

\node (existden) [result, below=1.0cm of malliavin] {Existence \& smoothness\\of density};
\node (convden) [result, below=1.0cm of uniform] {Convergence rate\\of density};
\node (strong) [result, right=1.0cm of convden] {Strong\\ convergence rate};
\node (geoerg) [result, below=3.0cm of minor] {Geometric\\ergodicity};

\draw [line] (root) -- (expint);
\draw [line] (root) -- (lyap);
\draw [line] (root) -- (minor);

\draw [arrow] (expint) -> (malliavin);
\draw [arrow] (expint) -> (uniform);
\draw [arrow] (expint) -> (strong);

\draw [arrow] (malliavin) -> (existden);
\draw [arrow] (uniform) -> (existden);
\draw [arrow] (uniform) -> (convden);
\draw [arrow] (strong) -> (convden);

\draw [arrow] (lyap) -> (geoerg);
\draw [arrow] (minor) -> (geoerg);
\end{tikzpicture}
\caption{The theoretical framework of main results.}
\label{fig:avf_theory_framework}
\end{figure}
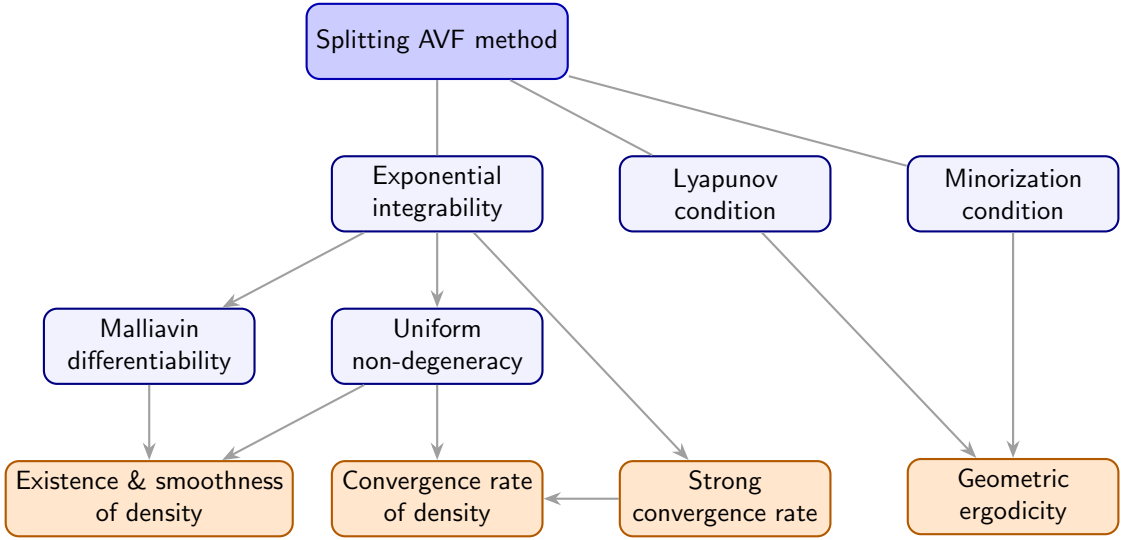

For the numerical study of the GLE \eqref{GLE} with superquadratic potentials and degenerate noise, the main contributions of this paper are summarized below:\ 
\begin{itemize}
\item[(i)] We present a novel splitting AVF method for the GLE \eqref{GLE}, which is proven to achieve first-order strong convergence, building upon its exponential integrability.
 
\item[(ii)] Based on Malliavin differentiability and uniform non-degeneracy, we show that the numerical solution generated by the splitting AVF method admits a smooth probability density function that converges to that of the exact solution with first-order accuracy. 

\item[(iii)] We establish the geometric ergodicity of the numerical solution generated by the splitting AVF method, based on the Lyapunov and minorization conditions.
\end{itemize}

The rest of the paper is organized as follows. Section \ref{sec.Preliminaries} reviews the fundamental properties of the GLE \eqref{GLE} and develops a splitting AVF method, along with an analysis of its solvability. Section \ref{sec.StrongConvergence} first shows that the splitting AVF method preserves the exponential integrability of the GLE \eqref{GLE} and then establishes its first-order strong convergence. Section \ref{sec.PDF} presents the existence, smoothness, and convergence of the probability density function associated with the numerical solution generated by the splitting AVF method, based on its Malliavin differentiability and uniform non-degeneracy. In Section \ref{sec.Ergodicity}, the splitting AVF method is shown to preserve the ergodicity of the GLE \eqref{GLE}. Finally, Section \ref{sec.experiment} reports numerical experiments that validate and illustrate the theoretical findings.

\section{The splitting AVF method for GLEs} 
\label{sec.Preliminaries}

In this section, we first review the fundamental properties of the GLE \eqref{GLE}, and propose the splitting AVF method \eqref{numsolu}--\eqref{sdesolu} tailored to this system.

\textbf{Notation}. 
The sets of nonnegative and positive integers are denoted by $\mathbb{N}$ and $\mathbb{N}_{+}$, respectively. For two real numbers $a$ and $b$, we denote $a \wedge b:= \min \{a, b\}$ and $a \vee b:= \max \{a, b\}$. Let $C$ be a generic constant, independent of step size $h$, that may differ from line to line. Let $T \in (0, \infty)$ and $\mathbb{H} := L^2([0, T], \mathbb{R}^{k+1})$ be the Hilbert space equipped with the inner product $\langle g, h \rangle_{\mathbb{H}} = \int_0^T \langle g(t), h(t) \rangle_{\mathbb{R}^{k+1}} \, \mathrm{d}t$. The symbol $D$ denotes the standard derivative, while $\mathcal{D}$ denotes the Malliavin derivative operator on the Wiener space $\Omega = C_0([0, T], \mathbb{R}^{k+1})$. We next recall several basic tools from Malliavin calculus. Additional results and details can be found in \cite{CuiHongSheng2026,NualartDavid2006}. Within the framework of Malliavin calculus, $\mathcal{S}$ denotes the class of smooth cylindrical random variables of the form $F=f(B(h_1), \ldots, B(h_n))$ for $f \in C_p^{\infty}(\mathbb{R}^n)$ and $h_i \in \mathbb{H}$, with $B(h)$ denoting the Wiener integral for $h \in \mathbb{H}$. The Malliavin--Sobolev space $\mathbb{D}^{\alpha, p}$ (for $\alpha \geq 1$ and $p \geq 1$) is defined as the closures of $\mathcal{S}$ with respect to the norm 
\begin{align*}
\|F\|_{\alpha, p} = \bigg( \mathbb{E}\Big[ |F|^p + \sum_{j = 1}^{\alpha} \|\mathcal{D}^jF\|_{\mathbb{H}^{\otimes j}}^p \Big] \bigg)^{\frac{1}{p}},
\end{align*}
where $\mathcal{D}^j F$ denotes the $j$th Malliavin derivative of $F$, defined recursively. Additionally, we consider the topological projective limits $L^{\infty-}(\Omega) = \bigcap_{p \geq 1} L^p(\Omega)$, $\mathbb{D}^{\alpha, \infty} = \bigcap_{p \geq 1} \mathbb{D}^{\alpha, p}$, and $\mathbb{D}^{\infty} = \bigcap_{\alpha \geq 1} \bigcap_{p \geq 1} \mathbb{D}^{\alpha, p}$, along with the spaces of distributions $\mathbb{D}^{-\alpha, q}$, which are the dual spaces of $\mathbb{D}^{\alpha, p}$, and generalized Wiener functionals $\mathbb{D}^{-\infty} = \bigcup_{p \geq 1} \bigcup_{\alpha \geq 1} \mathbb{D}^{-\alpha, p}$. Lastly, for a real separable Hilbert space $V$, the notation $\mathbb{D}^{\alpha, p}(V)$ refers to the completion of $V$-valued smooth random variables under the corresponding norm $\|\cdot\|_{\alpha, p, V}$.

\subsection{Properties of GLE \eqref{GLE}}

In this subsection, we collect the fundamental properties of the exact solution to the GLE \eqref{GLE}. We begin by imposing the following mild hypotheses on the potential function $U$, which will be used throughout this paper.

\begin{assumption} \label{ass.1}
Let $m \ge 2$ be an integer. Assume that $U \in C_{p}^{\infty}(\hR)$ and there exist some constants $\epsilon \in(0, 2)$, $C_{i} > 0$, $i = 1, 2, 3, 4, 5$, and $K > 0$ such that
\begin{align}
- C_{3} + C_{1}|x|^{2 m} 
&\leq U(x) \leq C_{2}|x|^{2 m} + C_{3}, 
&& \forall\, x \in \hR, \label{superu} \\
|\nabla^{2} U(x)| 
&\leq C_{2}|x|^{2 m-\epsilon} + C_{3}, 
&& \forall\, x \in \hR, \label{secondu} \\
\langle\nabla U(x), x\rangle 
&\ge C_{4}|x|^{2 m} - C_{5}, 
&& \forall\, x \in \hR, \label{fisrtu} \\
\nabla^2 U(x) 
&\ge - K,
&& \forall\, x \in \hR. \label{nabla2U}
\end{align}
\end{assumption}

Clearly, it follows from \eqref{secondu} that 
\begin{align} 
|\nabla U(x) - \nabla U(y)| 
&\leq C_U (1 + |x|^{2 m-\epsilon} + |y|^{2 m-\epsilon}) |x - y|, 
&& \forall\, x, y \in \hR, \label{deu} \\ 
|\nabla U(x)|
&\leq C(1 + |x|^{2 m + 1 - \epsilon}), 
&& \forall\, x \in \hR. \label{onedu}
\end{align}
Here, $C_U$ and $C$ are positive constants. Next, we state the well-posedness result that guarantees the existence and uniqueness of strong solution to the GLE \eqref{GLE}. For the detailed proof, we refer to \cite[Proposition 6]{GlattNathanHerzogDavidMcKinley2020}.

\begin{proposition}[Well-posedness] \label{prop:GLE:well-posed} 
If Assumption \ref{ass.1} holds, then the GLE \eqref{GLE} admits a unique strong solution $Y = \big(v, z^{\top}, x\big)^{\top} \in L^p(\Omega, \cC ([0, T], \hR^{k + 2}))$ with $p \geq 1$, namely, 
\begin{align} \label{ME}
 \hE \left[ \sup_{0\le t\le T}\|Y(t)\|^p \right] \le C. 
\end{align}
\end{proposition}

By \eqref{superu}, there exists $C > 0$ such that $U(x) + C > 0$ for all $x \in \hR$. In addition to \eqref{ME}, the solution to the GLE \eqref{GLE} is also exponentially integrable, which follows from \cite[Lemma 3.2]{CuiHongSheng2022} with $\mathcal{V} = \widetilde{H}_{0}$ and $\bar{\mathcal{V}} = - K_{0} \big( \gamma + \sum_{\ell = 1}^{k} \alpha_{\ell} \big)$. Here, $\widetilde{H}_{0} = K_{0}(H_{0} + C)$ with some $K_{0} > 0$.

\begin{proposition}[Exponential integrability] \label{expo}
Let $\overline{\alpha}_{\gamma} := \gamma \vee \alpha_{1} \vee \cdots \vee \alpha_{k}$ and $\underline{\alpha}_{\gamma} := \gamma \wedge \alpha_{1} \wedge \cdots \wedge \alpha_{k}$. If Assumption \ref{ass.1} holds, then for any constant $\eta \ge 2 (K_{0}\overline{\alpha}_{\gamma} - \underline{\alpha}_{\gamma}) \vee 0$, 
\begin{align} \label{exactexpo}
\sup_{t \in[0, T]} \hE \left[ \exp \left( \frac{\widetilde{H}_{0}(Y(t))}{e^{\eta t}} \right) \right] 
\leq \exp \left(\frac{K_{0}}{\eta} \Big( \gamma + \sum_{\ell = 1}^{k} \alpha_{\ell} \Big) \right) e^{\widetilde{H}_{0}(Y(0))} . 
\end{align}
\end{proposition}

Using a similar argument as in \cite[Lemma 3.4]{CuiHongSheng2022}, the following two propositions can also be obtained for the exact solution to the GLE \eqref{GLE}.

\begin{proposition}[Malliavin differentiability] \label{Malliavindifferentiability}
If Assumption \ref{ass.1} holds, then
\begin{itemize}
\item[(i)] $Y(t) \in \mathbb{D}^{\infty}(\hR^{k + 2})$. 
 
\item[(ii)] For $p \geq 1$ and $\alpha \in \mathbb{N}_{+}$, 
\begin{align*}
\sup_{r_1, \ldots, r_\alpha \in[0, T]} \hE \left[ \sup_{r_1 \vee \cdots \vee r_\alpha \leq t \leq T} \left\|\mathcal{D}_{r_1, \ldots, r_\alpha}(Y(t))\right\|^p \right] \leq C(p, T, \alpha). 
\end{align*}
\end{itemize}
\end{proposition}

\begin{proposition} [Uniform non-degeneracy]
\label{exactcovarianmatrix}
If Assumption \ref{ass.1} holds, then for any $1 \leq p<\infty$, 
\begin{align*}
\sup _{t \in (0,T]}\left\|\operatorname{det}\left(\Gamma_{t}\right)^{-1}\right\|_{L^p(\Omega)}<\infty ,
\end{align*}
where $\Gamma_t := \int_0^t \mathcal{D}_r Y(t)\left(\mathcal{D}_r Y(t)\right)^{\top} \rd r$ is the Malliavin covariance matrix of $Y(t)$. 
\end{proposition}

In view of \cite[Theorem 2.1.4]{NualartDavid2006}, Propositions \ref{Malliavindifferentiability} and \ref{exactcovarianmatrix} imply that, for any fixed $t \in (0, T]$, the solution $Y(t)$ to the GLE \eqref{GLE} admits an infinitely differentiable probability density function.

If the system admits a unique invariant measure, then it is ergodic, implying that time averages coincide with ensemble averages. Let $Y(t, y)$ denote the solution process at time $t$ starting from the initial state $y$. The ergodicity of the GLE \eqref{GLE} is established by \cite[Theorem 2.8]{DuongNguyen2024} and is stated in the following proposition.

\begin{proposition}[Ergodicity] \label{ergodicity}
If Assumption \ref{ass.1} holds, then the GLE \eqref{GLE} admits a unique invariant measure $\pi$, as defined in \eqref{Gibbs--Boltzmann}. Moreover, for all $g \in L^2(\hR^{k + 2}, \pi)$, 
\begin{align*}
\lim_{t\to\infty} \frac1t \int_0^t \mathbb E[g(Y(r,\cdot))] \,\rd r 
 = \int_{\hR^{k+2}} g \,\rd \pi \qquad \mbox{in}~~ L^2(\hR^{k + 2}, \pi).
\end{align*} 
\end{proposition}

\subsection{The splitting AVF method}

In this subsection, we present a new splitting AVF method tailored for the GLE \eqref{GLE} and analyze its solvability. Inspired by the approach in \cite{ChenDangHongZhang2025}, we first reformulate the GLE \eqref{GLE} as 
\begin{align*}
\rd 
\begin{pmatrix}
v(t) \\
z_1(t) \\
\vdots \\
z_k(t) \\
x(t)
\end{pmatrix}
= 
\underbrace{
\begin{pmatrix}
- \frac{\gamma}{2} v(t) - \nabla U(x(t)) + \sum_{\ell = 1}^k \lambda_{\ell} z_{\ell}(t) \\
- \lambda_1 v(t) - \frac{\gamma}{2} \lambda_1 x(t) \\
\vdots \\
- \lambda_k v(t) - \frac{\gamma}{2} \lambda_k x(t) \\
\frac{\gamma}{2} x(t) + v(t)
\end{pmatrix}
}_{ = :\, \Psi^D} \rd t 
+ 
\underbrace{
\begin{pmatrix}
- \frac{\gamma}{2} v(t) \rd t + \sqrt{2\gamma} \rd W_0(t) \\
( - \alpha_1 z_1(t) + \frac{\gamma}{2} \lambda_1 x(t)) \rd t + \sqrt{2\alpha_1} \rd W_1(t) \\
\vdots \\
( - \alpha_k z_k(t) + \frac{\gamma}{2} \lambda_k x(t)) \rd t + \sqrt{2\alpha_k} \rd W_k(t) \\
- \frac{\gamma}{2} x(t) \rd t 
\end{pmatrix}
}_{ = :\, \Psi^S}. 
\end{align*} 
Here, the deterministic part $\Psi^{D}$ possesses a new Hamiltonian 
\begin{align} \label{newHamil}
H(v, z, x) := H_0(v, z, x) + \frac{\gamma}{2}vx, \qquad v \in \hR,~~z \in \hR^{k},~~x \in \hR,
\end{align}
where $H_0$ is defined in \eqref{Hamiltonian0}, and the stochastic part $\Psi^S$ is linear and dissipative.

Let $h\in(0, 1)$ denote the uniform step size, and set $t_n := n h$ for $n = 0, 1, 2, \cdots, T/h$. Furthermore, denote by $\lfloor t\rfloor :=\sup _{n \leq T/h}\left\{t_n: t_n \leq t\right\}$. We advance the numerical method step by step over the time intervals $T_n := [t_n, t_{n + 1} )$. The numerical solution $Y_{n} := (v_{n}, z_{n}^{\top}, x_{n})^{\top}$ 
is defined recursively by $Y_{0} = (v(0), z(0)^{\top}, x(0))$ and 
\begin{align} \label{numsolu}
Y_{n + 1} = \Phi_{T_{n}, t_{n + 1}}^{S} \Phi_{T_{n}, t_{n + 1}}^{D}(Y_{n}), \qquad n \in \mathbb{N}, 
\end{align}
where $\Phi_{T_{n}, t_{n + 1}}^{D}$ and $\Phi_{T_{n}, t_{n + 1}}^{S}$ denote the one-step approximation of the deterministic subsystem and the flow of the stochastic subsystem, respectively. More precisely, for the deterministic subsystem, we apply the AVF method, denoted by $\bar{Y}_{T_n, t_{n+1}} = (\bar{v}_{T_n, t_{n+1}}, \bar{z}_{T_n, t_{n+1}}^\top, \bar{x}_{T_n, t_{n+1}})^\top =: \Phi_{T_{n}, t_{n + 1}}^{D}(Y_{n})$, to preserve its Hamiltonian, in which 
\begin{align} \label{avfmethod}
\bar{v}_{T_{n}, t_{n + 1}} = v_{n} + \mathcal{A}^{n}, \qquad
\bar{z}^{\ell}_{T_{n}, t_{n + 1}} = z^{\ell}_{n} + \mathcal{B}_{\ell}^{n}, \qquad
\bar{x}_{T_{n}, t_{n + 1}} = x_{n} + \mathcal{C}^{n} 
\end{align} 
with $\ell = 1, 2, \cdots, k$ and 
\begin{align*}
&\mathcal{A}^{n} 
:= - \frac{\gamma}{4}h(\bar{v}_{T_n, t_{n + 1}} + v_{n}) - h \int_{0}^{1}\nabla U(x_{n} + \theta(\bar{x}_{T_n, t_{n + 1}} - x_{n}))\rd \theta + \frac{h}{2} \sum\limits^{k}_{\ell = 1}\lambda_{\ell}(\bar{z}_{T_{n}, t_{n + 1}}^{\ell} + z_{n}^{\ell}), \\
&\mathcal{B}_{\ell}^{n}
:= - \frac{\lambda_{\ell}}{2}h(\bar{v}_{T_n, t_{n + 1}} + v_{n}) - \frac{\gamma\lambda_{\ell}}{4}h(\bar{x}_{T_n, t_{n + 1}} + x_{n}), \\
&\mathcal{C}^{n}
:= \frac{\gamma}{4}h(\bar{x}_{T_n, t_{n + 1}} + x_{n}) + \frac{h}{2}(\bar{v}_{T_n, t_{n + 1}} + v_{n}). 
\end{align*} 
For the stochastic subsystem, using Duhamel's formula yields $\tilde{Y}_{T_{n}, t} = (\tilde{v}_{T_{n}, t}, \tilde{z}_{T_{n}, t}^{\top}, \tilde{x}_{T_{n}, t})^\top$ with 
\begin{align} \label{sdesolu}
\begin{cases}
\tilde{v}_{T_n, t} 
= e^{-\frac{\gamma}{2} (t - t_n)} \bar{v}_{T_{n}, t_{n + 1}} + \int_{t_{n}}^{t} e^{-\frac{\gamma}{2}(t - s)} \sqrt{2 \gamma} \rd W_{0}(s), \\
\tilde{z}_{T_{n}, t}^{\ell} 
= e^{-\alpha_{\ell} (t - t_n)} \bar{z}_{T_{n}, t_{n + 1}}^{\ell} + \frac{\gamma \lambda_{\ell} (e^{-\frac{\gamma}{2} (t - t_n)} - e^{-\alpha_{\ell} (t - t_n)})}{2 \alpha_{\ell} - \gamma} \bar{x}_{T_{n}, t_{n + 1}} + \int_{t_{n}}^{t} e^{-\alpha_{\ell}(t - s)} \sqrt{2 \alpha_{\ell}} \rd W_{\ell}(s), \\
\tilde{x}_{T_{n}, t} 
= e^{-\frac{\gamma}{2} (t - t_n)} \bar{x}_{T_{n}, t_{n + 1}}. 
\end{cases}
\end{align}
Consequently, we arrive at $Y_{n+1} = (\tilde{v}_{T_{n}, t_{n+1}}, \tilde{z}_{T_{n}, t_{n+1}}^{\top}, \tilde{x}_{T_{n}, t_{n+1}})^\top = :\Phi_{T_{n}, t_{n+1}}^{S}(\bar{Y}_{T_n, t_{n+1}})$. To facilitate the theoretical analysis, we also introduce a time-continuous version for $t \in T_n$, given by 
\begin{align} \label{continuousversion}
\begin{cases}
\tilde{v}_{t} 
 = v_n + \int_{t_n}^t \big( h^{-1} \mathcal{A}^n - \frac{\gamma}{2} \tilde{v}_{s} \big) \rd s + \int_{t_{n}}^{t} \sqrt{2 \gamma} \rd W_{0}(s), \\
\tilde{z}_{t}^{\ell} 
 = z^\ell_n + \int_{t_n}^t \big( h^{-1} \mathcal{B}_\ell^n - \alpha_\ell \tilde{z}_{s}^\ell + \frac{\gamma}{2} \lambda_\ell \tilde{x}_{s} \big) \rd s + \int_{t_{n}}^{t} \sqrt{2 \alpha_{\ell}} \rd W_{\ell}(s), \quad \ell=1,2,\cdots,k, \\
\tilde{x}_{t} 
 = x_n + \int_{t_n}^t \big( h^{-1} \mathcal{C}^n - \frac{\gamma}{2} \tilde{x}_{s} \big) \rd s. 
\end{cases}
\end{align}

Next, we consider the solvability of the splitting AVF method \eqref{numsolu}--\eqref{sdesolu}. Clearly, only \eqref{avfmethod} is implicit, and we denote 
\begin{align} \label{implicitavf}
G(\bar{Y}_{T_{n}, t_{n + 1}}, Y_{n}, h) := \left(\begin{array}{c}
\bar{v}_{T_n, t_{n + 1}} - v_{n} - \mathcal{A}^{n} \\
\bar{z}_{T_{n}, t_{n + 1}}^{1} - z_n^1 -\mathcal{B}_1^n \\
\vdots\\
\bar{z}_{T_{n}, t_{n + 1}}^{k} - z_n^k -\mathcal{B}_k^n \\
\bar{x}_{T_{n}, t_{n + 1}} - x_n - \mathcal{C}^n \end{array}\right)
 = \mathbf{0} \in \hR^{k+2}
\end{align}
and
\begin{align} \label{eq.def:h_star}
h^{*} := \min \left\{\frac{4}{K + 1}, \, \min_{1 \le \ell \le k} \left\{\frac{8}{\gamma \lambda_{\ell}^{2}}\right\}, \, \frac{8}{2 \gamma + 2(K + 1) + \gamma k}\right\}. 
\end{align}

\begin{proposition} [Existence and uniqueness] \label{solvability} 
The splitting AVF method \eqref{numsolu}--\eqref{sdesolu} is uniquely solvable for all $h \in (0 , h^*)$.
\end{proposition}

\begin{proof}
We define $\Delta y = (\Delta v, \Delta z^1, \cdots, \Delta z^k ,\Delta x)^\top = y_1 - y_2$, for $y_1,y_2 \in \hR^{k+2}$. By \eqref{nabla2U} and Young's inequality, 
\begin{align*}
&\quad\ \langle \Delta y, G(y_1, Y_{n}, h) - G(y_2, Y_{n}, h)\rangle \\
&\ge(1 + \frac{\gamma}{4} h)|\Delta v|^{2} + h \int_{0}^{1} \int_{0}^{1} - K \xi|\Delta v||\Delta x| \rd \theta \rd \xi + \sum_{\ell = 1}^{k}|\Delta z^{\ell}|^{2} - \frac{\gamma h}{4} \sum_{\ell = 1}^{k} \lambda_{\ell} |\Delta z^{\ell}||\Delta x| \\
&\quad + (1 - \frac{\gamma}{4} h)|\Delta x|^{2} - \frac{h}{2}|\Delta x||\Delta v| \\
&\ge(1 + \frac{\gamma}{4} h - \frac{K + 1}{4} h)|\Delta v|^{2} + \sum_{\ell = 1}^{k}(1 - \frac{\gamma h}{8} \lambda_{\ell}^{2})|\Delta z^{\ell}|^{2} + (1 - \frac{\gamma}{4} h - \frac{K + 1}{4} h - \frac{\gamma k}{8} h)|\Delta x|^{2}, 
\end{align*}
which indicates that $G(y, Y_{n}, h)$ is monotone with respect to $y$ for $h \in (0 , h^*)$. In particular, by taking $y_2 = 0$, it is readily verify that the mapping $y \mapsto G(y, Y_{n}, h)$ is coercive in the sense that
\begin{align*}
\lim_{\|y\| \to \infty} \frac{\langle y, G(y, Y_{n}, h)\rangle}{\|y\|} = \infty. 
\end{align*}
By the monotonicity and coercivity of the mapping $y \mapsto G(y, Y_{n}, h)$, \cite[Lemma 3.1]{MaoSzpruchLukasz2013} ensures that, for any $h \in (0, h^{*})$, there exists a unique $y \in \hR^{k+2}$ such that $G(y, Y_{n}, h) = 0$. Thus, \eqref{avfmethod} is uniquely solvable, i.e., $G(y, Y_{n}, h) = 0$ determines a function $y = \phi_h(Y_n)$ such that for $h \in (0, h^{*})$,
\begin{align} \label{demapphi}
\bar{Y}_{T_{n}, t_{n + 1}} = \phi_h(Y_{n})=(\phi_{v}(Y_{n}),\phi_{z^1}(Y_{n}),\ldots,\phi_{z^k}(Y_{n}),\phi_{x}(Y_{n}))^\top.
\end{align}
The proof is completed. 
\end{proof}

\begin{remark}
\label{phi:bijection_smooth}
From the above proof, combined with \cite[Lemma 3.1]{MaoSzpruchLukasz2013}, we can conclude that $\phi_h$ is a bijection. Furthermore, by the implicit function theorem, the assumption $U \in C_{p}^{\infty}(\hR)$ implies that $\phi_h$ is smooth.
\end{remark}

\section{First-order strong convergence}
\label{sec.StrongConvergence}

In this section, we provide the exponential integrability of the splitting AVF method \eqref{numsolu}--\eqref{sdesolu} and, based on this property, establish its optimal strong convergence rate.

\subsection{Exponential integrability of the numerical approximation}
In this subsection, we show the exponential integrability of the splitting AVF method \eqref{numsolu}--\eqref{sdesolu}. Exponential integrability plays a key role in the optimal strong convergence order analysis of numerical methods under non-globally monotone conditions \cite{ChenDangHongZhang2025, CuiHongSheng2022, daiJiangWang2025, DaiWang2025, HutzenthalerJentzenWang2018}.

\begin{proposition}[Exponential integrability] \label{Numerical:expo} 
Assume that Assumption \ref{ass.1} holds, and let $C_{0} \ge C_{3} + 1$ be fixed, where $C_{3}$ is determined by \eqref{superu}. Define $\widetilde{H} = K_{0}(H + C_{0})$ with $K_{0} > 0$ and $H$ is given by \eqref{newHamil}. Then there exists $C > 0$ such that 
\begin{align*}
\sup_{n \leq T/h} \hE \left[ \exp \left(\frac{\widetilde{H}(Y_{n})}{e^{\eta t_{n}}} \right) \right] \leq e^{C + \widetilde{H}(Y(0))}, 
\end{align*}
where $\eta \ge 2 (K_{0}\overline{\alpha}_{\gamma} - \frac{\underline{\alpha}_\gamma}{2}) \vee 0$ with $\overline{\alpha}_{\gamma} := \gamma \vee \alpha_{1} \vee \cdots \vee \alpha_{k}$ and $\underline{\alpha}_{\gamma} := \gamma \wedge \frac{\gamma C_4}{4 C_2} \wedge \alpha_{1} \wedge \cdots \wedge \alpha_{k}$.
\end{proposition}

\begin{proof}
For $Y = (v, z^{\top}, x)^{\top}$, we denote $\tilde{\mu}(Y) 
= ( - \frac{\gamma}{2} v, \, - \alpha_{1} z_{1} + \frac{\gamma}{2} \lambda_{1} x, \, \cdots, \, - \alpha_{k} z_{k} + \frac{\gamma}{2} \lambda_{k} x, \, - \frac{\gamma}{2} x )^{\top}$ and $\tilde{\sigma} 
= \mbox{diag}( \sqrt{2 \gamma}, \, \sqrt{2 \alpha_{1}}, \, \cdots, \, \sqrt{2 \alpha_{k}}, \, 0 )$. Using \eqref{superu}, \eqref{fisrtu} and Young's inequality yields 
\begin{align*}
\quad\ D \widetilde{H}(Y) \tilde{\mu}(Y) 
&= K_{0}\left( - \frac{\gamma}{2} v^{2} - \frac{\gamma}{2} x \nabla U(x) - \frac{\gamma^{2}}{2} x v - \sum_{\ell = 1}^{k} \alpha_{\ell} z_{\ell}^{2} + \frac{\gamma}{2}\sum_{\ell = 1}^{k} \lambda_{\ell} z_{\ell} x \right) \\
&\leq K_{0}\left( - \frac{\gamma}{2} v^{2} - (\frac{\gamma}{2} C_{4} - \frac{\varepsilon}{2}) x^{2 m} - \frac{\gamma^{2}}{2} x v - \sum_{\ell = 1}^{k}(\alpha_{\ell} - \frac{\varepsilon_{\ell}}{2}) z_{\ell}^{2} + C \right) \\
&\leq K_{0}\left( - \frac{\gamma}{2} v^{2} - (\frac{\gamma}{2} C_{4} - \frac{\varepsilon}{2}) U(x)/C_2 - \frac{\gamma^{2}}{2} x v - \sum_{\ell = 1}^{k}(\alpha_{\ell} - \frac{\varepsilon_{\ell}}{2}) z_{\ell}^{2} + C \right) \\
&\leq K_0\left(-\frac{\gamma^2}{2}xv - \underline{\alpha}_\gamma H_0(Y) + C\right) \\
&\leq K_0\left(\frac{\gamma}{2}\left(\frac{\underline{\alpha}_\gamma}{2} + \gamma\right)\delta v^2 + \frac{\gamma}{2}\left(\frac{\underline{\alpha}_\gamma}{2} + \gamma\right)\frac{1}{4\delta }x^2 - \frac{\underline{\alpha}_\gamma}{2}H_0(Y) - \frac{\underline{\alpha}_\gamma}{2}H(Y) + C\right) \\
&\leq K_0\left(\frac{\gamma}{2}\left(\frac{\underline{\alpha}_\gamma}{2} + \gamma\right)\delta v^2 + \tilde{\delta }U(x) + C_{\tilde{\delta }} - \frac{\underline{\alpha}_\gamma}{2}H_0(Y) - \frac{\underline{\alpha}_\gamma}{2}H(Y) + C\right) \\
&\leq -\frac{\underline{\alpha}_\gamma}{2}\widetilde{H}(Y) + C, 
\end{align*}
where $\varepsilon_{\ell} := \alpha_{\ell}$, $\varepsilon := \frac{\gamma C_4}{2} $, $ \delta \leq \alpha_\gamma / (2 \gamma (\frac{\underline{\alpha}_\gamma}{2} + \gamma))$, and $\tilde{\delta } \leq \frac{\underline{\alpha}_\gamma}{2}$. On the other hand, we have $\operatorname{tr}(D^{2} \widetilde{H}(Y) \tilde{\sigma} \tilde{\sigma}^{\top}) = 2 K_0 \big(\gamma + \sum_{\ell = 1}^{k} \alpha_{\ell}\big)$ and 
\begin{align*}
\|D\widetilde{H}(Y)\tilde{\sigma}\|^{2} 
&= K_{0}^{2} \left| \sqrt{2 \gamma}(v + \frac{\gamma}{2} x) \right|^2 + K_{0}^{2} \sum_{\ell = 1}^{k} | \sqrt{2 \alpha_{\ell}} z_{\ell} |^{2} \\
&\leq K_{0}^{2} \left( 2 \gamma \Big( v^{2} + \gamma xv + 2 C_{1} x^{2 m} + \frac{m-1}{m \left(2 m C_1\right)^{\frac{1}{m-1}}}\left(\frac{\gamma}{4}\right)^{\frac{m}{m-1}} \Big) + 2 \sum_{\ell = 1}^{k} \alpha_{1} z_{\ell}^{2} \right) \\
&\leq K_{0}^{2} \left( 2 \gamma(v^{2} + \gamma xv + 2 U(x) + C) + 2 \sum_{\ell = 1}^{k} \alpha_{\ell} z_{\ell}^{2} \right) \\
&\leq 2 K_{0}\overline{\alpha}_{\gamma} \widetilde{H}(Y) + C. 
\end{align*}
Then, one can arrive at 
\begin{align*}
D \widetilde{H}(\tilde{Y}_{T_n, t}) \tilde{\mu}(\tilde{Y}_{T_n, t}) + \frac{\operatorname{tr}(D^{2} \widetilde{H}(\tilde{Y}_{T_n, t}) \tilde{\sigma} \tilde{\sigma}^{\top})}{2} + \frac{\|D \widetilde{H}(\tilde{Y}_{T_n, t})\tilde{\sigma} \|^{2}}{2 e^{\eta t}} 
\leq \widetilde{H}(\tilde{Y}_{T_n, t})(2 (K_{0} \overline{\alpha}_{\gamma} - \frac{\underline{\alpha}_{\gamma}}{2} )) + C, 
\end{align*}
which together with \cite[Lemma 3.2]{CuiHongSheng2022} yields
\begin{align*}
& \mathbb{E}\left[\exp \left(\frac{\widetilde{H}(\tilde{Y}_{T_n, t_{n+1}})}{e^{\eta t_{n+1}}}\right)\right] \leq \mathbb{E}\left[\exp \left(\frac{\widetilde{H}(\tilde{Y}_{T_n, t_n})}{e^{\eta t_n}}\right)\right] \exp \left[C\left(e^{-\eta t_n}-e^{-\eta t_{n+1}}\right)\right] .
\end{align*}
In view of $\widetilde{H}(\tilde{Y}_{T_n, t_n})=\widetilde{H}(\bar{Y}_{T_n, t_{n+1}})=\widetilde{H}(Y_n)$ and $\widetilde{H}(\tilde{Y}_{T_n, t_{n+1}})=\widetilde{H}(Y_{n+1})$, it follows that
\begin{align*}
\mathbb{E}\left[\exp \left(\frac{\widetilde{H}(Y_{n+1})}{e^{\eta t_{n+1}}}\right)\right] \leq \mathbb{E}\left[\exp \left(\frac{\widetilde{H}(Y_n)}{e^{\eta t_n}}\right)\right] \exp \left[C\left(e^{-\eta t_n}-e^{-\eta t_{n+1}}\right)\right] .
\end{align*}
Accordingly,
\begin{align*}
\sup _{n \leq T/h} \mathbb{E}\left[\exp \left(\frac{\widetilde{H}(Y_n)}{e^{\eta t_n}}\right)\right] &\leq \prod_{i=0}^{T/h-1} \exp \left[C\left(e^{-\eta t_i}-e^{-\eta t_{i+1}}\right)\right] e^{\widetilde{H}(Y(0))} \\
&\leq e^{C + \widetilde{H}(Y(0))},
\end{align*}
which completes the proof.
\end{proof}

In view of Proposition \ref{Numerical:expo}, the following moment boundedness result for the numerical solutions $\{\bar{Y}_{T_{n}, t_{n + 1}}\}$ and $\{Y_{n}\}$ follows from an application of It\^o's formula together with the Burkholder--Davis--Gundy (BDG) inequality.

\begin{corollary} \label{numerical:momentbound}
If Assumption \ref{ass.1} holds, then there exists a positive constant $C$, independent of $h$, such that for any $p \ge 1$,
\begin{align*}
 \hE \left[ \sup_{n \leq T/h} \left| \widetilde{H}(\bar{Y}_{T_{n}, t_{n + 1}}) \right|^{p} \right] 
 + \hE \left[ \sup_{n \leq T/h} \left| \widetilde{H}(Y_{n}) \right|^{p} \right] 
\leq C. 
\end{align*}
\end{corollary}

\subsection{Strong error analysis}

In this subsection, we establish the first-order strong convergence for the splitting AVF method \eqref{numsolu}--\eqref{sdesolu} applied to the GLE \eqref{GLE}. The optimal strong error analysis of numerical methods for the classical Langevin equation with superquadratically growing potentials is typically carried out in two steps; see, e.g., \cite{CuiHongSheng2022,daiJiangWang2025}. In the first step, a preliminary convergence rate of order $\frac{1}{2}$ is obtained by exploiting the exponential integrability of both the exact and numerical solutions. In the second step, this preliminary rate is refined to attain the desired result. Here, we adopt the same two-step strategy to establish the optimal first-order strong convergence.

\begin{lemma} \label{1/2strongconver}
Let $h^{*}$ be defined by \eqref{eq.def:h_star}. If Assumption \ref{ass.1} holds, then there exists $C := C(T)>0$ and $\tilde{h} \in(0, h^{*})$ such that for any $h \in(0, \tilde{h})$ and $p\ge1$, 
\begin{align*}
\sup_{n \leq T/h}\|Y(t_{n}) - Y_{n}\|_{L^{2p}(\Omega, \hR^{k+2})} \leq C h^{1/2}. 
\end{align*}
\end{lemma}

\begin{proof}
Denote $e_v(t_{n}) := v_n - v(t_{n})$, $e_{z^{\ell}}(t_{n}) := z_{n}^{\ell} - z_{\ell}(t_{n}) \, (\ell = 1, \cdots, k)$, and $e_x(t_{n}) := x_n - x(t_{n})$. It follows from \eqref{GLE} that for any $0 \leq s<t \leq T$, 
\begin{align*} 
\begin{cases}
v(t) = e^{-\frac{\gamma}{2}(t - s)} v(s) - \int_{s}^{t} e^{-\frac{\gamma}{2}(t - u)}(\frac{\gamma}{2} v(u) + \nabla U(x(u)) - \sum_{\ell = 1}^{k} \lambda_{\ell} z_{\ell}(u)) \rd u + \int_{s}^{t} e^{-\frac{\gamma}{2}(t - u)} \sqrt{2 \gamma} \rd W_{0}(u), \\
z_{\ell}(t) = e^{-\alpha_{\ell} (t - s)} z_{\ell}(s) - \int_{s}^{t} e^{-\alpha_{\ell} (t - u)} \lambda_{\ell} v(u) \rd u + \int_{s}^{t} e^{-\alpha_{\ell}(t - u)} \sqrt{2 \alpha_{\ell}} \rd W_{\ell}(u), \qquad \ell = 1, 2, \cdots, k, \\
x(t) = e^{-\frac{\gamma}{2}(t - s)} x(s) + \int_{s}^{t} e^{-\frac{\gamma}{2}(t - u)}(\frac{\gamma}{2} x(u) + v(u)) \rd u, 
\end{cases}
\end{align*}
which together with \eqref{avfmethod} and \eqref{sdesolu} indicates 
\begin{align}
e_v(t_{n + 1}) 
&= e^{-\frac{\gamma}{2} h}e_v(t_{n}) + \int_{t_n}^{t_{n + 1}}( - e^{\frac{\gamma}{2} h} + e^{-\frac{\gamma}{2}(t_{n + 1} - t)}) \big(\frac{\gamma}{2} v(t) + \nabla U(x(t)) - \sum_{\ell = 1}^{k} \lambda_{\ell} z_{\ell}(t)\big) \rd t \notag \\
&\quad + e^{-\frac{\gamma}{2} h} \int_{t_{n}}^{t_{n + 1}} \big( \frac{\gamma}{2} v(t) - \frac{\gamma}{4}(\bar{v}_{T_n, t_{n+1}} + v_n) - \sum_{\ell = 1}^{k} \lambda_{\ell} (z_{\ell}(t) - \frac{1}{2}(\bar{z}_{T_n, t_{n+1}}^{\ell} + z_n^\ell)) \notag \\
&\quad + \nabla U(x(t)) - \int_0^1 \nabla U (x_n + \xi (\bar{x}_{T_n, t_{n+1}} - x_n)) \rd \xi \big) \rd t, \label{vvgle} \\
e_{z_{\ell}}(t_{n + 1}) 
&= e^{-\alpha_{\ell} h}e_{z_{\ell}}(t_n) + \int_{t_n}^{t_{n + 1}}( - e^{-\alpha_{\ell} h} + e^{-\alpha_{\ell}(t_{n + 1} - t)})\big(\lambda_{\ell}v (t) + \frac{\gamma}{2} \lambda_{\ell} x(t) \big) \rd t \notag \\
&\quad + e^{-\alpha_{\ell} h} \int_{t_n}^{t_{n + 1}}\big(\lambda_{\ell} v(t) - \frac{\lambda_{\ell}}{2}(\bar{v}_{T_n, t_{n+1}} + v_n) + \frac{\gamma \lambda_{\ell}}{2} x(t) - \frac{\gamma \lambda_{\ell}}{4}(\bar{x}_{T_n, t_{n+1}} + x_n)\big) \rd t \notag \\
&\quad + \int_{t_n}^{t_{n + 1}} e^{-\alpha_{\ell}(t_{n + 1} - t)} \frac{\gamma}{2} \lambda_{\ell}(\tilde{x}_{T_n, t} - x(t)) \rd t, \label{zzgle} \\
e_x(t_{n + 1}) 
&= e^{-\frac{\gamma}{2} h}e_x(t_{n}) - \int_{t_n}^{t_{n + 1}}( - e^{-\frac{\gamma}{2} h} + e^{-\frac{\gamma}{2}(t_{n + 1} - t)})\big(\frac{\gamma}{2} x(t) + v(t)\big) \rd t \notag \\
&\quad - e^{-\frac{\gamma}{2} h} \int_{t_n}^{t_{n + 1}}\big(\frac{\gamma}{2} x(t) - \frac{\gamma}{4}(\bar{x}_{T_n, t_{n+1}} + x_n) + v(t) - \frac{1}{2}(\bar{v}_{T_n, t_{n+1}} + v_n)\big) \rd t. \label{xxgle}
\end{align}
Using \eqref{GLE}, \eqref{avfmethod} and \eqref{continuousversion} shows 
\begin{align}
&v(t)-\frac{1}{2}(\bar{v}_{T_{n}, t_{n+1}}+v_{n}) = e_{v}(t_{n})- \int_{t_{n}}^{t}(\gamma v(s)+\nabla U(x(s))- \sum_{\ell = 1}^{k} \lambda_{\ell} z_{\ell}(s)) \rd s -\frac{1}{2} \mathcal{A}^{n}+ \int_{t_{n}}^{t} \sqrt{2 \gamma} \rd W_{0}(s), \label{vbarv} \\
&z_{\ell}(t)-\frac{1}{2}(\bar{z}_{T_{n}, t_{n+1}}^\ell+z_{n}^\ell) = e_{z_{\ell}}(t_{n})- \int_{t_{n}}^{t}(\alpha_{\ell} z_{\ell}(s)+\lambda_{\ell} v(s)) \rd s-\frac{1}{2} \mathcal{B}_{\ell}^{n} + \int_{t_{n}}^{t} \sqrt{2 \alpha_{\ell}} \rd W_{\ell}(s), \label{zbarz} \\
& x(t)-\frac{1}{2}(\bar{x}_{T_{n}, t_{n+1}}+x_{n}) = e_{x}(t_{n})- \int_{t_{n}}^{t} v(s) \rd s-\frac{1}{2} \mathcal{C}^{n}, \label{xbarx} \\
& \tilde{x}_{T_{n}, t}-x(t) = x_{n}-x(t_{n})- \int_{t_{n}}^{t}(\frac{\gamma}{2} x(s)+v(s)) \rd s+ \mathcal{C}^{n}+ \int_{t_{n}}^{t} \frac{\gamma}{2}(x(s) -\tilde{x}_{T_{n}, s}) \rd s. \label{xtidlex}
\end{align}
Applying Gr\"onwall's inequality implies 
\begin{align} \label{du_du}
|\tilde{x}_{T_{n}, t}-x(t)| \leq C(|e_{x}(t_{n})|+ \int_{t_{n}}^{t}|\frac{\gamma}{2} x(s)+v(s)| \rd s+|\mathcal{C}^{n}|) .
\end{align}
By the mean value theorem, one gets 
\begin{align*}
&\ \nabla U(x(t))- \int_{0}^{1} \nabla U(x_{n}+\xi(\bar{x}_{T_{n}, t_{n+1}}-x_{n})) \rd \xi \\
&= \int_{0}^{1} \int_{0}^{1} \nabla^{2} U(\theta x(t)+(1-\theta)(x_{n}+\xi(\bar{x}_{T_{n}, t_{n+1}}-x_{n}))(x(t)-x_{n}-\xi(\bar{x}_{T_{n}, t_{n+1}}-x_{n})) \rd \theta \rd \xi \\
&= \int_{0}^{1} \int_{0}^{1} \nabla^{2} U(\theta x(t)+(1-\theta)(x_{n}+\xi \bar{x}_{T_{n}, t_{n+1}}-x_{n}))(x(t_{n})-x_{n}) \rd \theta \rd \xi \\
&\quad + \int_{0}^{1} \int_{0}^{1} \nabla^{2} U(\theta x(t)+(1-\theta)(x_{n}+\xi(\bar{x}_{T_{n}, t_{n+1}}-x_{n}))( \int_{t_{n}}^{t} v(s) \rd s-\xi \mathcal{C}^{n}) \rd \theta \rd \xi. 
\end{align*}
By combining \eqref{secondu}, \eqref{deu}, \eqref{vvgle}--\eqref{du_du}, Corollary \ref{numerical:momentbound}, and Young's inequality, one can obtain 
\begin{align*}
|e_{v}(t_{n+1})| 
&\leq |e_{v}(t_{n})| + C|e_{x}(t_{n})| \int_{t_{n}}^{t_{n+1}}(1+|x(t)|^{2 m-\epsilon}+|x_{n}|^{2 m-\epsilon}+ |\bar{x}_{T_{n}, t_{n+1}} |^{2m-\epsilon}) \rd t \\
&\quad + C h(|e_{v}(t_{n})|+|e_{z_{\ell}}(t_{n})|) +C h^{2} \Upsilon_{n}+e^{-\frac{\gamma}{2} h} \frac{\gamma}{2} \Big| \int_{t_{n}}^{t_{n+1}} \int_{t_{n}}^{t} \sqrt{2 \gamma} \rd W_{0}(s) \rd t\Big| \\
&\quad + e^{-\frac{\gamma}{2} h} \sum_{\ell = 1}^{n} \lambda_{\ell} \Big| \int_{t_{n}}^{t_{n+1}} \int_{t_{n}}^{t} \sqrt{2 \alpha_{\ell}} \rd W_{\ell}(s) \rd t\Big|, \\
|e_{z_{\ell}}(t_{n+1})| 
&\leq |e_{z_{\ell}}(t_{n})|+C h(|e_{v}(t_{n})|+|e_{x}(t_{n})|)+C h^{2} \Upsilon_{n} + e^{-\alpha_{\ell} h} \lambda_{\ell} \Big| \int_{t_{n}}^{t_{n+1}} \int_{t_{n}}^{t} \sqrt{2 \gamma} \rd W_{0}(s) \rd t\Big|, \\
|e_{x}(t_{n+1})|
&\leq |e_{x}(t_{n})|+C h(|e_{v}(t_{n})|+|e_{x}(t_{n})|)+C h^{2} \Upsilon_{n} + e^{-\frac{\gamma}{2} h} \Big| \int_{t_{n}}^{t_{n+1}} \int_{t_{n}}^{t} \sqrt{2 \gamma} \rd W_{0}(s) \rd t\Big|, 
\end{align*}
where $\Upsilon_{n} = 1+\sup_{t \in[0, T]}|v(t)|^{2}+|\bar{v}_{T_{n}, t_{n}+1}|^{2}+|v_{n}|^{2}+ \sum_{\ell = 1}^{k}(\sup_{t \in[0, T]}| z_{\ell}(t)|+|\bar{z}_{T_{n}, t_{n+1}}^{\ell}|+|z_{n}^{\ell}|) + \sup_{t \in[0, T]}|x(t)|^{2 m} + |\bar{x}_{T_{n}, t_{n+1}}|^{2 m}+|x_{n}|^{2 m}$. Denote $\mathcal{E}(t_{n}):= |e_{v}(t_{n})|+ \sum_{\ell = 1}^{n}|e_{z_{\ell}}(t_{n})|+|e_{x}(t_{n})|$. It holds that 
\begin{align*}
\mathcal{E}(t_{n+1}) \leq \mathcal{E}(t_{n})+C \mathcal{E}(t_{n}) \int_{t_{n}}^{t_{n+1}}\big( 1+|x(t)|^{2 m-\epsilon}+|x_{n}|^{2 m-\epsilon}+|\bar{x}_{T_{n}, t_{n+1}}|^{2 m-\epsilon} \big) \rd t +C h^{2} \Upsilon_{n}+M_{n}, 
\end{align*}
where $M_{n}
= \sqrt{2 \gamma}(e^{-\frac{\gamma}{2} h} \frac{\gamma}{2}+ \sum_{\ell = 1}^{k} e^{-\alpha_{\ell} h} \lambda_{\ell}+e^{-\frac{\gamma}{2} h}) | \int_{t_{n}}^{t_{n+1}}(t_{n}-s) \rd W_{0}(s)| + e^{-\frac{\gamma}{2} h} \sum_{\ell = 1}^{k} \lambda_{\ell}\sqrt{2 \lambda_{\ell}} | \int_{t_{n}}^{t_{n+1}}(t_{n}-s) \rd W_{\ell}(s)|$. Moreover, by H\"older's inequality, 
\begin{align*}
\||M_{j}|^{p}\|_{L^{2}(\Omega)}^{2} 
\leq C \sum_{\ell = 0}^{k} \hE\bigg[ \Big| \int_{t_{j}}^{t_{j+1}}(t_{j+1}-s) \rd W_{\ell}(s) \Big|^{2 p}\bigg] 
\leq C h^{3 p}, \qquad \forall\, j = 0, 1, \cdots, T/h-1.
\end{align*}
Then, applying the discrete Gr\"onwall inequality and $\mathcal{E}(t_{0}) = 0$ indicates 
\begin{align*}
& \mathcal{E}(t_{n+1}) \leq \Big( \sum_{j = 0}^{n} ( C h^{2} \Upsilon_{j}+M_{j})\Big) \exp \bigg( \sum_{j = 0} ^ {n} C \int_{t_{j}}^{t_{j+1}} \big( 1+|x(t)|^{2 m-\epsilon} + |x_{n}|^{2 m-\epsilon}+|\bar{x}_{T_{n}, t_{n+1}}|^{2 m-\epsilon} \big) \rd t \bigg).
\end{align*}
By H\"older's inequality, one has that for any $n = 0, 1, \cdots, T/h-1$, 
\begin{align*}
\|\mathcal{E}^{p} (t_{n+1}) \|_{L^{2}(\Omega)}
&\leq \Big( \sum_{j = 0}^{n} \big( C h^{2 p}\|\Upsilon_{j}^{p}\|_{L^{2}(\Omega)}+C\| | M_{j}|^{p}\|_{L^{2}(\Omega)} \big) \Big) \\
& \qquad \times \Big\|\exp \Big( \sum_{j = 0}^{n} C p \int_{t_{j}}^{t_{j+1}} \big( 1+|x(t)|^{2 m-\epsilon}+|x_{n}|^{2 m-\epsilon}+ | \bar{x}_{T_{n}, t_{n+1}}|^{2 m-\epsilon} \big) \rd t\Big) \Big\|_{L^{2}(\Omega)}.
\end{align*}
Here, it follows from Jensen's inequality and Propositions \ref{expo} and \ref{Numerical:expo} that 
\begin{align*}
&\ \hE\Big[\exp \Big( \sum_{j = 0}^{n} C p \int_{t j}^{t_{j+1}}(1+|x(t)|^{2 m-\epsilon}+| x_n|^{2 m-\epsilon}+|\bar{x}_{T_n, t_{n+1}}|^{2 m-\epsilon}) \rd t\Big)\Big] \\
&\leq \hE\Big[\exp \Big(C p \int_{0}^{T}(1+|x(t)|^{2 m-\epsilon}+|x_{n}|^{2 m-\epsilon}+|\bar{x}_{T_n, t_{n+1}}|^{2 m-\epsilon}) \rd t\Big)\Big] \\
&\leq \frac{1}{T} \int_{0}^{T} \hE\Big[\exp \Big(T C p(1+|x(t)|^{2 m-\epsilon}+|x_{n}|^{2 m-\epsilon}+|\bar{x}_{T_n, t_{n+1}}|^{2 m-\epsilon})\Big)\Big] \rd t \\
&\leq C(p, T).
\end{align*}
Combining \eqref{ME}, \eqref{exactexpo}, Corollary \ref{numerical:momentbound} and the above estimates deduces that $ \hE[\mathcal{E}_n^{2 p}] \leq C h^{p}$ for all $n = 1, \cdots, T/h$. This together with the fact that $\|Y_{n}-Y(t_{n})\|^{2 p} \leq C \mathcal{E}_{n}^{2 p}$ completes the proof.
\end{proof}

Then, the following corollary provides a strong convergence result of order $\frac{1}{2}$ for the time-continuous version of the splitting AVF method.

\begin{corollary}\label{coro:half_convergence_rate}
Let $\tilde{Y}_t = (\tilde{v}_{t},\tilde{z}_{t},\tilde{x}_{t})^\top$ be defined by \eqref{continuousversion}. If Assumption \ref{ass.1} holds, then there exists $C>0$ such that for any $p \geq 1$, 
\begin{align*}
\sup_{t\in[0, T]} \| Y(t) - \tilde{Y}_t \|_{L^{2p}(\Omega, \hR^{k+2})}
\leq C h^{1/2}.
\end{align*}
\end{corollary}

\begin{proof} According to \eqref{continuousversion} and Corollary \ref{numerical:momentbound}, one has 
\begin{align*}
\sup_{t\in[0, T]} \| \tilde{Y}_{\lfloor t \rfloor}-\tilde{Y}_t \|_{L^{2p}(\Omega, \hR^{k+2})}
\leq C h^{1/2}.
\end{align*}
Then it follows from 
Lemma \ref{1/2strongconver} that 
\begin{align*}
\sup_{t\in[0, T]} \| Y(t)-\tilde{Y}_t \|_{L^{2p}(\Omega, \hR^{k+2})}
= \sup_{t\in[0, T]} 
\| Y(t) - Y(\lfloor t \rfloor) + Y(\lfloor t \rfloor) - \tilde{Y}_{\lfloor t \rfloor} + \tilde{Y}_{\lfloor t \rfloor} - \tilde{Y}_t \|_{L^{2p}(\Omega, \hR^{k+2})}
\leq Ch^{1/2}, 
\end{align*}
which completes the proof. 
\end{proof}

Based on the established strong convergence result of order $\frac{1}{2}$, we further derive the first-order convergence of the splitting AVF method \eqref{numsolu}--\eqref{sdesolu}. To this end, we rewrite the GLE \eqref{GLE} in the following compact form:\ 
\begin{align*}
\rd Y(t) = \mu(Y(t)) \rd t + \sigma \rd \tilde{W} (t), \qquad t \geq 0, 
\end{align*}
where $\sigma:= \text{diag} \left(\sqrt{2 \gamma}, \sqrt{2 \alpha_1}, \cdots, \sqrt{2 \alpha_k}, 0\right)$, $\tilde{W} := (W_0,W_1,\cdots,W_k,0)^\top$, and
\begin{align*}
\mu(y):= &\left( 
\begin{array}{c}
- \gamma y_0 - \nabla U(y_{k+1}) + \sum_{\ell = 1}^k \lambda_{\ell} y_{\ell} \\
- \alpha_1 y_1 - \lambda_{1} y_0 \\
\vdots \\
- \alpha_k y_{k} - \lambda_{k} y_{0} \\
y_0
\end{array} 
\right), \qquad 
\text{for} \,
y = 
\left( 
\begin{array}{c}
y_0\\
y_1\\
\vdots\\
y_{k}\\
y_{k+1}
\end{array} 
\right)\in \hR^{k+2}.
\end{align*}
For the sake of convenience, we also define
\begin{align*}
a(t) := 
\left(\begin{array}{c}
h^{-1} \mathcal{A}^n - \frac{\gamma}{2} \tilde{v}_{T_n, t} \\
h^{-1} \mathcal{B}_1^n - \alpha_1 \tilde{z}_{T_n, t}^1 + \frac{\gamma}{2} \lambda_1 \tilde{x}_{T_n, t} \\
\vdots \\
h^{-1} \mathcal{B}_k^n - \alpha_k \tilde{z}_{T_n, t}^k + \frac{\gamma}{2} \lambda_k \tilde{x}_{T_n, t} \\
h^{-1} \mathcal{C}^n - \frac{\gamma}{2} \tilde{x}_{T_n, t}
\end{array}\right), 
\qquad 
\hat{\mu}(y) := \left( 
\begin{array}{c}
- \frac{\gamma}{2} y_0 - \nabla U(y_{k+1}) + \sum_{\ell = 1}^k \lambda_{\ell} y_{\ell} \\
- \lambda_{1} y_0 - \frac{\gamma}{2} \lambda_{1} y_{k+1} \\
\vdots \\
- \lambda_{k} y_0 - \frac{\gamma}{2} \lambda_{k} y_{k+1}\\
\frac{\gamma}{2} y_{k+1} + y_0
\end{array} 
\right).
\end{align*}
Then $\tilde{Y}_t$ can be expressed in the form $
\rd \tilde{Y}_t = 
a(t) \rd t + \sigma \rd \tilde{W} (t)$.
In addition, it follows from Assumption \ref{ass.1} that for any $y, y' \in \hR^{k+2}$, 
\begin{align} \label{mubound}
\|\mu(y)-\mu(y')\| 
\leq \Big( C_U \big( 1 + \|y\|^{2 m-\epsilon} + \|y'\|^{2 m-\epsilon} \big) + \gamma + \sum_{\ell = 1}^k (\alpha_{\ell} + 2 \lambda_{\ell}) + 1 \Big) \|y - y'\|. 
\end{align}

\begin{theorem} [Strong convergence rate] \label{1strongconver}
Let $h^{*}$ be defined by \eqref{eq.def:h_star}. If Assumption \ref{ass.1} holds, then there exists $C := C(T)>0$ and $\tilde{h} \in(0, h^{*})$ such that for any $h \in(0, \tilde{h})$ and $p\ge1$, 
\begin{align*}
\Big\| \sup_{n \leq T/h}\|Y(t_{n}) - Y_{n}\| \Big\|_{L^{2p}(\Omega)} \leq C h. 
\end{align*}
\end{theorem}

\begin{proof}
For brevity, we denote 
\begin{align*} 
\chi_t := 
2 \Big( C_U \big( 1 + |x (t_n)|^{2 m-\epsilon} + |x_n|^{2 m-\epsilon} \big) + \gamma + \sum_{\ell = 1}^k (\alpha_{\ell} + 2 \lambda_{\ell}) + 1 \Big), \qquad t\in[t_n, t_{n+1}). 
\end{align*}
Using the chain rule and \eqref{mubound} yields 
\begin{align*}
\tfrac{\|Y(t) - \tilde{Y}_t\|^2}{\exp ( \int_0^t \chi_s \rd s)} 
&= \|Y(0) - \tilde{Y}_0\|^2 + \int_0^t \tfrac{2 \langle Y(s) - \tilde{Y}_s, \mu(Y(s))-a(s)\rangle - \|Y(s) - \tilde{Y}_s\|^2\chi_s}{\exp ( \int_0^s \chi_r \rd r)} \rd s \\
&= \int_0^t \tfrac{2 \langle Y(s) - \tilde{Y}_s, \mu(Y(s))-\mu(\tilde{Y}_s)\rangle - \|Y(s) - \tilde{Y}_s\|^2\chi_s}{\exp ( \int_0^s \chi_r \rd r)} \rd s + \int_0^t \tfrac{2 \langle Y(s) - \tilde{Y}_s, \mu(\tilde{Y}_s)-a(s)\rangle}{\exp ( \int_0^s \chi_r \rd r)} \rd s \\ 
&\leq \int_0^t \tfrac{2 C_{U} \|Y(s)-\tilde{Y}_s\|^2 \left(|x(s)|^{2m-\epsilon} - |x( \lfloor s \rfloor)|^{2m-\epsilon} + |\tilde{x}_s|^{2m-\epsilon} - |\tilde{x}_{\lfloor s \rfloor}|^{2m-\epsilon}\right)}{\exp ( \int_0^s \chi_r \rd r)} \rd s + \int_0^t \tfrac{2 \langle Y(s) - \tilde{Y}_s, \mu(\tilde{Y}_s)-a(s)\rangle}{\exp ( \int_0^s \chi_r \rd r)} \rd s.
\end{align*}
Owing to \eqref{GLE} and \eqref{continuousversion}, 
$\| x(t)- x(\lfloor t \rfloor) \|_{L^p(\Omega)} + \| \tilde{x}_t - \tilde{x}_{\lfloor t \rfloor} \|_{L^p(\Omega)} \leq C h$, which together with Corollary \ref{coro:half_convergence_rate} shows that for any $ \tau \in [0, T]$ and $p \geq 1$, 
\begin{align} \label{eq:order_one_estima_origin}
&\ \Big\| \sup_{t\in[0, \tau]} \tfrac{\|Y(t) - \tilde{Y}_t\|^2}{\exp ( \int_0^t \chi_s \rd s)} \Big\|_{L^p(\Omega)} \notag \\
&\leq \int_0^\tau \big\| \tfrac{C \|Y(t) - \tilde{Y}_t\|^2 \left((|x(t)|^{2m-\epsilon-1} + |x(\lfloor t \rfloor)|^{2m-\epsilon-1}) |x(t) - x(\lfloor t \rfloor)| + (|\tilde{x}_t|^{2m-\epsilon-1} + |\tilde{x}_{\lfloor t \rfloor}|^{2m-\epsilon-1}) |\tilde{x}_t - \tilde{x}_{\lfloor t \rfloor}|\right)}{\exp ( \int_0^t \chi_s \rd s)} \big\|_{L^p(\Omega)} \rd t \notag \\
&\quad + \Big\|\sup_{t\in[0, \tau]} \int_0^t \tfrac{2 \langle Y(s) - \tilde{Y}_s, \mu(\tilde{Y}_s)-a(s)\rangle}{\exp ( \int_0^s \chi_r \rd r)} \rd s \Big\|_{L^p(\Omega)} \notag \\
&\leq \Big\| \sup_{t\in[0, \tau]} \int_0^t \tfrac{2 \langle Y(s) - \tilde{Y}_s, \mu(\tilde{Y}_s)-a(s)\rangle}{\exp ( \int_0^s \chi_r \rd r)} \rd s \Big\|_{L^p(\Omega)} +Ch^2.
\end{align}
For $s \in [t_n, t_{n+1})$, $n = 0, 1, \cdots, T/h-1$, one has 
\begin{align*}
\mu(\tilde{Y}_s)-a(s) = \hat{\mu}(\tilde{Y}_s) - \hat{\mu}(\tilde{Y}_{t_n}) + J_n, 
\end{align*}
where
\begin{align*}
J_n = \left(\begin{array}{c}
-\frac{\gamma}{2} \bar{v}_{T_n, t_{n+1}} - \nabla U(\bar{x}_{T_n, t_{n+1}}) + \sum_{\ell = 1}^k \lambda_\ell \bar{z}_{T_n, t_{n+1}}^\ell - h^{-1}\mathcal{A}^n \\
-\lambda_1 \bar{v}_{T_n, t_{n+1}} - \frac{\gamma}{2} \lambda_1 \bar{x}_{T_n, t_{n+1}} - h^{-1}\mathcal{B}_1^n \\
\vdots \\
-\lambda_k \bar{v}_{T_n, t_{n+1}} - \frac{\gamma}{2} \lambda_k \bar{x}_{T_n, t_{n+1}} - h^{-1}\mathcal{B}_k^n \\
\frac{\gamma}{2} \bar{x}_{T_n, t_{n+1}} + \bar{v}_{T_n, t_{n+1}} - h^{-1}\mathcal{C}^n
\end{array}\right). 
\end{align*}
A similar argument to that in \cite{QuispelMcLaren2008} leads to 
\begin{align} \label{eq:estimate_for_J_n}
\sup_{n\leq T/h}\|J_{n}\|_{L^{2p}(\Omega)}^{2} \leq C h^{2}.
\end{align}
On the other hand, applying It\^o's formula yields 
\begin{align}\label{eq:mu(Y_s)-mu(Y_n)-ItoFormula}
\hat{\mu}(\tilde{Y}_{s})-\hat{\mu}(\tilde{Y}_{t_n}) = \int_{t_{n}}^{s} \big( \nabla\hat{\mu}(\tilde{Y}_{r}) a(r)+\frac{1}{2} \operatorname{tr}(\sigma^{\top} \nabla^2 \hat{\mu}(\tilde{Y}_{r}) \sigma) \big) \rd r + \int_{t_{n}}^{s} \nabla\hat{\mu}(\tilde{Y}_{r}) \sigma \rd \tilde{W} (r).
\end{align}
By the H\"older and Young inequalities, 
\begin{align} 
&\ \Big\|\sup_{t\in[0, \tau]} \int_0^t \tfrac{2 \langle Y(s) - \tilde{Y}_s, \mu (\tilde{Y}_s)-a(s)\rangle}{\exp ( \int_0^s \chi_r \rd r)} \rd s \Big\|_{L^p(\Omega)} \notag \\
&\leq \Big\|\sup_{t\in[0, \tau]} \int_0^t \tfrac{2 \left\langle Y(s) - \tilde{Y}_s, \int_{\lfloor s \rfloor}^s \left(\nabla \hat{\mu}(\tilde{Y}_r)a(r) + \frac{1}{2}\operatorname{tr}(\sigma^{\top} \nabla^2 \hat{\mu}(\tilde{Y}_{r}) \sigma)\right) \rd r \right\rangle}{\exp ( \int_0^s \chi_r \rd r)} \rd s \Big\|_{L^p(\Omega)} \notag \\
&\quad + \Big\|\sup_{t\in[0, \tau]} \int_0^t\tfrac{2 \left\langle Y(s) - \tilde{Y}_s, \int_{\lfloor s \rfloor}^s \nabla\hat{\mu}(\tilde{Y}_r)\sigma \rd \tilde{W} (r) \right\rangle}{\exp ( \int_0^s \chi_r \rd r)} \rd s \Big\|_{L^p(\Omega)} + \Big\|\sup_{t\in[0, \tau]} \int_0^t \tfrac{2 \left\langle Y(s) - \tilde{Y}_s, J_{\lfloor s \rfloor /h} \right\rangle}{\exp ( \int_0^s \chi_r \rd r)} \rd s \Big\|_{L^p(\Omega)} \notag \\
&\leq C \int_0^\tau \Big\|\tfrac{\|Y(s) - \tilde{Y}_s\|}{\exp ( \int_0^s \frac12\chi_r \rd r)} \Big\|^2_{L^{2p}(\Omega)} \rd s + \Big\|\sup_{t\in[0, \tau]} \int_0^t \tfrac{2 \left\langle Y(s) - \tilde{Y}_s, \int_{\lfloor s \rfloor}^s \nabla\hat{\mu}(\tilde{Y}_r)\sigma \rd \tilde{W} (r) \right\rangle}{\exp ( \int_0^s \chi_r \rd r)} \rd s \Big\|_{L^p(\Omega)} +Ch^2, \label{eq:X(s)-Ys, F(Ys)-a(s)_in_Thm}
\end{align}
where Corollary \ref{numerical:momentbound} and \eqref{eq:estimate_for_J_n} were also used. According to the chain rule, 
\begin{align*}
\tfrac{Y(s) - \tilde{Y}_s}{\exp ( \int_0^s {\chi_r} \rd r)} = \tfrac{Y({\lfloor s \rfloor}) - \tilde{Y}_{\lfloor s \rfloor}}{\exp ( \int_0^{\lfloor s \rfloor} {\chi_r} \rd r)} + \int_{\lfloor s \rfloor}^s \tfrac{\mu(Y(r)) - a(r)}{\exp ( \int_0^r {\chi_{\iota}} \rd \iota)}\rd r + \int_{\lfloor s \rfloor}^s \tfrac{(Y(r)-\tilde{Y}_r)(-\chi_r)}{\exp ( \int_0^r {{\chi_{\iota}}} \rd \iota)}\rd r, 
\end{align*}
which helps us to make the following decomposition:\ 
\begin{align*}
&\ \Big\|\sup_{t\in[0, \tau]} \int_0^t \tfrac{2 \langle Y(s) - \tilde{Y}_s, \int_{\lfloor s \rfloor}^s \nabla \hat{\mu}(\tilde{Y}_r)\sigma \, \rd \tilde{W} (r)\rangle}{\exp ( \int_0^s \chi_r \, \rd r)} \rd s \Big\|_{L^p(\Omega)} \\
&\leq \Big\| 2\sup_{t\in[0, \tau]} \int_0^t \Big\langle \tfrac{Y(\lfloor s \rfloor) - \tilde{Y}_{\lfloor s \rfloor}}{\exp ( \int_0^{\lfloor s \rfloor} \chi_r \, \rd r)} , \int_{\lfloor s \rfloor}^s \nabla \hat{\mu}(\tilde{Y}_r)\sigma \, \rd \tilde{W} (r) \Big\rangle \rd s \Big\|_{L^p(\Omega)} \\ 
&\quad + \Big\|2\sup_{t\in[0, \tau]} \int_0^t \Big\langle \int_{\lfloor s \rfloor}^s \tfrac{\mu(Y(r)) - a(r)}{\exp ( \int_0^r \chi_{\iota} \, \rd \iota)}\rd r , \int_{\lfloor s \rfloor}^s \nabla \hat{\mu}(\tilde{Y}_r)\sigma \, \rd \tilde{W} (r) \Big\rangle \rd s \Big\|_{L^p(\Omega)} \\
&\quad + \Big\|2\sup_{t\in[0, \tau]} \int_0^t \Big\langle \int_{\lfloor s \rfloor}^s \tfrac{(Y(r)-\tilde{Y}_r)(-\chi_r)}{\exp ( \int_0^r \chi_{\iota} \, \rd \iota)}\rd r , \int_{\lfloor s \rfloor}^s \nabla \hat{\mu}(\tilde{Y}_r)\sigma \, \rd \tilde{W} (r) \Big\rangle \rd s \Big\|_{L^p(\Omega)} \\
&= : K_1+K_2+K_3.
\end{align*} 
To proceed, we further decompose the term $K_1$ as follows:\ 
\begin{align*}
\| K_1 \|_{L^p(\Omega)} 
&\leq \Big\|\sup_{t\in[0, \tau]}\Big| \sum\limits_{k = 0}^{\lfloor t \rfloor/h-1} \int_{t_k}^{t_{k+1}} {2\Big\langle \tfrac{Y(\lfloor s \rfloor) - \tilde{Y}_{\lfloor s \rfloor}}{\exp ( \int_0^{\lfloor s \rfloor} {\chi_r} \rd r)}, \int_{\lfloor s \rfloor}^s \nabla\hat{\mu}(\tilde{Y}_r)\sigma \rd \tilde{W} (r) \Big\rangle}\rd s \Big| \Big\|_{L^p(\Omega)} \\
&\quad + \Big\|\sup_{t\in[0, \tau]} \int_{\lfloor t \rfloor}^t {2\Big\langle \tfrac{Y(\lfloor s \rfloor) - \tilde{Y}_{\lfloor s \rfloor}}{\exp ( \int_0^{\lfloor s \rfloor} {\chi_r} \rd r)}, \int_{\lfloor s \rfloor}^s \nabla\hat{\mu}(\tilde{Y}_r)\sigma \rd \tilde{W} (r) \Big\rangle}\rd s \Big\|_{L^p(\Omega)} \\
&= :K_{1, 1}+K_{1, 2}.
\end{align*}
Here, duo to It\^o's isometry and Lemma \ref{1/2strongconver}, the term $K_{1, 2}$ can be immediately estimated by
\begin{align}\label{eq:order_one_estima_S_12}
K_{1, 2}\leq Ch^2.
\end{align}
In contrast, the term $K_{1, 1}$ should be treated more carefully. Noting that
\begin{align*}
\zeta_n := \sum\limits_{k = 0}^{n-1} \int_{t_k}^{t_{k+1}} 
{2\Big\langle \tfrac{Y(\lfloor s \rfloor) - \tilde{Y}_{\lfloor s \rfloor}}{\exp ( \int_0^{\lfloor s \rfloor} {\chi_r} \rd r)} , \int_{\lfloor s \rfloor}^s \nabla\hat{\mu}(\tilde{Y}_r)\sigma \rd \tilde{W} (r) \Big\rangle}\rd s, \qquad n\in \mathbb{N}
\end{align*}
is a discrete martingale. Then using Doob's discrete martingale inequality, BDG's inequality and H\"older's inequality shows 
\begin{align}\label{eq:order_one_estima_S_11}
K_{1, 1} 
&\leq C_p \Big\| \sum\limits_{k = 0}^{\lfloor \tau \rfloor/h-1} \int_{t_k}^{t_{k+1}} {2\Big\langle \tfrac{Y(\lfloor s \rfloor) - \tilde{Y}_{\lfloor s \rfloor}}{\exp ( \int_0^{\lfloor s \rfloor} {\chi_r} \rd r)}, \int_{\lfloor s \rfloor}^s \nabla\hat{\mu}(\tilde{Y}_r)\sigma \rd \tilde{W} (r) \Big\rangle}\rd s \Big\|_{L^p(\Omega)} \notag \\
&\leq C_p \bigg( \sum\limits_{k = 0}^{\lfloor \tau \rfloor/h-1} \Big\| \int_{t_k}^{t_{k+1}} {2\Big\langle \tfrac{Y(\lfloor s \rfloor) - \tilde{Y}_{\lfloor s \rfloor}}{\exp ( \int_0^{\lfloor s \rfloor} {\chi_r} \rd r)}, \int_{\lfloor s \rfloor}^s \nabla\hat{\mu}(\tilde{Y}_r)\sigma \rd \tilde{W} (r) \Big\rangle}\rd s \Big\|^2_{{L^p(\Omega)}}\bigg)^{\frac{1}{2}} \notag \\
&\leq C_p \bigg(h \int_{0}^{\tau} \Big\|\tfrac{\|Y(\lfloor s \rfloor)-\tilde{Y}_{\lfloor s \rfloor}\|}{\exp ( \int_0^{\lfloor s \rfloor} {\frac12} {{\chi_r}} \rd r)} \Big\|^2_{{L^{2p}(\Omega)}} \Big\| \int_{\lfloor s \rfloor}^s \nabla\hat{\mu}(\tilde{Y}_r)\sigma \rd \tilde{W} (r) \Big\|^2_{{L^{2p}(\Omega, \hR^{k+2})}} \rd s \bigg)^{\frac{1}{2}} \notag \\
&\leq C_p \sup_{t\in [0, \tau]} \Big\|\tfrac{\|Y(t) - \tilde{Y}_t\|}{\exp ( \int_0^t {\frac12} {\chi_r} \rd r)} \Big\|_{L^{2p}(\Omega)} \bigg( h \int_{0}^{\tau} \Big\| \int_{\lfloor s \rfloor}^s \nabla\hat{\mu}(\tilde{Y}_r)\sigma \rd \tilde{W} (r) \Big\|^2_{L^{2p}(\Omega, \hR^{k+2})}\rd s \bigg)^{\frac{1}{2}} \notag \\
&\leq \dfrac{1}{8} \sup_{t\in [0, \tau]} \Big\|\tfrac{\|Y(t) - \tilde{Y}_t\|}{\exp ( \int_0^t {\frac12} {\chi_r} \rd r)} \Big\|^2_{L^{2p}(\Omega)} + Ch^2.
\end{align}
Next, we estimate the term $K_2$. In view of \eqref{mubound}, \eqref{eq:mu(Y_s)-mu(Y_n)-ItoFormula} and \eqref{eq:estimate_for_J_n}, 
\begin{align}\label{eq:order_one_estima_S_2}
K_2 &\leq \Big\|2\sup_{t\in[0, \tau]} \int_0^t \Big\langle \int_{\lfloor s \rfloor}^s \tfrac{\mu(Y(r)) - \mu(\tilde{Y}_r)}{\exp ( \int_0^r {\chi_{\iota}} \rd \iota)}\rd r , \int_{t_n}^s \nabla\hat{\mu}(\tilde{Y}_r)\sigma \rd \tilde{W} (r) \Big\rangle \rd s \Big\|_{L^p(\Omega)} \notag \\
&\quad + \Big\|2\sup_{t\in[0, \tau]} \int_0^t \Big\langle \int_{\lfloor s \rfloor}^s \tfrac{\mu(\tilde{Y}_r)-a(r)}{\exp ( \int_0^r {\chi_{\iota}} \rd \iota)}\rd r , \int_{t_n}^s \nabla\hat{\mu}(\tilde{Y}_r)\sigma \rd \tilde{W} (r) \Big\rangle \rd s \Big\|_{L^p(\Omega)} \notag \\
&\leq C \int_0^\tau \Big( \int_{\lfloor s \rfloor}^s \Big\|\tfrac{( 1 + |x(r)|^{2m} + |\tilde{x}_r|^{2m} ) \| Y(r) - \tilde{Y}_r \|}{\exp ( \int_0^r {\chi_{\iota}} \rd \iota)} \Big\|_{L^{2p}(\Omega)} \rd r \Big) \Big\| \int_{t_n}^s \nabla\hat{\mu}(\tilde{Y}_r)\sigma \rd \tilde{W} (r) \Big\|_{L^{2p}(\Omega, \hR^{k+2})} \rd s \notag \\
&\quad + 2 \int_0^\tau \Big( \int_{\lfloor s \rfloor}^s \Big\|\tfrac{|\hat{\mu}(\tilde{Y}_r) - \hat{\mu}(\tilde{Y}_{\lfloor r \rfloor})| + |J_{\lfloor r \rfloor / h}}{\exp ( \int_0^r {\chi_{\iota}} \rd \iota)} \Big\|_{L^{2p}(\Omega)} \rd r \Big)\Big\| \int_{t_n}^s \nabla\hat{\mu}(\tilde{Y}_r)\sigma \rd \tilde{W} (r) \Big\|_{L^{2p}(\Omega, \hR^{k+2})} \rd s \notag \\
&\leq Ch^2.
\end{align}
In a similar manner, one can also get $K_3 \leq Ch^2$. Recalling \eqref{eq:order_one_estima_S_11}, \eqref{eq:order_one_estima_S_12}, \eqref{eq:order_one_estima_S_2} indicates 
\begin{align*}
\Big\|\sup_{t\in[0, \tau]} \int_0^t \tfrac{2 \left\langle Y(s) - \tilde{Y}_s, \int_{\lfloor s \rfloor}^s \nabla\hat{\mu}(\tilde{Y}_r)\sigma \rd \tilde{W} (r) \right\rangle}{\exp ( \int_0^s \chi_r \rd r)} \rd s \Big\|_{L^p(\Omega)} \leq \dfrac{1}{8} \sup_{t\in [0, \tau]} \Big\|\tfrac{\|Y(t) - \tilde{Y}_t\|}{\exp ( \int_0^t \frac{1}{2}{\chi_r} \rd r)}\Big\|^2_{L^{2p}(\Omega)}+Ch^2.
\end{align*}
Inserting the above estimate into \eqref{eq:X(s)-Ys, F(Ys)-a(s)_in_Thm} yields 
\begin{align*}
&\ \Big\|\sup_{t\in[0, \tau]} \int_0^t \tfrac{2 \langle Y(s) - \tilde{Y}_s, \mu(\tilde{Y}_s)-a(s)\rangle}{\exp ( \int_0^s \chi_r \rd r)} \rd s \Big\|_{L^p(\Omega)} \\
&\leq C \int_0^\tau \Big\|\tfrac{\|Y(s) - \tilde{Y}_s\|}{\exp ( \int_0^s \frac12\chi_r \rd r)} \Big\|^2_{L^{2p}(\Omega)} \rd s + \dfrac{1}{8} \sup_{t\in [0, \tau]} \Big\|\tfrac{\|Y(t) - \tilde{Y}_t\|}{\exp ( \int_0^t \frac{1}{2}{\chi_r} \rd r)}\Big\|^2_{L^{2p}(\Omega)}+Ch^2, 
\end{align*}
which together with \eqref{eq:order_one_estima_origin} in turns gives that for any $\tau\in[0, T]$, 
\begin{align*}
&\ \Big\| \sup_{t\in[0, \tau]} \tfrac{\|Y(t) - \tilde{Y}_t\|^2}{\exp ( \int_0^t \chi_s \rd s)} \Big\|_{L^p(\Omega)} 
= \Big\| \sup_{t\in[0, \tau]} \tfrac{\|Y(t) - \tilde{Y}_t\|}{\exp ( \int_0^t \frac12 \chi_s \rd s)} \Big\|^2_{L^{2p}(\Omega)} \\
&\leq C \int_0^\tau \Big\|\sup_{s\in[0, t]} \tfrac{\|Y(s) - \tilde{Y}_s\|}{\exp ( \int_0^s \frac12\chi_r \rd r)}\Big\|^2_{L^{2p}(\Omega)} \rd t + \dfrac{1}{8} \Big\|\sup_{t\in [0, \tau]} \tfrac{\|Y(t)-\tilde{Y}_t\|}{\exp ( \int_0^t \frac{1}{2}{\chi_r} \rd r)} \Big\|^2_{L^{2p}(\Omega)} +Ch^2.
\end{align*}
By Gr\"onwall's inequality, 
\begin{align*}
&\Big\| \sup_{t\in[0, T]} \tfrac{\|Y(t) - \tilde{Y}_t\|^2}{\exp ( \int_0^t \chi_s \rd s)} \Big\|_{L^p(\Omega)} \leq Ch^2 e^{CT}.
\end{align*}
Finally, using H\"{o}lder's inequality and Propositions \ref{expo} and \ref{Numerical:expo} concludes 
\begin{align*}
\!
\Big\|\sup_{t\in[0, T]} \|Y(t) - \tilde{Y}_t\|\Big\|_{L^{2p}(\Omega)} \leq \Big\|\sup_{t\in[0, T]} \tfrac{\|Y(t) - \tilde{Y}_t\|}{\exp ( \int_0^t {\frac{1}{2}\chi_r} \rd r)}\Big\|_{L^{4p}(\Omega)}
\!
\Big\| {\exp \Big( \int_0^{T} {\tfrac{1}{2}{\chi_r}} \rd r \Big)}\Big\|_{L^{4p}(\Omega)} \leq C h, 
\end{align*}
which completes the proof. 
\end{proof}

\section{Probability density function and its convergence}
\label{sec.PDF}

In this section, we first study the existence and smoothness of the probability density function of the numerical solution generated by the splitting AVF method \eqref{numsolu}--\eqref{sdesolu}, and then show that it converges with order 1 to the probability density function of the exact solution to the GLE \eqref{GLE}.

\subsection{Existence and smoothness of the probability density function}
In this subsection, we show the Malliavin differentiability and uniform non-degeneracy of the splitting AVF scheme \eqref{numsolu}--\eqref{sdesolu}, which in turn implies that the numerical solution admits a smooth probability density function. With the deterministic initial value, we first establish the Malliavin differentiability of the numerical solution using mathematical induction.

\begin{proposition}[Malliavin differentiability] \label{Numerical:Malliavindifferentiability}
If Assumption \ref{ass.1} holds, then $Y_n \in \mathbb{D}^{\infty}(\hR^{k + 2})$ for $n = 0, 1, \cdots, T/h$. 
\end{proposition}

\begin{proof}
Assume that $Y_n \in \mathbb{D}^{\infty}(\hR^{k+2})$. We aim to show that $Y_{n + 1} \in \mathbb{D}^{\infty} (\hR^{k+2})$. To this end, we begin by showing $\bar{Y}_{T_n, t_{n + 1}} \in \mathbb{D}^{\infty}(\hR^{k+2})$. For $y\in \hR^{k + 2}$, define $\rho_N(y)$ to be equal to $1$ if $\|y\|<N$, and $0$ if $\|y\| \geq N+1$. Clearly, 
\begin{align} \label{Drho}
\sup_{y\in \hR^{k + 2}, N \ge 1} | D^{\alpha} \rho_N (y)| < \infty, \qquad \forall \, \alpha \in \mathbb{N}^{k + 2}.
\end{align}
Denote $\phi_h^N := \phi_h \rho_N$. Then $\phi_h^N(Y_n) \to \phi_h(Y_n)$ almost surely as $N \to \infty$, and $\|\phi_h^N(Y_n)\| \leq\|\phi_h(Y_n)\| = \|\bar{Y}_{T_n, t_{n + 1}}\| \in L^{\infty -}(\Omega)$ by Corollary \ref{numerical:momentbound}. It follows from the dominated convergence theorem that for any $p \ge 1$, 
\begin{align*}
\lim_{N \to \infty} \hE \big[ \|\phi_h^N(Y_n) - \phi_h(Y_n)\|^p \big] = 0. 
\end{align*}
Using $\phi_h^N \in C_0^{\infty}(\hR^{k + 2})$ and the induction hypothesis $Y_n \in \mathbb{D}^{1, \infty}(\hR^{k + 2})$ indicates $\phi_h^N(Y_n) \in \mathbb{D}^{1, \infty}(\hR^{k + 2})$ (see, e.g., \cite[Proposition 1.2.4]{NualartDavid2006}). By \cite[Lemma 1.5.3]{NualartDavid2006} and H\"older's inequality, to prove $\bar{Y}_{T_n, t_{n + 1}} \in \mathbb{D}^{1, \infty}(\hR^{k+2})$, it remains to show that for any $p \ge 1$, 
\begin{align} \label{DphiN}
\sup_{N \ge 1} \hE \big[ \|\mathcal{D} \phi_h^N(Y_n)\|^p_{\mathbb{H}^{\otimes k+2}} \big] <\infty.
\end{align}
Applying the chain rule yields that for a.e.\ $r \in [0, t_n]$, 
\begin{align*}
\mathcal{D}_{r} \phi_{h}^{N}(Y_n) = \nabla \phi_{h}(Y_n) \mathcal{D}_{r} Y_{n} \rho_{N}(Y_n) + \phi_{h}(Y_n) \nabla \rho_{N}(Y_n) \mathcal{D}_{r} Y_{n} . 
\end{align*}
Based on \eqref{Drho} and the inductive hypothesis $Y_{n} \in \mathbb{D}^{1, \infty}(\hR^{k + 2})$, the proof of \eqref{DphiN} reduces to showing that the gradient $\nabla \phi_{h}(Y_n)$ belongs to $L^{\infty-}(\Omega, \hR^{k + 2} \otimes \hR^{k + 2} )$. To analyze this gradient, we define
\begin{align*}
& F_{1}(x_{n}, \bar{x}_{T_{n}, t_{n + 1}}) := \int_{0}^{1} \nabla^{2} U(x_{n} + \xi(\bar{x}_{T_{n}, t_{n + 1}} - x_{n})) \xi \rd \xi, \\
& F_{2}(x_{n}, \bar{x}_{T_{n}, t_{n + 1}}) := \int_{0}^{1} \nabla^{2} U(x_{n} + \xi(\bar{x}_{T_{n}, t_{n + 1}} - x_{n}))(1 - \xi) \rd \xi.
\end{align*}
From \eqref{nabla2U}, we have $F_1 \wedge F_2 \ge -K/2$. By differentiating the relation \eqref{implicitavf} and applying the chain rule, one can express the gradient as
\begin{align} \label{nablaphih}
\nabla \phi_{h}(Y_{n}) = - \Big(\frac{\partial G}{\partial \bar{Y}_{T_n, t_{n + 1}}}\Big)^{-1} \frac{\partial G}{\partial Y_{n}} = - \frac{1}{\operatorname{det}(\frac{\partial G}{\partial \bar{Y}_{T_n, t_{n + 1}}})}\Big(\frac{\partial G}{\partial \bar{Y}_{T_n, t_{n + 1}}}\Big)^{*} \frac{\partial G}{\partial Y_{n}}, 
\end{align}
where
{\footnotesize 
\begin{align*}
\frac{\partial G}{\partial \bar{Y}_{T_n, t_{n + 1}}} = \left(\begin{array}{ccccc}
1 + \frac{\gamma h}{4} &- \frac{\lambda_1}{2} h &\cdots &- \frac{\lambda_k}{2} h &h F_1 \\
\frac{\lambda_1}{2} h &1 && 0 &\frac{\gamma \lambda_1}{4} h \\
\vdots &&\ddots &&\vdots \\
\frac{\lambda_k}{2} h & 0 &&1 &\frac{\gamma \lambda_k}{4} h \\
- \frac{h}{2} & 0 &\cdots & 0 &1 - \frac{\gamma h}{4}
\end{array}\right), 
\qquad
\frac{\partial G}{\partial Y_n} = - \left(\begin{array}{ccccc}
1 - \frac{\gamma h}{4} &\frac{\lambda_1}{2} h &\cdots &\frac{\lambda_k}{2} h &h F_2 \\
- \frac{\lambda_1}{2} h &1 && 0 &- \frac{\gamma \lambda_1}{4} h \\
\vdots &&\ddots &&\vdots \\
- \frac{\lambda_k}{2} h & 0 &&1 &- \frac{\gamma \lambda_k}{4} h \\
\frac{h}{2} & 0 &\cdots & 0 &1 - \frac{\gamma h}{4}
\end{array}\right). 
\end{align*}
}
By a direct calculation, 
\begin{align*} 
\operatorname{det}\left( \frac{\partial G}{\partial \bar{Y}_{T_n, t_{n + 1}}} \right) 
 = 1 + h^{2} \left( \frac{1}{4} \sum_{\ell = 1}^{k} \lambda_{\ell}^{2} + F_{1} - \frac{\gamma^{2}}{16} \right). 
\end{align*}
One can choose $h \leq \sqrt{(\frac{\gamma^{2}}{16} + \frac{K^{2}}{2}) / 2}$ to ensure $1 + h^{2}(\frac{1}{4} \sum_{\ell = 1}^{k} \lambda_{\ell}^{2} + F_{1} - \frac{\gamma^{2}}{16}) \ge 1 - h^{2}(\frac{\gamma^{2}}{16} + \frac{K}{2}) \ge \frac{1}{2}$. Consequently, 
\begin{align} \label{redetG}
\frac{1}{\operatorname{det}\left( \frac{\partial G}{\partial \bar{Y}_{T_n, t_{n + 1}}}\right)} 
\leq 1 + \frac{h^{2}(\frac{\gamma^{2}}{16} + \frac{K}{2})}{1 - h^{2}(\frac{\gamma^{2}}{16} + \frac{K}{2})} \leq 1 + h^{2} \left( \frac{\gamma^{2}}{8} + K \right).
\end{align}
The adjugate matrix $\big(\frac{\partial G}{\partial \bar{Y}_{T_n, t_{n + 1}}}\big)^{*}$ can be written explicitly as 
\begin{align*}
\left(\frac{\partial G}{\partial \bar{Y}_{T_n, t_{n + 1}}}\right)^{*} 
 = \left(\begin{array}{ccccc}
1 - \frac{\gamma}{4} h &\frac{h}{2}(1 - \frac{\gamma}{4} h) \lambda_1 &\cdots &\frac{h}{2}(1 - \frac{\gamma}{4} h) \lambda_k &- hF_1 - \frac{\gamma}{8} h^{2} \sum_{\ell = 1}^{k} \lambda_{\ell}^{2} \\
- \frac{h}{2} \lambda_1 &1 - \frac{h^2}{4} \lambda_1^2 &&- \frac{h^2}{4} \lambda_1 \lambda_k &- h \lambda_1(\frac{\gamma}{4} + \frac{\gamma^2}{16} h - \frac{h}{2} F_1) \\
\vdots &&\ddots &&\vdots \\
- \frac{h}{2} \lambda_k &- \frac{h^2}{4} \lambda_k \lambda_1 &&1 - \frac{h^2}{4} \lambda_k^{2} &- h \lambda_k(\frac{\gamma}{4} + \frac{\gamma^2}{16} h - \frac{h}{2} F_1) \\
\frac{h}{2} &\frac{h^2}{4} \lambda_1 &\cdots &\frac{h^2}{4} \lambda_k &1 + \frac{\gamma}{4} h + \frac{h^2}{4} \sum_{\ell = 1}^{k} \lambda_{\ell}^{2}
\end{array}\right).
\end{align*}
Meanwhile, for the 1-norm of the matrices $\frac{\partial G}{\partial Y_n}$ and $\Big(\frac{\partial G}{\partial \bar{Y}_{T_n, t_{n + 1}}}\Big)^{*}$, one has 
\begin{align}
&\Big\|\frac{\partial G}{\partial Y_n}\Big\|_1 \notag 
\leq \max \Big\{|1 - \frac{\gamma h}{4}| + h \sum_{\ell = 1}^k \frac{\lambda_{\ell}}{2} + \frac{h}{2}, \max_{1\le \ell \le k} \big\{1 + \frac{\lambda_{\ell} h}{2}\big\}, h| F_2| + h \sum_{\ell = 1}^h \frac{\gamma \lambda_{\ell}}{4} + |1 - \frac{\gamma h}{4}|\Big\}\notag \\
&\qquad\quad\, \ \leq 1 + h(C + |F_2|), \label{GYnorn} \\
&\Big\|\Big(\frac{\partial G}{\partial \bar{Y}_{T_n, t_{n + 1}}}\Big)^{*}\Big\|_{1} \notag \\
&\leq \max \bigg\{|1 - \frac{\gamma}{4} h| + h \sum_{\ell = 1}^{k} \frac{\lambda_{\ell}}{2} + \frac{h}{2}, \max_{1\le j \le k} \Big\{\frac{\lambda_{j} h}{2}|1 - \frac{\gamma}{4} h| + |1 - \frac{h^{2}}{4} \lambda_{j}^{2}| + h^{2} \sum_{1 \le \ell \le k, \ell \neq j} \frac{\lambda_{\ell} \lambda_{j}}{4} + h^{2} \frac{\lambda_{j}}{4}\Big\}, \notag \\
&\qquad\qquad\ |h F_{1} + h^{2} \frac{\gamma}{8} \sum_{\ell = 1}^{k} \lambda_{\ell}^{2}| + h \sum_{\ell = 1}^{k} \lambda_{\ell}|\frac{\gamma}{4} + \frac{\gamma^{2}}{16} h - \frac{h}{2} F_{1}| + 1 + \frac{\gamma}{4} h + \frac{h^{2}}{4} \sum_{\ell = 1}^{k} \lambda_{\ell}^{2}\bigg\} \notag \\
&\leq 1 + h(C + |F_{2}| + h(C + |F_{1}|)). \label{GYTnorm}
\end{align}
Combining the estimates \eqref{redetG}, \eqref{GYnorn}, and \eqref{GYTnorm} with Assumption \ref{ass.1} and Corollary \ref{numerical:momentbound} concludes that each component of $\nabla \phi_{h}(Y_n)$ belongs to $L^{\infty -}(\Omega)$. This establishes that $\bar{Y}_{T_{n}, t_{n + 1}} \in \mathbb{D}^{1, \infty}(\hR^{k + 2})$. Using \cite[Proposition 1.5.1]{NualartDavid2006} and taking the Malliavin derivative to \eqref{implicitavf} show that for a.e.\ $r \in[0, t_n]$, 
\begin{align*}
\mathcal{D}_{r} \bar{Y}_{T_{n}, t_{n + 1}} = - \Big(\frac{\partial G}{\partial \bar{Y}_{T_n, t_{n + 1}}}\Big)^{-1} \frac{\partial G}{\partial Y_{n}}\mathcal{D}_{r} Y_{n} = - \frac{1}{\operatorname{det}(\frac{\partial G}{\partial \bar{Y}_{T_n, t_{n + 1}}})}\Big(\frac{\partial G}{\partial \bar{Y}_{T_n, t_{n + 1}}}\Big)^{*} \frac{\partial G}{\partial Y_{n}}\mathcal{D}_{r} Y_{n}, 
\end{align*} 
while for a.e.\ $r \in(t_{n}, t_{n + 1}]$, $\mathcal{D}_{r} \bar{Y}_{T_{n}, t_{n + 1}} = 0$. Finally, recalling \eqref{sdesolu} yields that for a.e.\ $r \in[0, t_{n}]$, 
\begin{align} \label{DrYn1}
\mathcal{D}_{r} Y_{n + 1} = M \mathcal{D}_{r} \bar{Y}_{T_{n}, t_{n + 1}} = A_{n} \mathcal{D}_{r} Y_{n}, 
\end{align} 
where $A_{n} = - M \big( \frac{\partial G}{\partial \bar{Y}_{T_n, t_{n + 1}}} \big)^{-1} \frac{\partial G}{\partial Y_{n}}$ with
\begin{align} \label{Mmatrix}
M = \left(\begin{array}{ccccc}
e^{-\frac{\gamma}{2} h} & 0 &\cdots & 0 & 0 \\
0 &e^{-\alpha_{1} h}&& 0 &\frac{\gamma \lambda_{1}}{2 \alpha_{1} - \gamma} ( e^{-\frac{\gamma}{2} h} - e^{-\alpha_{1} h}) \\
\vdots &&\ddots &&\vdots \\
0 & 0 &&e^{-\alpha_{k} h} &\frac{\gamma \lambda_{k}}{2 \alpha_{k} - \gamma} ( e^{-\frac{\gamma}{2} h} - e^{-\alpha_{k} h}) \\
0 & 0 &\cdots & 0 &e^{-\frac{\gamma}{2} h}
\end{array}\right). 
\end{align}
For $r \in(t_{n}, t_{n + 1}]$, it follows from $\bar{Y}_{T_{n}, t_{n + 1}} \in \mathbb{D}^{1, \infty}(\hR^{k + 2})$ that
\begin{align} \label{DrYn2}
\mathcal{D}_{r} Y_{n + 1} = \text{diag} \left(e^{-\frac{\gamma}{2}(t_{n + 1} - r)} \sqrt{2 \gamma}, \, e^{-\alpha_{1}(t_{n + 1} - r)}\sqrt{2 \alpha_1}, \, \cdots, \, e^{-\alpha_{k}(t_{n + 1} - r)}\sqrt{2 \alpha_k}, \, 0\right), 
\end{align}
which implies $Y_{n + 1} \in \mathbb{D}^{1, \infty}(\hR^{k + 2})$. By a similar argument, one can show that $Y_{n + 1} \in \mathbb{D}^{\alpha, \infty}(\hR^{k + 2})$ for all integers $\alpha \ge 2$. The proof is completed. 
\end{proof}

For the numerical solution $\{Y_n\}$ generated by the splitting AVF method \eqref{numsolu}--\eqref{sdesolu}, we next give the following moment estimate for its Malliavin derivative, whose proof relies on a recursive argument and Proposition \ref{Numerical:expo} (i.e., the exponential integrability of numerical solutions), and is similar to that in \cite[Lemma 4.3]{CuiHongSheng2022}. The proof is omitted here.

\begin{lemma} \label{Numercal:Malliavinbound}
Let $\alpha \in \mathbb{N}_{+}$ and $p \geq 1$. Under Assumption \ref{ass.1}, there exists $h_0 \in (0 , h^*)$ such that for any $h \in(0, h_{0}]$, 
\begin{align*}
& \sup_{r_{1}, \cdots, r_{\alpha} \in[0, T]} \hE \big[ \sup_{r_{1} \vee \cdots \vee r_{\alpha} \leq t_{n} \leq T}\|\mathcal{D}_{r_{1}, \cdots, r_{\alpha}} \bar{Y}_{T_{n}, t_{n + 1}}\|^{p} \big] \leq C, \\
& \sup_{r_{1}, \cdots, r_{\alpha} \in[0, T]} \hE \big[ \sup_{r_{1}\vee \cdots \vee r_{\alpha}\leq t_n \leq T}\|\mathcal{D}_{r_{1}, \cdots, r_{\alpha}} Y_{n}\|^{p} \big] \leq C, 
\end{align*}
for some positive constant $C = C(\alpha, p, T)$. 
\end{lemma}

We denote by $\gamma_{n} = \int_{0}^{t_{n}} \mathcal{D}_{r} Y_{n}(\mathcal{D}_{r} Y_{n})^{\top} \rd r$ the Malliavin covariance matrix of $Y_{n}$. To verify whether $\gamma_{n}$ is invertible, we proceed to establish a recursive relation between $\gamma_{n + 1}$ and $\gamma_{n}$. As a direct consequence of \eqref{DrYn1} and \eqref{DrYn2}, it follows that
\begin{align} \label{gamman1}
\gamma_{n + 1} &= \int_{0}^{t_{n + 1}} \mathcal{D}_{r} Y_{n + 1}(\mathcal{D}_{r} Y_{n + 1})^{\top} \rd r \notag \\
&= \int_{0}^{t_{n}} \mathcal{D}_{r} Y_{n + 1}(\mathcal{D}_{r} Y_{n + 1})^{\top} \rd r + \int_{t_{n}}^{t_{n + 1}} \mathcal{D}_{r} Y_{n + 1}(\mathcal{D}_{r} Y_{n + 1})^{\top} \rd r \notag \\
&= A_{n} \gamma_{n} A_{n}^{\top} + \gamma_{1} \quad \text {a.s., }
\end{align}
where $\gamma_{1} = \mbox{diag} \left(
2 - 2 e^{-\gamma h}, \, 1 - e^{-2 \alpha_{1} h}, \, 
\cdots, \, 1 - e^{-2 \alpha_{k} h}, \, 0 \right)$. Following \cite[Theorem 4.5]{CuiHongSheng2022}, the Malliavin covariance matrix $\gamma_n$ of $Y_n$ is invertible a.s.\ for $n=2, \ldots, T/h$, which implies that its smallest eigenvalue satisfies $\lambda_{\min }\left(\gamma_{n}\right)>0$ a.s.\ for all $n=2, \ldots, T/h$.

We next demonstrate that the numerical solution $Y_{T/h}$ is uniformly nondegenerate with respect to a sufficiently small stepsize $h > 0$, and as a consequence of \cite[Theorem 2.1.4]{NualartDavid2006}, it admits a smooth probability density function.

\begin{proposition} [Uniform non-degeneracy] \label{covarianmatrix}
Let Assumption \ref{ass.1} hold. Then for any $1 \leq p<\infty$, there exists a positive constant $\nu(p)$ such that
\begin{align*}
\|\operatorname{det}(\gamma_{T/h})^{-1}\|_{L^{p}(\Omega)} = \cO(h^{-\nu(p)}), \qquad \mbox{as } h \to 0. 
\end{align*}
\end{proposition}

\begin{proof}
In view of 
\begin{align} \label{detgamma}
\operatorname{det}(\gamma_{T/h})^{-1} \leq \lambda_{\min}(\gamma_{T/h})^{-(k + 2)}, 
\end{align}
It follows from \eqref{gamman1} that
\begin{align*}
\gamma_{T/h} &= A_{T/h - 1} \gamma_{T/h - 1} A_{T/h - 1}^{\top} + \gamma_{1} \\
&= \sum_{j = 0}^{T/h - 2} A_{T/h - 1} \cdots A_{j + 1}\left(\begin{array}{cc}
\sigma \sigma^{\top} & 0 \\
0 & 0
\end{array}\right) A_{j + 1}^{\top} \cdots A_{T/h - 1}^{\top} + \left(\begin{array}{cc}
\sigma \sigma^{\top} & 0 \\
0 & 0
\end{array}\right), 
\end{align*}
where $\sigma \sigma^{\top} = \gamma_{1}$. We rewrite 
\begin{align*}
A_{T/h - 1} = \frac{1}{\operatorname{det}(\frac{\partial G}{\partial \bar{Y}_{T_n, t_{n + 1}}})}
\left(\begin{array}{ll}
A_{T/h - 1}^{1, 1} &A_{T/h - 1}^{1, 2} \\
A_{T/h - 1}^{2, 1} &A_{T/h - 1}^{2, 2}
\end{array}\right), 
\end{align*}
where $A_{T/h - 1}^{1, 1}$ is a $(k + 1)\times (k + 1)$ matrix, $A_{T/h - 1}^{1, 2}$ is a $(k + 1)$-dimensional column vector, $A_{N^{h - 1}}^{2, 1}$ is a $(k + 1)$-dimensional row vector, and $A_{N^{h - 1}}^{2, 2}$ is a number. On the other hand, for any $y = (y_{1}^{\top}, y_{2}) \in \hR^{k + 2}$ with $\|y\| = 1$, where $y_{1}$ is a $(k + 1)$-dimensional column vector, and $y_{2}$ is a number, it holds that 
\begin{align*}
(y_{1}^{\top}, y_{2}) A_{T/h - 1}\left(\begin{array}{cc}
\sigma \sigma^{\top} & 0 \\
0 & 0
\end{array}\right) A_{T/h - 1}^{\top}\binom{y_{1}}{y_{2}} 
 = \Big(\frac{1}{\operatorname{det}(\frac{\partial G}{\partial \bar{Y}_{T_n, t_{n + 1}}})}\Big)^{2}\|y_{1}^{\top} A_{T/h - 1}^{1, 1} \sigma + y_{2}^{\top} A_{T/h - 1}^{2, 1} \sigma\|^{2}. 
\end{align*}
Combining the above equalities together, one gets 
\begin{align*}
\lambda_{\min}(\gamma_{T/h}) 
&= \min_{\substack{y = (y_1^\top, y_2^\top)^\top\in\hR^{k + 2} \\\|y\| = 1}}y^\top\gamma_{T/h}y \\
&\ge \min_{\substack{y = (y_1^\top, y_2^\top)^\top\in\hR^{k + 2} \\\|y\| = 1}}
\left(\begin{array}{c}y_1^\top, y_2^\top\end{array}
\right)
\left\{A_{T/h - 1}\left(\begin{array}{cc}\sigma\sigma^\top& 0 \\
0 & 0 \end{array}\right)A_{T/h - 1}^\top + \left(\begin{array}{cc}\sigma\sigma^\top& 0 \\
0 & 0 \end{array}\right)\right\}\left(\begin{array}{c}y_1 \\y_2\end{array}\right) \\
&= : \min_{\substack{y = (y_1^\top, y_2^\top)^\top\in\hR^{k + 2} \\
\|y\| = 1}} f(y), 
\end{align*}
where
\begin{align*}
f(y)
&= \Big(\frac{1}{\operatorname{det}(\frac{\partial G}{\partial \bar{Y}_{T_n, t_{n + 1}}})}\Big)^{2}\|y_{1}^{\top} A_{T/h - 1}^{1, 1} \sigma + y_{2}^{\top} A_{T/h - 1}^{2, 1} \sigma\|^{2} + \|y_1^\top\sigma\|^2 \\
&= \Big(\frac{1}{\operatorname{det}(\frac{\partial G}{\partial \bar{Y}_{T_n, t_{n + 1}}})}\Big)^{2}\Big[y_{1}^{\top} A_{T/h - 1}^{1, 1} \sigma \sigma^{\top}(A_{T/h - 1}^{1, 1})^{\top} y_{1} + y_{2} A_{T/h - 1}^{2, 1} \sigma \sigma^{\top}(A_{T/h - 1}^{2, 1})^{\top} y_{2} \\
&\quad + 2 y_{1}^{\top} A_{T/h - 1}^{1, 1} \sigma \sigma^{\top}(A_{T/h - 1}^{2, 1})^{\top} y_{2}\Big] + y_{1}^{\top} \sigma \sigma^{\top} y_{1}. 
\end{align*}
To proceed, we split the unit sphere of $\hR^{k + 2}$ into $\{\|y\| = 1, \|y_{1}\| \geq h\}$ and $\{\|y\| = 1, \|y_1\| < h\}$ being later determined. Then 
\begin{align} \label{mingamma}
\lambda_{\min}(\gamma_{T/h}) = \min_{\substack{y = (y_1^\top, y_2^\top)^\top\in\hR^{k + 2} \\\|y\| = 1}}\Big\{\min_{\|y_1\|\geq h} f(y), \min_{\|y_1\|< h} f(y)\Big\}. 
\end{align}
Here the lower bound of $\min_{\|y_1\|\geq h} f(y)$ follows from 
\begin{align} \label{minfy1}
\min_{\|y_{1}\| \geq h} f(y) \geq \min_{\|y_1\|\geq h} y_{1}^{\top} \sigma \sigma^{\top} y_{1} \geq \lambda_{\min}(\sigma \sigma^{\top}) h^{2}. 
\end{align}
To estimate the lower bound of the term $\min_{\|y_1\|< h} f(y)$, we set $h \leq 1$. It follows from $\|y_{1}\|^{2}<h^{2}$ and $\|y_{1}\|^{2} + \|y_{2}\|^{2} = 1$ that $\|y_{1}\|^{2}>h^{2}$. Using Young's inequality yields that for any $\varepsilon \in (0, 1)$, 
\begin{align}
f(y) 
&\geq (1 - \varepsilon)\Big(\frac{1}{\operatorname{det}(\frac{\partial G}{\partial \bar{Y}_{T_n, t_{n + 1}}})}\Big)^{2} y_{2} A_{T/h - 1}^{2, 1} \sigma \sigma^{\top}(A_{T/h - 1}^{2, 1})^{\top} y_{2} \notag \\
&\quad + (1 - \frac{1}{\varepsilon})\Big(\frac{1}{\operatorname{det}(\frac{\partial G}{\partial \bar{Y}_{T_n, t_{n + 1}}})}\Big)^{2} y_{1}^{\top} A_{T/h - 1}^{1, 1} \sigma \sigma^{\top}(A_{T/h - 1}^{11})^{\top} + y_{1}^{\top} \sigma \sigma^{\top} y_{1}. \label{fyieq}
\end{align}
According to \eqref{redetG}, \eqref{GYnorn}, \eqref{GYTnorm} and Corollary \ref{numerical:momentbound}, one obtains $\frac{1}{\operatorname{det}(\frac{\partial G}{\partial \bar{Y}_{T_n, t_{n + 1}}})} \leq 1 + h^{2}(\frac{\gamma^{2}}{8} + K)$ and $\|A_{T/h}^{1, 1}\| \leq 1 + C h$. Then by denoting $a(h) = [1 + h^{2}(\frac{\gamma^{2}}{8} + K)]^2 (1 + C h)^2$, 
\begin{align*}
\Big(\frac{1}{\operatorname{det}(\frac{\partial G}{\partial \bar{Y}_{T_n, t_{n + 1}}})}\Big)^{2}y_{1}^{\top} A_{T/h - 1}^{1, 1} \sigma \sigma^{\top}(A_{T/h - 1}^{1, 1})^{\top} y_{1}
&\leq \Big(\frac{1}{\operatorname{det}(\frac{\partial G}{\partial \bar{Y}_{T_n, t_{n + 1}}})}\Big)^{2} \lambda \max (\sigma \sigma^{\top})\|A_{T/h - 1}^{1, 1}\|^{2}\|y_{1}\|^{2} \\
&\leq \lambda \max (\sigma \sigma^{\top}) a(h)\|y_{1}\|^{2}, 
\end{align*}
which implies that for any $\varepsilon \in(0, 1)$, 
\begin{align*}
&\ (1 - \frac{1}{\varepsilon})\Big(\frac{1}{\operatorname{det}(\frac{\partial G}{\partial \bar{Y}_{T_n, t_{n + 1}}})}\Big)^{2} y_{1}^{\top} A_{T/h - 1}^{1, 1} \sigma \sigma^{\top}(A_{T/h - 1}^{1, 1})^{\top} y_{1} + y_{1}^{\top} \sigma \sigma^{\top} y_{1} \\
&\geq (1 - \frac{1}{\varepsilon}) \lambda_{\max}(\sigma \sigma^{\top}) a(h)\|y_{1}\|^{2} + \lambda_{\min}(\sigma \sigma^{\top})\|y_{1}\|^{2}.
\end{align*}
Based on the fact that $a(h)\rightarrow1$ as $h\rightarrow0$, there exists a sufficiently small $\hat{h} := \hat{h}(\gamma, K)\le 1$ such that $0<a(h)<2$ holds for all $h\le \hat{h}$. Moreover, $\frac{\lambda_{\max} \left(\sigma\sigma^\top\right)a(h)}{\lambda_{\max} \left(\sigma\sigma^\top\right)a(h) + \lambda_{\min}\left(\sigma\sigma^\top\right)}<\varepsilon<1$ holds for any $\varepsilon \in (\frac{2\lambda_{\max} \left(\sigma\sigma^\top\right)}{2\lambda_{\max} \left(\sigma\sigma^\top\right) + \lambda_{\min}\left(\sigma\sigma^\top\right)}, 1)$. Then one has
\begin{align} \label{Aieq}
(1 - \frac{1}{\varepsilon}) y_{1}^{\top} A_{T/h - 1}^{1, 1} \sigma \sigma^{\top}(A_{T/h - 1}^{1, 1})^{\top} y_{1} + y_{1}^{\top} \sigma \sigma y_{1} \geq 0. 
\end{align}
Next, taking 
\begin{align*}
\varepsilon_{0} = \frac{2 \lambda_{\max}(\sigma\sigma^{\top})}{2 \lambda_{\max}(\sigma\sigma^{\top}) + \lambda_{\min}(\sigma \sigma^{\top})}, \qquad 
h_{0} = \min \{h(\gamma, K), h^{*}, \tilde{h} \}, 
\end{align*}
and using \eqref{fyieq} and \eqref{Aieq} yield 
\begin{align*}
f(y) &\geq(1 - \varepsilon_{0}) y_{2} A_{T/h - 1}^{2, 1} \sigma \sigma^{\top}(A_{T/h - 1}^{2, 1})^{\top} y_{2}\Big(\frac{1}{\operatorname{det}(\frac{\partial G}{\partial \bar{Y}_{T_n, t_{n + 1}}})}\Big)^{2} \\
&\geq \frac{\lambda_{\min}(\sigma \sigma^{\top})}{2 \lambda_{\max}(\sigma \sigma^{\top}) + \lambda_{\min}(\sigma \sigma^{\top})} \lambda_{\min}(\sigma \sigma^{\top})\|A_{T/h - 1}^{2, 1} y_{2}\|^{2}\Big(\frac{1}{\operatorname{det}(\frac{\partial G}{\partial \bar{Y}_{T_n, t_{n + 1}}})}\Big)^{2}. 
\end{align*}
For convenience, we rewrite 
$(\frac{\partial G}{\partial \bar{Y}_{T_n, t_{n + 1}}})^{*} = \left( \begin{array}{ll}A^{1, 1} &A^{1, 2} \\ A^{2, 1} &A^{2, 2}\end{array} \right)$ and $
\frac{\partial G}{\partial Y_{n}} = \left( \begin{array}{ll}B^{1, 1} &B^{1, 2} \\
B^{2, 1} &B^{2, 2}
\end{array} \right)$. Here, $A^{1, 1}$ and $B^{1, 1}$ are $(k + 1) \times (k + 1)$-dimensional matrices, $A^{1, 2}$ and $B^{1, 2}$ are two $(k + 1)$-dimensional column vectors, $A^{2, 1} = h(\frac{1}{2}, \frac{h}{4} \lambda_{1}, \cdots, \frac{h}{4} \lambda_{2})$ and $B^{2, 1} = h(\frac{1}{2}, 0, \cdots 0)$ are two $(k + 1)$-dimensional row vectors, as well as $A^{2, 2}$ and $B^{2, 2}$ are two numbers. Then one gets 
\begin{align*}
\|A_{T/h - 1}^{2, 1}\| = e^{-\frac{\gamma}{2} h}\|A^{2, 1} B^{1, 1} + A^{2, 2} B^{2, 1}\|, 
\end{align*}
which together with Corollary \ref{numerical:momentbound} shows that $\|A_{T/h - 1}^{2, 1}\| \geq C h$ with $C>0$. Using \eqref{fyieq} and \eqref{Aieq} indicates that for any $\|y_{1}\| < h$ with $h \leq h_{0}$, 
\begin{align*}
f(y) 
\geq C h^{2}(1 - h^{2}) \left(1 + h^{2} \Big( \frac{1}{4} \sum_{\ell = 1}^{k} \lambda_{\ell}^{2} + F_{1} - \frac{\gamma^{2}}{16} \Big) \right)^{-2} \frac{(\lambda_{\min}(\sigma \sigma^{\top}))^{2}}{2 \lambda_{\max}(\sigma \sigma^{\top}) + \lambda_{\min}(\sigma \sigma^{\top})}. 
\end{align*}
Recalling \eqref{mingamma} and \eqref{minfy1} yields 
\begin{align*}
\lambda_{\min}(\gamma_{T/h}) \geq \min \Big\{\lambda_{\min}(\sigma \sigma^{\top}) h^{2}, \frac{C h^{2}(1 - h^{2}) (\lambda_{\min}(\sigma \sigma^{\top}))^{2}}{\left(1 + h^{2} \Big( \frac{1}{4} \sum_{\ell = 1}^{k} \lambda_{\ell}^{2} + F_{1} - \frac{\gamma^{2}}{16} \Big) \right)^{2} (2 \lambda_{\max}(\sigma \sigma^{\top}) + \lambda_{\min}(\sigma \sigma^{\top}))}\Big\}, 
\end{align*}
which implies 
\begin{align*}
\lambda_{\min}^{-1}(\gamma_{T/h}) \leq \max \Big\{\frac{1}{\lambda_{\min (\sigma \sigma^{\top}) h^{2}}}, \frac{(1 + h^{2}(\frac{1}{4} \sum_{\ell = 1}^{k} \lambda_{\ell}^{2} + F_{1} - \frac{\gamma^{2}}{16}))^{2}(2 \lambda_{\max}(\sigma \sigma^{\top}) + \lambda_{\min}(\sigma \sigma^{\top}))}{C h^{2}(1 - h^{2})(\lambda\min (\sigma \sigma^{\top}))^{2}}\Big\}. 
\end{align*}
Thanks to Corollary \ref{numerical:momentbound}, $ \hE [|F_{1}(x_{T/h - 1}, \bar{x}_{T_{T/h - 1}, t_{T/h}})|^{p}] \leq C(p, T)$ holds for any $p \geq 1$. Based on the facts that $2 - 2 e^{-\gamma h} = \cO(h^{-1})$ and $1 - e^{-2 \alpha_{\ell} h} = \cO(h^{-1})$ as $h \to 0$, one obtains that 
\begin{align*}
 \hE |\lambda_{\min}^{-1}(\gamma_{T/h})|^{p} \leq C(p) \max \{h^{-2 p}, h^{-2 p}\} h^{-p}, \qquad \mbox{for } p \geq 1. 
\end{align*}
Finally, using \eqref{detgamma}, one concludes that $\|\operatorname{det}(\gamma_{T/h})^{-1}\|_{L^p(\Omega)} = \cO(h^{-\nu(p)})$ with $\nu(p) \leq 3 (k+2)$ as $h \to 0$. The proof is completed. 
\end{proof}

The following corollary is a direct result of Lemma \ref{Numercal:Malliavinbound}, Proposition \ref{covarianmatrix}, and \cite[Theorem 2.1.4]{NualartDavid2006}.

\begin{corollary} [Existence and smoothness of density]
Under Assumption \ref{ass.1}, the numerical solution $Y_{T/h}$ admits a smooth probability density function. 
\end{corollary}

\subsection{Convergence rate of the probability density function}

This subsection is devoted to establishing an error estimate between the probability density function of the numerical solution and that of the exact solution. We begin by analyzing the error between their Malliavin derivatives.

\begin{theorem} \label{pdfconver}
Let $\alpha \in \mathbb{N}_{+}$ and $p \geq 1$. Under Assumption \ref{ass.1}, there exists $C = C(p, T, \alpha)$ and $h_0 \in (0 , h^*)$ such that for any $h \in(0, h_{0}]$, 
\begin{align} \label{Dconver}
\sup_{n \leq T/h}\|\mathcal{D}^{\alpha} Y_{n} - \mathcal{D}^{\alpha} Y(t_{n})\|_{L^{p}(\Omega, \, \mathbb{H}^{\otimes \alpha} \otimes \hR^{k + 2})} 
\leq C h.
\end{align}
\end{theorem}

\begin{proof}
Based on the Lyapunov inequality, it suffices to establish \eqref{Dconver} for all integers $\alpha \geq 1$ and $p \geq 2$. To this end, we use an induction argument on $\alpha$. For the case $\alpha = 1$ and $p \geq 2$, using H\"older's inequality indicates 
\begin{align*}
\|\mathcal{D} Y_{n} - \mathcal{D} Y(t_{n})\|_{L^{p}(\Omega, \, \mathbb{H}^{\otimes \alpha} \otimes \hR^{k + 2})}^{p} 
&= \hE \left[ \Big| \int_{0}^{T}\|\mathcal{D}_{r_1} Y_{n} - \mathcal{D}_{r_1} Y(t_{n})\|^{2} \rd r_{1}\Big|^{\frac{p}{2}} \right] \\
&\leq C (p, T) \int_{0}^{T} \hE \big[ \|\mathcal{D}_{r_1} Y_{n} - \mathcal{D}_{r_1} Y(t_{n})\|^{p} \big] \rd r_{1}. 
\end{align*}
Then, we need to show that for a.e.\ $r_{1} \in[0, T]$, 
\begin{align*}
\sup_{n \leq T/h} \hE \big[ \|\mathcal{D}_{r_1} Y_{n} - \mathcal{D}_{r_1} Y(t_{n})\|^{p} \big] \leq C h^{p}. 
\end{align*}
Fix $r_{1} \in (t_{i}, t_{i + 1}]$ for some integer $0 \leq i \leq T/h - 1$. Taking the Malliavin derivatives on both sides of \eqref{vvgle} yields that for $i < n \leq T/h - 1$, 
\begin{align*}
\mathcal{D}_{r_1} v_{n + 1} - \mathcal{D}_{r_1} v(t_{n + 1}) 
&= e^{-\frac{\gamma}{2} h}(\mathcal{D}_{r_1} v_{n} - \mathcal{D}_{r_1} v(t_{n})) + e^{-\frac{\gamma}{2} h} \int_{t_{n}}^{t_{n + 1}} \frac{\gamma}{2}(\mathcal{D}_{r_1} v(t_{n}) - \mathcal{D}_{r_1} v_{n}) \rd t \\
&\quad + e^{-\frac{\gamma}{2} h} \int_{t_{n}}^{t_{n + 1}} \int_{0}^{1} \int_{0}^{1} \nabla^{2} U(\hat{x}_{n}^{t})(\mathcal{D}_{r_1} x(t_{n}) - \mathcal{D}_{r_1} x_{n}) \rd \theta \rd \xi \rd t \\
&\quad - e^{-\frac{\gamma}{2} h} \sum_{\ell = 1}^{k} \lambda_{\ell} \int_{t_{n}}^{t_{n + 1}}(\mathcal{D}_{r_1} z_{\ell}(t_{n}) - \mathcal{D}_{r_{1}} z_{n}^{\ell}) \rd t + \mathcal{S}_{n}^{1}, 
\end{align*}
where $\hat{x}_{n}^{t} := \theta x(t) + (1 - \theta)(x_{n} + \xi(\bar{x}_{T_{n}, t_{n + 1}} - x_{n}))$ and 
\begin{align*}
&\mathcal{S}_{n}^{1} = \mathcal{S}_{n}^{1, 1} + \mathcal{S}_{n}^{1, 2} + \mathcal{S}_{n}^{1, 3} + \mathcal{S}_{n}^{1, 4} + \mathcal{S}_{n}^{1, 5}, \\
&\mathcal{S}_{n}^{1, 1} = \int_{t_{n}}^{t_{n + 1}} \big( - e^{\frac{\gamma}{2} h} + e^{-\frac{\gamma}{2}(t_{n + 1} - t)} \big) \big(\frac{\gamma}{2} \mathcal{D}_{r_1} v(t) + \nabla^{2} U(x(t)) \mathcal{D}_{r_{1}} x(t) - \sum_{\ell = 1}^{k} \lambda_{\ell} \mathcal{D}_{r_{1}} z_{\ell}(t) \big) \rd t, \\
&\mathcal{S}_{n}^{1, 2} = e^{-\frac{\gamma}{2} h} \int_{t_{n}}^{t_{n + 1}} \int_{0}^{1} \int_{0}^{1} \mathcal{D}_{r_1}[\nabla^{2} U(\hat{x}_{n}^{t})](x(t_{n}) - x_{n}) \rd \theta \rd \xi \rd t, \\
&\mathcal{S}_{n}^{1, 3} = e^{-\frac{\gamma}{2} h} \int_{t_{n}}^{t_{n + 1}} \int_{t_{n}}^{t} - \frac{\gamma}{2} \big( \gamma \mathcal{D}_{r_1} v(s) + \nabla^{2} U(x(s)) \mathcal{D}_{r_1} x(s) - \sum_{\ell = 1}^{k} \lambda_{\ell} \mathcal{D}_{r_1} z_{\ell}(s) \big) \rd s \rd t \\
&\qquad\quad - e^{-\frac{\gamma}{2} h} \frac{\gamma}{4} \int_{t_{n}}^{t_{n + 1}} \mathcal{D}_{r_{1}} \mathcal{A}^n \rd t, \\
&\mathcal{S}_{n}^{1, 4} = e^{-\frac{\gamma}{2} h} \int_{t_{n}}^{t_{n + 1}} \int_{0}^{1} \int_{0}^{1} \mathcal{D}_{r_1} \left[ \nabla^{2} U(\hat{x}_{n}^{t}) \Big( \int_{t_n}^{t} v(s) \rd s - \xi \mathcal{C}^{n} \Big) \right] \rd \theta \rd \xi \rd t, \\
&\mathcal{S}_{n}^{1, 5} = - e^{-\frac{\gamma}{2} h} \sum_{\ell = 1}^{k} \lambda_{\ell} \int_{t_n}^{t_{n + 1}} \int_{t_n}^{t} \big( \alpha_{\ell} \mathcal{D}_{r_1} z_{\ell}(s) + \lambda_{\ell} \mathcal{D}_{r_1} v(s) \big) \rd s \rd t - \frac{1}{2} e^{-\frac{\gamma}{2} h} \sum_{\ell = 1}^{k} \lambda_{\ell} \int_{t_{n}}^{t_{n + 1}} \mathcal{D}_{r_1} \mathcal{B}_{\ell}^{n} \rd t.
\end{align*}
Taking the Malliavin derivative on both sides of \eqref{zzgle} yields that for $i < n \leq T/h - 1$, 
\begin{align*}
\mathcal{D}_{r_{1}} z_{n + 1}^{\ell} - \mathcal{D}_{r_1} z_{\ell}(t_{n + 1}) 
&= e^{-\alpha_{\ell} h}(\mathcal{D}_{r_1} z_{n}^{\ell} - \mathcal{D}_{r_1} z_{\ell}(t_{n})) + e^{-\alpha_{\ell} h} \int_{t_{n}}^{t_{n + 1}} \lambda_{\ell}(\mathcal{D}_{r_1} v (t_{n}) - \mathcal{D}_{r_1} v_{n}) \rd t \\
&\quad + e^{-\alpha_{\ell} h} \int_{t_n}^{t_{n + 1}} \frac{\gamma}{2} \lambda_{\ell}(\mathcal{D}_{r_1} x(t_{n}) - \mathcal{D}_{r_1} x_{n}) \rd t + \mathcal{S}_{n, \ell}^{2}, 
\end{align*}
where
\begin{align*}
&\mathcal{S}_{n, \ell}^{2} = \mathcal{S}_{n, \ell}^{2, 1} + \mathcal{S}_{n, \ell}^{2, 2} + \mathcal{S}_{n, \ell}^{2, 3} + \mathcal{S}_{n, \ell}^{2, 4}, \\
&\mathcal{S}_{n, \ell}^{2, 1} = \int_{t_{n}}^{t_{n + 1}}( - e^{-\alpha_{\ell} h} + e^{-\alpha_{\ell}(t_{n + 1} - t)})(\lambda_{\ell} \mathcal{D}_{r_{1}} v(t) + \frac{\gamma}{2} \lambda_{\ell} \mathcal{D}_{r_1} x(t)) \rd t, \\
&\mathcal{S}_{n, \ell}^{2, 2} = - e^{-\alpha_{\ell} h} \lambda_{\ell} \int_{t_n}^{t_{n + 1}} \int_{t_n}^{t} \big( \gamma \mathcal{D}_{r_1} v(s) + \nabla^{2} U(x(s)) \mathcal{D}_{r_1} x(s) - \sum_{\ell = 1}^{k} \lambda_{\ell} \mathcal{D}_{r_1} z_{\ell}(s) \big) \rd s \rd t \\
&\qquad\quad - e^{-\alpha_{\ell} h} \frac{\lambda_{\ell}}{2} \int_{t_{n}}^{t_{n + 1}} \mathcal{D}_{r_{1}} \mathcal{A}^{n} \rd t, \\
&\mathcal{S}_{n, \ell}^{2, 3} = - e^{-\alpha_{\ell} h} \frac{\gamma \lambda_{\ell}}{2} \int_{t_n}^{t_{n + 1}} \int_{t_n}^{t} \mathcal{D}_{r_1} v(s) \rd s \rd t - e^{-\alpha_{\ell} h} \frac{\gamma \lambda_{\ell}}{4} \int_{t_{n}}^{t_{n + 1}} \mathcal{D}_{r_1} \mathcal{C}^{n} \rd t, \\
&\mathcal{S}_{n, \ell}^{2, 4} = \int_{t_{n}}^{t_{n + 1}} e^{-\alpha_{\ell}(t_{n + 1} - t)} \frac{\gamma\lambda_{\ell}}{2} \mathcal{D}_{r_{1}}(\tilde{x}_{T_{n}, t} - x(t)) \rd t.
\end{align*}
Taking the Malliavin derivatives on the both sides of \eqref{xxgle} yields that for $i < n \leq T/h - 1$, 
\begin{align*}
\mathcal{D}_{r_1} x_{n + 1} - \mathcal{D}_{r_1} x(t_{n + 1}) 
&= e^{-\frac{\gamma}{2} h}(\mathcal{D}_{r_1} x_{n} - \mathcal{D}_{r_1} x(t_{n})) - e^{-\frac{\gamma}{2} h} \int_{t_n}^{t_{n + 1}} \frac{\gamma}{2}(\mathcal{D}_{r_{1}} x(t_{n}) - \mathcal{D}_{r_{1}} x_{n}) \rd t \\
&\quad - e^{-\frac{\gamma}{2} h} \int_{t_n}^{t_{n + 1}}(\mathcal{D}_{r_1} v(t_{n}) - \mathcal{D}_{r_1} v_{n}) \rd t + \mathcal{S}_{n}^{3}, 
\end{align*}
where
\begin{align*}
&\mathcal{S}_{n}^{3} = \mathcal{S}_{n}^{3, 1} + \mathcal{S}_{n}^{3, 2} + \mathcal{S}_{n}^{3, 3}, \\
&\mathcal{S}_{n}^{3, 1} = - \int_{t_n}^{t_{n + 1}}( - e^{-\frac{\gamma}{2} h} + e^{-\frac{\gamma}{2}(t_{n + 1} - t)})(\frac{\gamma}{2} \mathcal{D}_{r_{1}} x(t) + \mathcal{D}_{r_{1}} v(t)) \rd t, \\
&\mathcal{S}_{n}^{3, 2} = e^{-\frac{\gamma}{2} h} \frac{\gamma}{2} \int_{t_{n}}^{t_{n + 1}} \int_{t_{n}}^{t} \mathcal{D}_{r_{1}} v(s) \rd s \rd t + e^{-\frac{\gamma}{2} h} \frac{\gamma}{4} \int_{t_{n}}^{t_{n + 1}} \mathcal{D}_{r_{1}}\mathcal{C}^{n} \rd t, \\
&\mathcal{S}_{n}^{3, 3} = e^{-\frac{\gamma}{2} h} \int_{t_{n}}^{t_{n + 1}} \int_{t_n}^{t}(\gamma \mathcal{D}_{r_1} v(s) + \nabla^{2} U(x(s)) \mathcal{D}_{r_1} x(s) - \sum_{\ell = 1}^{k} \lambda_{\ell} \mathcal{D}_{r_1} z_{\ell}(s)) \rd s \rd t + \frac{1}{2}e^{-\frac{\gamma}{2} h} \int_{t_{n}}^{t_{n + 1}} \mathcal{D}_{r_{1}} \mathcal{A}^{n} \rd t. 
\end{align*}
By the triangle inequality, one gets 
\begin{align*}
&\|\mathcal{D}_{r_1} v_{n + 1} - \mathcal{D}_{r_1} v(t_{n + 1})\| \\
&\leq \|\mathcal{D}_{r_1} v_{n} - \mathcal{D}_{r_1} v(t_{n})\| + C h\|\mathcal{D}_{r_{1}} v_{n} - \mathcal{D}_{r_{1}} v(t_{n})\| + C h\|\mathcal{D}_{r_{1}} z_{n}^{\ell} - \mathcal{D}_{r_{1}} z_{\ell}(t_{n})\| \\
&\quad + C \int_{t_n}^{t_{n+1}} \big( 1 + |x(t)|^{2 m-\epsilon} + |x_{n}|^{2 m-\epsilon} + |\bar{x}_{T_{n}, t_{n + 1}}|^{2 m-\epsilon} \big) \rd t \|\mathcal{D}_{r_{1}} x_{n} - \mathcal{D}_{r_{1}} x(t_{n})\| + \|\mathcal{S}_{n}^{1}\|, \\
&\|\mathcal{D}_{r_1} z_{n + 1}^{\ell} - \mathcal{D}_{r_1} z_{\ell}(t_{n + 1})\| \\
&\leq \|\mathcal{D}_{r_1} z_{n}^{\ell} - \mathcal{D}_{r_1} z_{\ell}(t_{n})\| + C h\|\mathcal{D}_{r_1} v_{n} - \mathcal{D}_{r_1} v(t_{n})\| + C h\|\mathcal{D}_{r_1} x_{n} - \mathcal{D}_{r_1} x(t_{n})\| + \|\mathcal{S}_{n, \ell}^{2}\|, \\
&\|\mathcal{D}_{r_{1}} x_{n + 1} - \mathcal{D}_{r_{1}} x(t_{n + 1})\| \\
&\leq \|\mathcal{D}_{r_1} x_{n} - \mathcal{D}_{r_1} x(t_{n})\| + C h\|\mathcal{D}_{r_1} x_{n} - \mathcal{D}_{r_1} x(t_{n})\| + C h\|\mathcal{D}_{r_1} v_{n} - \mathcal{D}_{r_1} v(t_{n})\| + \| \mathcal{S}_{n}^{3} \|. 
\end{align*}
For convenience, we denote
\begin{align*}
\mathcal{R}_{n + 1} := \|\mathcal{D}_{r_1} v_{n + 1} - \mathcal{D}_{r_1} v(t_{n + 1})\| + \sum_{\ell = 1}^{k}\|\mathcal{D}_{r_1} z_{n + 1}^{\ell} - \mathcal{D}_{r_1} z_{\ell}(t_{n + 1})\| + \|\mathcal{D}_{r_1} x_{n + 1} - \mathcal{D}_{r_1} x(t_{n + 1})\|. 
\end{align*}
Then it holds that 
\begin{align} \label{maliverror}
\mathcal{R}_{n + 1} 
\leq \mathcal{R}_{n} + C \mathcal{R}_{n} \int_{t_n}^{t_{n+1}} \big( 1 + |x(t)|^{2 m-\epsilon} + |x_{n}|^{2 m-\epsilon} + |\bar{x}_{T_{n}, t_{n + 1}}|^{2 m-\epsilon} \big) \rd t + \mathcal{S}_{n} 
\end{align}
with $\mathcal{S}_{n} = \| \mathcal{S}_{n}^{1} \| + \sum_{\ell = 1}^{k}\| \mathcal{S}_{n, \ell}^{2} \| + \| \mathcal{S}_{n}^{3} \|$. Recalling \eqref{xtidlex} and applying Gr\"onwall's inequality yield 
\begin{align*}
\|D_{n}(\tilde{x}_{T_{n}, t} - x(t))\| \leq C \Big( \|\mathcal{D}_{r_{1}}(x_{n} - x(t_{n}))\| + \int_{t_n}^{t}\|\mathcal{D}_{r_1}(\frac{\gamma}{2} x(s) + v(s))\| \rd s + \|\mathcal{D}_{r_{1}} \mathcal{C}^{n}\| \Big). 
\end{align*} 
Using H\"older's inequality, Propositions \ref{prop:GLE:well-posed}, \ref{Malliavindifferentiability} and \ref{numerical:momentbound} as well as Lemma \ref{Numercal:Malliavinbound} shows that for $\kappa = 1$, $\iota = 1, 2, 3, 4, 5$ as well as $\kappa = 3$, $\iota = 1, 2, 3$, 
\begin{align} \label{sn}
\|\mathcal{S}_{n}^{\kappa, \iota}\|_{L^{q}(\Omega)} \leq C(q) h^{2}, \qquad q \geq 1, \quad n = 0, 1, \ldots, T/h - 1. 
\end{align}
Moreover, for $\ell = 1, 2, \cdots, k$, $\kappa = 2$, $\iota = 1, 2, 3, 4$, 
\begin{align} \label{sln}
\| \mathcal{S}_{n, \ell}^{\kappa, \iota}\|_{L^q (\Omega)} \leq C(q) h^{2}, \qquad q \geq 1, \quad n = 0, 1, \ldots, T/h - 1. 
\end{align}
Consequently, 
\begin{align} \label{sumsn}
\hE \Big[\Big| \sum_{j = i + 1}^{n} \mathcal{S}_{j}\Big|^{q}\Big] \leq(n - i)^{q - 1} \sum_{j = i + 1}^{n} \hE [|\mathcal{S}_{j}|^{q}] \leq C(q) h^{q}, \qquad \forall\, q \geq 1. 
\end{align}
For $n = i$, the fact that $r_{1} \in(t_{i}, t_{i + 1}]$ yields $\mathcal{D}_{r_1} Y(t_{i}) = 0$ and $\mathcal{D}_{r_1}Y_{i} = 0$. It follows that for $r_{1} \in(t_{i}, t_{i + 1}]$, 
\begin{align*}
& \mathcal{D}_{r_1} v_{i + 1} - \mathcal{D}_{r_1} v(t_{i + 1}) 
= \mathcal{S}_{i}^{1} + e^{-\frac{\gamma}{2} h} \Big( \frac{\gamma}{2} \sqrt{2 \gamma} + \sum_{\ell = 1}^{k} \lambda_{\ell} \sqrt{2 \alpha_{\ell}} \Big) h, \\
& \mathcal{D}_{r_{1}} z_{i + 1}^{\ell} - \mathcal{D}_{r_{1}} z_{\ell}(t_{i + 1}) 
= \mathcal{S}_{i, \ell}^{2} + e^{-\alpha_{\ell} h} \lambda_{\ell} \sqrt{2 \gamma} h, \\
& \mathcal{D}_{r_1} x_{i + 1} - \mathcal{D}_{r_1} x(t_{i + 1}) 
= \mathcal{S}_{i}^{3} - e^{-\frac{\gamma}{2} h} \sqrt{2 \gamma} h. 
\end{align*}
In view of \eqref{sn}, \eqref{sln}, and the inequality $|t_{i + 1} - r_{1}|<h$, one gets 
\begin{align} \label{sumr}
\hE [\mathcal{R}_{i + 1}^{q}] \leq C(q) h^{q}, \qquad \forall\, q \geq 1. 
\end{align}
By the discrete Gr\"onwall inequality together with \eqref{maliverror}, one obtains that for all $n = 0, \ldots, T/h - 1$, 
\begin{align*}
\mathcal{R}_{n + 1}^{p} 
\leq C \Big[ \Big( \sum_{j = i + 1}^{n} \mathcal{S}_{j} \Big)^{p} + \mathcal{R}_{i + 1}^{p} \Big] \exp \Big( \sum_{j = i + 1}^{n} C p \int_{t_j}^{t_{j+1}} 1 + |x(t)|^{2 m-\epsilon} + |x_{j}|^{2 m-\epsilon} + |\bar{x}_{T_{j}, t_{j + 1}}|^{2 m-\epsilon} \rd t \Big). 
\end{align*}
Thus, the proof of \eqref{Dconver} for the case $\alpha = 1$ follows from Proposition \ref{Numerical:expo}, Corollary \ref{numerical:momentbound}, \eqref{exactexpo}, \eqref{sumsn}, \eqref{sumr}, and H\"older's inequality.

Let $\alpha \geq 2$ and assume that \eqref{Dconver} holds for all integers $\alpha^* \in[1, \alpha - 1]$, that is, 
\begin{align} \label{assump}
\sup_{n \leq T/h}\|\mathcal{D}^{\alpha^*} Y_n - \mathcal{D}^{\alpha^*} Y(t_n)\|_{L^p(\Omega, \, \mathbb{H}^{\otimes \alpha^*} \otimes \hR^{k+2})} \leq C (p, T, \alpha^*) h, \qquad \forall\, p \geq 1. 
\end{align}
It remains to show that \eqref{Dconver} holds for all $\alpha \geq 2$ and all $p \geq 1$. For $\alpha \geq 2$, by computing the $\alpha$-th $(\alpha \geq 2)$ Malliavin derivative on both sides of \eqref{vvgle} and using the multivariate chain rule for Malliavin calculus, one gets 
\begin{align*}
\mathcal{D}^{\alpha} v_{n + 1} - \mathcal{D}^{\alpha} v(t_{n + 1}) 
&= e^{-\frac{\gamma}{2} h}(\mathcal{D}^{\alpha} v_{n} - \mathcal{D}^{\alpha} v(t_{n})) + e^{-\frac{\gamma}{2} h} \int_{t_{n}}^{t_{n + 1}} \frac{\gamma}{2}(\mathcal{D}^{\alpha} v(t_{n}) - \mathcal{D}^{\alpha} v_{n}) \rd t \\
&\quad + e^{-\frac{\gamma}{2} h} \int_{t_{n}}^{t_{n + 1}} \int_{0}^{1} \int_{0}^{1} \nabla^{2} U(\hat{x}_{n}^{t})(\mathcal{D}^{\alpha} x(t_{n}) - \mathcal{D}^{\alpha} x_{n}) \rd \theta \rd \xi \rd t \\
&\quad - e^{-\frac{\gamma}{2} h} \sum_{\ell = 1}^{k} \lambda_{\ell} \int_{t_{n}}^{t_{n + 1}}(\mathcal{D}^{\alpha} z_{\ell}(t_{n}) - \mathcal{D}^{\alpha} z_{n}^{\ell}) \rd t + \mathcal{S}_{n, \alpha}^1, 
\end{align*}
where
\begin{align*}
\mathcal{S}_{n, \alpha}^1
&= \mathcal{S}_{n, \alpha}^{1, 1} + \mathcal{S}_{n, \alpha}^{1, 2} + \mathcal{S}_{n, \alpha}^{1, 3} + \mathcal{S}_{n, \alpha}^{1, 4} + \mathcal{S}_{n, \alpha}^{1, 5}, \\
\mathcal{S}_{n, \alpha}^{1, 1}
&= \int_{t_n}^{t_{n + 1}} \big( e^{-\frac{\gamma}{2} h} + e^{-\frac{\gamma}{2}(t_{n + 1} - t)} \big) \big(\frac{\gamma}{2} \mathcal{D}^{\alpha} v(t) + \mathcal{D}^{\alpha}[\nabla U(x(t))] - \sum_{\ell = 1}^{k} \lambda_{\ell} \mathcal{D}^{\alpha} z_{\ell}(t) \big) \rd t, \\
\mathcal{S}_{n, \alpha}^{1, 2} 
&= e^{-\frac{\gamma}{2} h} \int_{t_{n}}^{t_{n + 1}} \int_{0}^{1} \int_{0}^{1} \mathcal{D}^{\alpha}[\nabla^{2} U(\hat{x}_{n}^{t})(x(t_{n}) - x_{n}) ] \rd \theta \rd \xi \rd t \\
&\quad - e^{-\frac{\gamma}{2} h} \int_{t_{n}}^{t_{n + 1}} \int_{0}^{1} \int_{0}^{1} \nabla^{2} U(\hat{x}_{n}^{t})(\mathcal{D}^{\alpha} x(t_{n}) - \mathcal{D}^{\alpha} x_{n}) \rd \theta \rd \xi \rd t, \\
\mathcal{S}_{n, \alpha}^{1, 3}
&= - e^{-\frac{\gamma}{2} h} \frac{\gamma}{2} \int_{t_{n}}^{t_{n + 1}} \int_{t_{n}}^{t} \big( \gamma \mathcal{D}^{\alpha} v(s) + \mathcal{D}^{\alpha}[\nabla U(x(s))] - \sum_{\ell = 1}^{k} \lambda_{\ell} \mathcal{D}^{\alpha} z_{\ell}(s) \big) \rd s \rd t - e^{-\frac{\gamma}{2} h} \frac{\gamma}{4} \int_{t_{n}}^{t_{n + 1}} \mathcal{D}^{\alpha} \mathcal{A}^{n} \rd t, \\
\mathcal{S}_{n, \alpha}^{1, 4}
&= e^{-\frac{\gamma}{2} h} \int_{t_{n}}^{t_{n + 1}} \int_{0}^{1} \int_{0}^{1} \mathcal{D}^{\alpha} \big[ \nabla^{2} U(\hat{x}_{n}^{t}) \big( \int_{t_{n}}^{t} v(s) \rd s - \xi \mathcal{C}^{n} \big) \big] \rd \theta \rd \xi \rd t, \\
\mathcal{S}_{n, \alpha}^{1, 5}
&= - e^{-\frac{\gamma}{2} h} \sum_{\ell = 1}^{k} \lambda_{\ell} \int_{t_{n}}^{t_{n + 1}} \int_{t_{n}}^{t} \big( \alpha_{\ell} \mathcal{D}^{\alpha} z_{\ell}(s) + \lambda_{\ell} \mathcal{D}^{\alpha} v(s) \big) \rd s \rd t - \frac{1}{2} e^{-\frac{\gamma h}{2}} \sum_{\ell = 1}^{k} \lambda_{\ell} \int_{t_n}^{t_{n + 1}} \mathcal{D}^{\alpha} \mathcal{B}_{\ell}^{n} \rd t.
\end{align*}
Taking the $\alpha$-th $(\alpha \geq 2)$ Malliavin derivative on both sides of \eqref{zzgle} and using the chain rule yield
\begin{align*}
\mathcal{D}^{\alpha} z_{n + 1}^{\ell} - \mathcal{D}^{\alpha} z_{\ell}(t_{n + 1}) 
&= e^{-\alpha_{\ell} h}(\mathcal{D}^{\alpha} z_{n}^{\ell} - \mathcal{D}^{\alpha} z_{\ell}(t_{n})) + e^{-\alpha_{\ell} h} \int_{t_{n}}^{t_{n + 1}} \lambda_{\ell}(\mathcal{D}^{\alpha} v(t_{n}) - \mathcal{D}^{\alpha} v_{n}) \rd t \\
&\quad + e^{-\alpha_{\ell} h} \int_{t_n}^{t_{n + 1}} \frac{\gamma}{2} \lambda_{\ell}(\mathcal{D}^{\alpha} x(t_{n}) - \mathcal{D}^{\alpha} x_{n}) \rd t + \mathcal{S}_{n, \ell, \alpha}^{2}, 
\end{align*}
where
\begin{align*}
\mathcal{S}_{n, \ell, \alpha}^{2} 
&= \mathcal{S}_{n, \ell, \alpha}^{2, 1} + \mathcal{S}_{n, \ell, \alpha}^{2, 2} + \mathcal{S}_{n, \ell, \alpha}^{2, 3} + \mathcal{S}_{n, \ell, \alpha}^{2, 4}, \\
\mathcal{S}_{n, \ell, \alpha}^{2, 1}
&= \int_{t_{n}}^{t_{n + 1}} \big( - e^{-\alpha_{\ell} h} + e^{-\alpha_{\ell}(t_{n + 1} - t)} \big) \big( \lambda_{\ell} \mathcal{D}^{\alpha} v(t) + \frac{\gamma}{2} \lambda_{\ell} \mathcal{D}^{\alpha} x(t) \big) \rd t, \\
\mathcal{S}_{n, \ell, \alpha}^{2, 2}
&= - e^{-\alpha_{\ell} h} \lambda_{\ell} \int_{t_{n}}^{t_{n + 1}} \int_{t_{n}}^{t} \big( \gamma \mathcal{D}^{\alpha} v(s) + \mathcal{D}^{\alpha}[\nabla U(x(s))] - \sum_{\ell = 1}^{k} \lambda_{\ell} \mathcal{D}^{\alpha} z_{\ell}(s) \big) \rd s \rd t - e^{-\alpha_{\ell} h} \frac{\lambda_{\ell}}{2} \int_{t_{n}}^{t_{n + 1}} \mathcal{D}^{\alpha} \mathcal{A}^{n} \rd t, \\
\mathcal{S}_{n, \ell, \alpha}^{2, 3}
&= - e^{-\alpha_\ell h} \frac{\gamma \lambda_{\ell}}{2} \int_{t_{n}}^{t_{n + 1}} \int_{t_{n}}^{t} \mathcal{D}^{\alpha} v(s) \rd s \rd t - e^{-\alpha_{\ell} h} \frac{\gamma \lambda_{\ell}}{4} \int_{t_{n}}^{t_{n + 1}} \mathcal{D}^{\alpha} \mathcal{C}^{n} \rd t, \\
\mathcal{S}_{n, \ell, \alpha}^{2, 4}
&= \int_{t_{n}}^{t_{n + 1}} e^{-\alpha_{\ell}(t_{n + 1} - t)} \frac{\gamma \lambda_{\ell}}{2} \mathcal{D}^{\alpha}(\tilde{x}_{T_{n}, t} - x(t)) \rd t.
\end{align*}
Taking the $\alpha$-th $(\alpha \geq 2)$ Malliavin derivative on both sides of \eqref{xxgle} and applying the chain rule specific to Malliavin calculus yield 
\begin{align*}
\mathcal{D}^{\alpha} x_{n + 1} - \mathcal{D}^{\alpha} x(t_{n + 1}) 
&= e^{-\frac{\gamma}{2} h}(\mathcal{D}^{\alpha} x_{n} - \mathcal{D}^{\alpha} x(t_{n})) - e^{-\frac{\alpha}{2} h} \int_{t_n}^{t_{n + 1}} \frac{\gamma}{2}(\mathcal{D}^{\alpha} x(t_{n}) - \mathcal{D}^{\alpha} x_{n}) \rd t \\
&\quad - e^{-\frac{\gamma}{2} h} \int_{t_n}^{t_{n + 1}}(\mathcal{D}^{\alpha} v(t_{n}) - \mathcal{D}^{\alpha} v_{n}) \rd t + \mathcal{S}_{n, \alpha}^3, 
\end{align*}
where
\begin{align*}
\mathcal{S}_{n, \alpha}^{3}
&= \mathcal{S}_{n, \alpha}^{3, 1} + \mathcal{S}_{n, \alpha}^{3, 2} + \mathcal{S}_{n, \alpha}^{3, 3}, \\
\mathcal{S}_{n, \alpha}^{3, 1}
&= - \int_{t_{n}}^{t_{n + 1}}( - e^{-\frac{\gamma}{2} h} + e^{-\frac{\gamma}{2}(t_{n + 1} - t)})(\frac{\gamma}{2} \mathcal{D}^{\alpha} x(t) + \mathcal{D}^{\alpha} v(t)) \rd t, \\
\mathcal{S}_{n, \alpha}^{3, 2}
&= e^{-\frac{\gamma}{2} h} \frac{\gamma}{2} \int_{t_{n}}^{t_{n + 1}} \int_{t_{n}}^{t} \mathcal{D}^{\alpha} v(s) \rd s \rd t + e^{-\frac{\gamma}{2} h} \frac{\gamma}{4} \int_{t_{n}}^{t_{n + 1}} \mathcal{D}^{\alpha} \mathcal{C}^{n} \rd t, \\
\mathcal{S}_{n, \alpha}^{3, 3}
&= e^{-\frac{\gamma}{2} h} \int_{t_{n}}^{t_{n + 1}} \int_{t_{n}}^{t}(\gamma \mathcal{D}^{\alpha} v(s) + \mathcal{D}^{\alpha}[\nabla U(x(s))] - \sum_{\ell = 1}^{k} \lambda_{\ell} \mathcal{D}^{\alpha} z_{\ell}(s)) \rd s \rd t + \frac{1}{2}e^{-\frac{\gamma}{2} h} \int_{t_{n}}^{t_{n + 1}} \mathcal{D}^{\alpha} \mathcal{A}^{n} \rd t. 
\end{align*}
Let
\begin{align*}
\mathcal{T}_{n, \alpha} := \|\mathcal{D}^{\alpha} v_{n} - \mathcal{D}^{\alpha} v(t_{n})\|_{\mathbb{H}^{\otimes \alpha}} + \sum_{\ell = 1}^{k}\|\mathcal{D}^{\alpha} z_{n}^{\ell} - \mathcal{D}^{\alpha} z_{\ell}(t_{n})\|_{\mathbb{H}^{\otimes \alpha}} + \|\mathcal{D}^{\alpha} x_{n} - \mathcal{D}^{\alpha} x(t_{n})\|_{\mathbb{H}^{\otimes \alpha}}. 
\end{align*}
By the initial condition $\mathcal{T}_{0, \alpha} = 0$ for $\alpha \geq 2$, one has 
\begin{align}
\mathcal{T}_{n + 1, \alpha} 
&\leq \mathcal{T}_{n, \alpha} + C \mathcal{T}_{n, \alpha} \int_{t_n}^{t_{n+1}} 1 + |x(t)|^{2 m-\epsilon} + |x_{n}|^{2 m - \epsilon} + |\bar{x}_{T_{n}, t_{n + 1}}|^{2 m - \epsilon} \rd t + \|\mathcal{S}_{n, \alpha}\|_{\mathbb{H}^{\otimes \alpha}} \notag \\
&\leq \exp \bigg( \sum_{j = 0}^{n} C \int_{t_j}^{t_{j+1}} 1 + |x(t)|^{2 m-\epsilon} + |x_{j}|^{2 m - \epsilon} + |\bar{x}_{T_{j}, t_{j + 1}}|^{2 m - \epsilon} \rd t \bigg) \bigg( \sum_{j = 0}^{n}\|\mathcal{S}_{j, \alpha}\|_{\mathbb{H}^{\otimes \alpha}} \bigg). \label{Tieq}
\end{align}
By H\"older's inequality, Proposition \ref{Malliavindifferentiability}, Lemma \ref{Numercal:Malliavinbound}, \cite[Lemma 6.1]{CuiHongSheng2022}, and the assumption $U \in C_{p}^{\infty}(\hR)$, it follows that for any $\eta \in \mathbb{N}_{+}$ and $q \geq 1$, there exists $C = C (\eta, q)$ such that for any $t \in[0, T]$ and $n \in\{0, 1, \cdots, T/h\}$, 
\begin{align*}
& \hE \Big[\|\mathcal{D}^{\eta}[\nabla U(x(t))]\|_{\mathbb{H}^{\otimes \eta}}^{q}\Big] + \hE \Big[\|\mathcal{D}^{\eta} v(t)\|_{\mathbb{H}^{\otimes \eta}}^{q}\Big] + \hE \Big[\|\mathcal{D}^{\eta} v_{n}\|_{\mathbb{H}^{\otimes \eta}}^{q}\Big] + \hE \Big[\|\mathcal{D}^{\eta} \bar{v}_{T_{n}, t_{n + 1}}\|_{\| \mathbb{H}^{\otimes \eta}}^{q}\Big] \\
&\qquad + \sum_{\ell = 1}^{k} \hE \Big[\|\mathcal{D}^{\eta} z_{\ell}(t)\|_{\mathbb{H}^{\otimes \eta}}^{q}\Big] + \sum_{\ell = 1}^{k} \hE \Big[\|\mathcal{D}^{\eta} z_{n}^{\ell}\|_{\mathbb{H}^{\otimes \eta}}^{q}\Big] + \sum_{\ell = 1}^{k} \hE \Big[\|\mathcal{D}^{\eta} \bar{z}_{T_{n}, t_{n + 1}}^{\ell}\|_{\mathbb{H}^{\otimes \eta}}^{q}\Big] \\
&\qquad + \hE \Big[\|\mathcal{D}^{\eta} x(t)\|_{\mathbb{H}^{\otimes \eta}}^{q}\Big] + \hE \Big[\|\mathcal{D}^{\eta} x_{n}\|_{\mathbb{H}^{\otimes \eta}}^{q}\Big] + \hE \Big[\|\mathcal{D}^{\eta} \bar{x}_{T_{n}, t_{n + 1}}\|_{\mathbb{H}^{\otimes \eta}}^{q}\Big] 
\leq C 
\end{align*}
and 
\begin{align} \label{boundU}
 \hE \Big[\|\mathcal{D}^{\eta}[\nabla U(x_{n} + \xi(\bar{x}_{T_{n}, t_{n} + 1} - x_{n}))]\|_{\mathbb{H}^{\otimes \eta}}^{q}\Big] + \hE \Big[\|\mathcal{D}^{\eta}[\nabla^{2} U(\hat{x}_{n}^{t})]\|_{\mathbb{H}^{\otimes \eta}}^{q}\Big] 
\leq C. 
\end{align}
Consequently, one obtains that for $q \geq 1$, 
\begin{align*}
&\ \sup_{n \leq T/h} \hE \Big[\|\mathcal{S}_{n, \alpha}^{1, 1}\|_{\mathbb{H}^{\otimes \alpha}} + \|\mathcal{S}_{n, \alpha}^{1, 3}\|_{\mathbb{H}^{\otimes \alpha}} + \|\mathcal{S}_{n, \alpha}^{1, 4}\|_{\mathbb{H}^{\otimes \alpha}} + \|\mathcal{S}_{n, \alpha}^{1, 5}\|_{\mathbb{H}^{\otimes \alpha}} + \sum_{\ell = 1}^{k}\| \mathcal{S}_{n, \ell, \alpha}^{2} \|_{\mathbb{H}^{\otimes \alpha}} + \|\mathcal{S}_{n, \alpha}^{3}\|_{\mathbb{H}^{\otimes \alpha}}\Big] \\
&\leq C (q, \alpha) h^{2 q}. 
\end{align*}
It remains to estimate $\mathcal{S}_{n, \alpha}^{1, 2}$. Using \cite[(6.3)]{CuiHongSheng2022}, one obtains
\begin{align}
&\ \mathcal{D}^{\alpha}[\nabla^{2} U(\hat{x}_{n}^{t})(x(t_{n}) - x_{n})] - \nabla^{2} U(\hat{x}_{n}^{t})(\mathcal{D}^{\alpha} x(t_{n}) - \mathcal{D}^{\alpha} x_{n})\notag \\
&= \sum_{\eta = 1}^{\alpha}\binom{\alpha}{\eta} \mathcal{D}^{\eta}[\nabla^{2} U(\hat{x}_{n}^{t})] \,\widetilde{\otimes}\, \mathcal{D}^{\alpha - \eta}[x(t_{n}) - x_{n}]. \label{tensor}
\end{align}
By the induction hypothesis \eqref{assump}, for any integer $1 \leq \eta \leq \alpha$, one has $0 \leq \alpha - \eta \leq \alpha - 1$ and
\begin{align*}
\sup_{n \leq T/h} \hE \big[\|\mathcal{D}^{\alpha - \eta}[x(t_{n}) - x_{n}]\|_{\mathbb{H}^{\otimes (\alpha - \eta)}}^{q} \big] \leq C(\alpha, q, T) h^{q}, 
\end{align*}
which combined with \cite[Lemma 6.1]{CuiHongSheng2022}, \eqref{boundU} and \eqref{tensor} implies 
\begin{align*}
\sup_{n \leq T/h} \hE [\|\mathcal{S}_{n, \alpha}^{1, 2}\|_{\mathbb{H}^{\otimes \alpha}}^{q} ] \leq C(q, \alpha, T) h^{2 q}. 
\end{align*}
Substituting the preceding estimates for $\{\mathcal{S}_{n, \alpha}^{1, i}\}_{i = 1, 2, 3, 4, 5}$, $\{\mathcal{S}_{n, \ell, \alpha}^{2, i}\}_{\ell = 1, \cdots, k, \, i = 1, 2, 3, 4}$ and $\{\mathcal{S}_{n, \alpha}^{3, i}\}_{i = 1, 2, 3}$ into \eqref{Tieq} yields that for any $p \geq 1$, 
\begin{align*}
&\|\mathcal{T}_{n + 1, \alpha}\|_{L^{p}(\Omega)} \\
&\leq \left\|\exp \bigg( \sum_{j = 1}^{n} C \int_{t_j}^{t_{j+1}} 1 + |x(t)|^{2 m-\epsilon} + |x_{j}|^{2 m - \epsilon} + |\bar{x}_{T_{j}, t_{j + 1}}|^{2 m - \epsilon} \rd t \bigg) \right\|_{L^{2 p}(\Omega)} \bigg( \sum_{j = 0}^{n}\|\mathcal{S}_{j, \alpha}\|_{\mathbb{H}^{\otimes \alpha}} \bigg) \\
&\leq C(p) h. 
\end{align*}
Finally, in light of the inequality
\begin{align*}
\|\mathcal{D}^{\alpha} Y_{n} - \mathcal{D}^{\alpha} Y(t_{n})\|_{\mathbb{H}^{\otimes \alpha} \otimes \hR^{k+2}} \leq C \mathcal{T}_{n, \alpha}, 
\end{align*}
the proof is completed.
\end{proof}

The convergence of the Malliavin derivatives of the numerical solutions, as established in Theorem \ref{pdfconver}, along with the uniform non-degeneracy result from Proposition \ref{covarianmatrix}, jointly implies the convergence of the numerical probability density functions, as demonstrated by \cite[Proposition 2.3]{CuiHongSheng2022}, leading to the following corollary.

\begin{corollary} [Convergence rate of density]
Let $\eta\ge0$, $1<q<\infty$, $\alpha>\eta + (k + 2)/q + 1$, $G\in \mathbb{D}^{\alpha, q}$, and $1/p + 1/q = 1$. If Assumption \ref{ass.1} holds, then $(1-\Delta)^{\eta / 2} \delta_y \circ Y_{T/h} \in \mathbb{D}^{-\alpha, p}$, $(1-\Delta)^{\eta / 2} \delta_y \circ Y(T) \in \mathbb{D}^{-\alpha, p}$, and 
\begin{align*}
\sup_{y\in\hR^{k + 2}} \left| (1 - \Delta)^{\eta/2} \hE \big[ G \cdot \delta_y \circ Y_{T/h} \big] - (1 - \Delta)^{\eta/2} \hE \big[ G \cdot \delta_y\circ Y(T) \big] \right| 
\leq Ch. 
\end{align*}
In particular, taking $G = 1$ yields 
\begin{align*}
\sup_{y\in\hR^{k + 2}}\left|(1 - \Delta)^{\eta/2}p^{T/h}_T(Y_0, y) - (1 - \Delta)^{\eta/2}p_T(Y(0), y)\right| \leq Ch, 
\end{align*}
where $p^{T/h}_T(Y_0, y) = \hE \left[\delta_y\circ Y_{T/h}\right]$ and $p_T(Y(0), y) = \hE \left[\delta_y\circ Y(T)\right]$ denote the probability density functions of $Y_{T/h}$ and $Y(T)$ with respect to $y$, respectively. 
\end{corollary}

\section{Ergodicity of the numerical approximation}
\label{sec.Ergodicity}
In this section, we prove the geometric ergodicity of the splitting AVF method \eqref{numsolu}--\eqref{sdesolu}. As $T \to \infty$, we aim to demonstrate that the stationary distribution of the numerical approximation converges exponentially fast to the unique invariant measure. According to \cite[Theorem 2.5]{MattinglyStuartHigham2002}, it suffices to verify the Lyapunov and minorization conditions to ensure the existence and uniqueness of an invariant measure. In the following lemma, we establish the existence of the invariant measure by verifying the Lyapunov condition.

\begin{lemma}[Lyapunov condition]
\label{Numerical:Lyapunov}
Let $\{Y_{n}\}_{n\in \mathbb{N}}$ be the numerical solution defined by \eqref{numsolu}--\eqref{sdesolu} with $h \in(0, h^{*})$. Under Assumption \ref{ass.1}, there exist constants $\alpha \in(0, 1)$ and $\beta \in(0, \infty)$ such that the Lyapunov function $H$ defined in \eqref{newHamil} satisfies
\begin{align*}
 \hE \big[ H(Y_{n + 1}) + C_{H} \mid \mathcal{F}_{t_n} \big] \leq \alpha \big( H(Y_{n}) + C_{H} \big) + \beta, \qquad \forall\, n \geq 1, 
\end{align*}
and hence $\{Y_{n}\}_{n\in \mathbb{N}}$ admits an invariant measure.
\end{lemma}

\begin{proof}
From the proof of Proposition \ref{Numerical:expo}, it follows that
\begin{align*}
D H(Y) \tilde{\mu}(Y) 
\leq -\frac{\underline{\alpha}_\gamma}{2}H(Y) + C, 
\end{align*}
where $\underline{\alpha}_{\gamma} := \gamma \wedge \frac{\gamma C_4}{4 C_2} \wedge \alpha_{1} \wedge \cdots \wedge \alpha_{k}$. Then, it follows from It\^o's formula that 
\begin{align*}
\rd H(\tilde{Y}_{T_n , t}) = D H(\tilde{Y}_{T_n , t}) \tilde{\mu}(\tilde{Y}_{T_n , t}) \rd t + D H(\tilde{Y}_{T_n , t})^{\top} \tilde{\sigma} \rd \tilde{W}(t), 
\end{align*}
where $\tilde{W} = (\, W_0,\, W_1,\, \cdots,\, W_k, \, 0 \,)^{\top}$. Note that $\bar{Y}_{T_{n}, t_{n + 1}}$ is $\mathcal{F}_{t_{n}}$-measurable and independent of $ \int_{t_n}^{t_{n + 1}} D H(\tilde{Y}_{T_n , t})^{\top} \tilde{\sigma} \rd \tilde{W}(t)$. Thus, applying Gr\"onwall's inequality shows 
\begin{align*}
 \hE [H(Y_{n + 1}) \mid \mathcal{F}_{t_{n}}] \leq e^{-\frac{\underline{\alpha}_\gamma}{2} h} H(\bar{Y}_{T_{n}, t_{n + 1}}) + \frac{2C}{\underline{\alpha}_\gamma} (1-e^{-\frac{\underline{\alpha}_\gamma}{2} h}) = \alpha H(Y_{n}) + \beta, 
\end{align*}
where $\alpha = e^{-\frac{\underline{\alpha}_\gamma}{2} h} \in (0, 1)$ and $\beta = \frac{2C}{\underline{\alpha}_\gamma} (1-e^{-\frac{\underline{\alpha}_\gamma}{2} h}) \in(0, \infty)$. In the last step, we also use the fact that the subsystem \eqref{avfmethod} preserves the Hamiltonian $H$. The proof is complete. 
\end{proof}

\begin{lemma}[Minorization condition]
\label{Numerical:minorization}
Fix any $y_{\max} >0$. Under Assumption \ref{ass.1}, there exists $\bar{h} \in(0, h^{*})$ such that for any $h \in (0, \bar{h})$, there is $\alpha>0$ such that for all $\varphi \in C_{0}(\hR^{k+2}, \hR)$
\begin{align*}
\inf _{\|y\| \leqslant y_{\max }}(\mathcal{P}_{h}^{2} \varphi)(y) \geqslant \alpha \int_{\hR^{k+2}} \varphi(y) \nu(d x), \qquad y \in \hR^{k+2},
\end{align*}
where $\mathcal{P}_{h}^{2}$ is the transition kernel of the two-step iteration and $\nu$ denotes the Lebesgue measure on $\mathscr{B}(\hR^{k+2})$.
\end{lemma}

\begin{proof}
In view of \eqref{demapphi}, let
\begin{align*}
\Psi_{h}(y, G_{1}, G_{2}) = Y_2 = M \phi_{h}\left(M \phi_{h}(y)+\binom{G_{1}}{0}\right)+\binom{G_{2}}{0},
\end{align*}
denote the stochastic flow map of the twice iterated splitting AVF method \eqref{numsolu}--\eqref{sdesolu}. Here, $M$ is defined in \eqref{Mmatrix}, $\phi_h$ is given by \eqref{demapphi}, with
\begin{align*}
G_1 := \Big(\int_{t_{0}}^{t_{1}} e^{-\frac{\gamma}{2}(t_{1} - t)} \sqrt{2 \gamma} \rd W_{0}(t), \int_{t_{0}}^{t_{1}} e^{-\alpha_{1}(t_{1} - t)} \sqrt{2 \alpha_{1}} \rd W_{1}(t),\cdots,\int_{t_{0}}^{t_{1}} e^{-\alpha_{k}(t_{1} - t)} \sqrt{2 \alpha_{k}} \rd W_{k}(t)\Big)^{\top} 
\end{align*}
and 
\begin{align*}
G_2 := \Big(\int_{t_{1}}^{t_{2}} e^{-\frac{\gamma}{2}(t_{2} - t)} \sqrt{2 \gamma} \rd W_{0}(t), \int_{t_{1}}^{t_{2}} e^{-\alpha_{1}(t_{2} - t)} \sqrt{2 \alpha_{1}} \rd W_{1}(t),\cdots,\int_{t_{1}}^{t_{2}} e^{-\alpha_{k}(t_{2} - t)} \sqrt{2 \alpha_{k}} \rd W_{k}(t)\Big)^{\top},
\end{align*} 
are Gaussian random variables. In particular, $(\mathcal{P}_{n}^{2} \varphi)(y)=\mathbb{E}[\varphi(\Psi_{h}(y, G_{1}, G_{2}))]$. Define
\begin{align*}
\Psi_{h, y}: \hR^{2(k+1)} \to \hR^{k+2}, \qquad (g_{1}, g_{2}) \mapsto \Psi_{h}(y, g_{1}, g_{2}).
\end{align*} 
To establish the validity of a localized minorization condition using \cite[Lemma 6.3]{Benaime2015}, we need to verify that, for any initial point $Y_0=y \in \mathbb{R}^{k+2}$ and target point $y^* \in \mathbb{R}^{k+2}$, the state $\Psi_{h,y}\left(G_1, G_2\right)=y^*= \left(y_0^*, \ldots, y_{k+1}^*\right)^{\top}$ can be achieved by choosing $G_1$ and $G_2$, and that the Jacobian $D_g \Psi_{h, y}$ with $g=(g_{1}, g_{2})$ is required to be of full rank $k+2$.

\textbf{Step 1.} From \eqref{sdesolu}, it follows that 
\begin{align*}
&y_0^* = v_2 \, = \, e^{-\frac{\gamma h}{2}}\bar{v}_{T_{1}, t_{2}} + \int_{t_{1}}^{t_{2}} e^{-\frac{\gamma}{2}(t_{2} - t)} \sqrt{2 \gamma} \rd W_{0}(t), \\
&y_{\ell}^* = z_2^\ell \, = \, e^{-\alpha_\ell h}\bar{z}^{\ell}_{T_{1}, t_{2}} + \frac{\gamma \lambda_{\ell}}{2 \alpha_{\ell} - \gamma}(e^{-\frac{\gamma}{2} h} - e^{-\alpha_{\ell} h})\bar x_{T_1, t_2} + \int_{t_{1}}^{t_{2}} e^{-\alpha_{\ell}(t_{2} - t)} \sqrt{2 \alpha_{\ell}} \rd W_{\ell}(t), \quad \ell = 1, 2, \ldots, k,
\end{align*}
which implies that $y_0^*,\cdots,y_k^*$ are attainable regardless of the values of $\bar{v}_{T_{1}, t_{2}}$, $\bar{z}^{\ell}_{T_{1}, t_{2}}$ and $\bar x_{T_1, t_2}$ by adjusting $G_2$. In addition, by Proposition \ref{solvability}, there exists $\phi_h:\hR^{k+2}\to \hR^{k+2}$ defined in \eqref{demapphi} such that $\bar v_{T_0, t_1} = \phi_{v}(y)$, $\bar z_{T_0, t_1}^\ell = \phi_{z^\ell}(y)$ and $\bar x_{T_0, t_1} = \phi_{x}(y)$ with $\ell = 1, 2, \cdots, k$. Together with \eqref{avfmethod} and \eqref{sdesolu}, this leads to
\begin{align*}
&v_1 \, = \, e^{-\frac{\gamma h}{2}}\phi_{v}(y) + \int_{t_{0}}^{t_{1}} e^{-\frac{\gamma}{2}(t_{1} - t)} \sqrt{2 \gamma} \rd W_{0}(t),\\
&z_1^\ell \, = \, e^{-\alpha_\ell h}\phi_{z^\ell}(y) + \frac{\gamma \lambda_{\ell}}{2 \alpha_{\ell} - \gamma}(e^{-\frac{\gamma}{2} h} - e^{-\alpha_{\ell} h})\phi_{x}(y) + \int_{t_{0}}^{t_{1}} e^{-\alpha_{\ell}(t_{1} - t)} \sqrt{2 \alpha_{\ell}} \rd W_{\ell}(t), \quad \ell=1,2,\cdots,k, \\
&x_1=e^{-\frac{\gamma h}{2}}\phi_{x}(y).
\end{align*}
Consequently, $x_1$ is determined by $y$, while $v_1$ and $z_1^\ell$ can reach arbitrary values via $G_1$. Then, it follows from
\begin{align*}
&y^*_{k+1} = x_2 = e^{- \frac{\gamma h}{2}} \bar x_{T_1, t_2} =e^{- \frac{\gamma h}{2}}( x_1 + \frac{\gamma}{4}h(\bar{x}_{T_1, t_{2}} + x_{1}) + \frac{h}{2}(\bar{v}_{T_1, t_{2}} + v_{1}) ) ,\\
&\bar{v}_{T_{1}, t_{2}} = v_{1} - \frac{\gamma}{4}h(\bar{v}_{T_1, t_{2}} + v_{1}) - h \int_{0}^{1}\nabla U(x_{1} + \theta(\bar{x}_{T_1, t_{2}} - x_{1}))\rd \theta + \frac{h}{2} \sum\limits^{k}_{\ell = 1}\lambda_{\ell}(\bar{z}_{T_{1}, t_{2}}^{\ell} + z_{1}^{\ell}), \\
&\bar{z}^{\ell}_{T_{1}, t_{2}} = z^{\ell}_{1} - \frac{\lambda_{\ell}}{2}h(\bar{v}_{T_1, t_{2}} + v_{1}) - \frac{\gamma\lambda_{\ell}}{4}h(\bar{x}_{T_1, t_{2}} + x_{1}),
\end{align*}
that $y^*_{k+1}$ is attainable by adjusting $G_1$, which is not uniquely determined. Thus, $\Psi_{h, y}$ is a surjective, and for any $y^* \in \hR^{k+2}$, there exists $(g_1, g_2) \in \hR^{2(k+1)}$ such that $\Psi_{h, y}\left(g_1, g_2\right)=y^*$.

\textbf{Step 2.} It follows that
\begin{align*}
D_{g} \Psi_{h, y}(g_{1}, g_{2})=\Big[M \nabla \phi_{h}\Big(M \phi_{h}(y)+\binom{G_{1}}{0}\Big)\binom{I}{0},\binom{I}{0}\Big]=\Big[M \nabla \phi_{h}(Y_1)\binom{I}{0},\binom{I}{0}\Big],
\end{align*}
where $\nabla \phi_{h}(Y_1)$ is given by \eqref{nablaphih} and $I$ is the $(k+1) \times (k+1)$ identity matrix. Since the first $k+1$ entries of the last row of $M \nabla \phi_h\left(Y_1\right)$ are nonzero, it follows that the last row of $M \nabla \phi_h\left(Y_1\right)\binom{I}{0}$ contains nonzero entries. Consequently, $D_g \Psi_{h, y}$ has full rank $k+2$, due to Remark \ref{phi:bijection_smooth}.

By \cite[Lemma 6.3]{Benaime2015}, for any $y, y^* \in \mathbb{R}^{k+2}$, there exist open neighborhoods $J_y$ and $J_{y^*}$ of $y$ and $y^*$, respectively, and a constant $c_{y, y^*}>0$ such that
\begin{align*}
\mathcal{P}_h^2(\mathbf{1}_{J_{y^*}} \varphi)(y^{\prime})>c_{y, y^*} \int_{J_{y^*}} \varphi(y) \nu(\mathrm{d} y), \qquad \forall y^{\prime} \in J_{y}, \quad \forall \varphi \in \mathcal{C}_0(\hR^{k+2}, \hR), 
\end{align*}
where $\mathbf{1}_{J_{y^*}}$ is the indicator function of the set $J_{y^*}$. The proof is then completed by following the compactness argument presented in \cite[SM1.2]{LeimkuhlerSachs2022}.
\end{proof}

\begin{theorem} [Geometric ergodicity] 
\label{Geometric ergodicity}
Let Assumption \ref{ass.1} hold. Then, the numerical solution $\{Y_{n}\}_{n\in \mathbb{N}}$, defined by \eqref{numsolu}--\eqref{sdesolu}, admits a unique invariant measure $\tilde{\pi}$. Moreover, for any measurable function $g : \hR^{k+2} \to\hR $ satisfying $|g| \leq(H + C_{0})^{p}$ for some $p \ge 1$, there exist constants $\kappa := \kappa(p) \in (0, 1)$ and $C := C(p)>0$ such that 
\begin{align*}
\big| \hE \big[ g(Y_{n}) \big] - \tilde{\pi}(g) \big| 
\leq C \kappa^{n} \big( 1 + |H(Y_{0})|^{p} \big), \qquad \forall\, n \geq 1, 
\end{align*} 
where $H$ is defined in \eqref{newHamil}.
\end{theorem}

\begin{proof}
By combining Lemmas \ref{Numerical:Lyapunov} and \ref{Numerical:minorization} with \cite[Theorem 2.5]{MattinglyStuartHigham2002}, the proof is completed.
\end{proof}

\section{Numerical experiments}
\label{sec.experiment}

In this section, we conduct numerical experiments for the proposed splitting AVF method \eqref{numsolu}--\eqref{sdesolu} applied to the GLE \eqref{GLE} to validate the strong and weak convergence orders, as well as to investigate its long-time performance in computing ergodic limits. In all experiments, we consider the double-well potential $U(x) = \frac{x^4}{4} - \frac{x^2}{2}$ for $x \in \hR$, with parameters $k = 3$, $\lambda_\ell = 2$, $\alpha_\ell = 3$ $(\ell = 1,2,3)$, and $\gamma = 5$.

We first test the strong and weak convergence orders of the splitting AVF method \eqref{numsolu}--\eqref{sdesolu} for the GLE \eqref{GLE}, with terminal time $T = 1$ and initial value $Y(0) = (1,\cdots,1)^\top \in \hR^5$. The numerical solutions are computed with step sizes $h_i = 2^{-i}$, $i = 8, 9, 10, 11, 12$. Let $Y_h(T, \omega_j)$ denote the $j$-th sample solution at time $T$ obtained with step size $h$. The strong and weak errors are evaluated by
\begin{align*}
& \mbox{Err}_{\mbox{strong}}(h_i) = \left(\frac{1}{M} \sum_{j = 1}^{M} \big\| Y_{h_i}\left(T, \omega_j\right)-Y_{h_{i+1}}\left(T, \omega_j\right) \big\|^2 \right)^{1/2}, \\
& \mbox{Err}_{\mbox{weak}}(h_i) = \left| \frac{1}{M} \sum_{j = 1}^M \Big( g\big( v_{h_i}(T, \omega^j), x_{h_i}(T, \omega^j) \big) - g\big( v_{h_{i+1}}(T, \omega^j), x_{h_{i+1}}(T, \omega^j) \big) \Big) \right|. 
\end{align*}
Here, the test function is chosen as $g(v, x) = \sin\big((v^2 + x^2)^{1/2}\big)$, and $M = 5{, }000$ sample paths are used to approximate the expectations. Table \ref{Table:strong_weak_order} shows that the splitting AVF method \eqref{numsolu}--\eqref{sdesolu} achieves first-order mean-square convergence (left) and first-order weak convergence (right), respectively.

Next, the splitting AVF method \eqref{numsolu}--\eqref{sdesolu}, with initial value $Y(0) = (1,\cdots,1)^\top \in \hR^5$ and step size $h = 2^{-3}$, is applied to $2{,}000$ sample paths over the time interval $[0, 512]$. The distributions of the numerical solution at different times $t = 2$, $16$, $128$, and $512$ are plotted in Figures \ref{fig:ergodicity3D} and \ref{fig:ergodicity2D}, indicating that the distribution of the numerical solution converges to a stationary distribution as time increases.

Finally, we examine the long-time performance of the splitting AVF method \eqref{numsolu}--\eqref{sdesolu} in computing the ergodic limits defined by $\int_{\hR^{5}} g \, \rd \pi$, where $\pi$ is the Gibbs–Boltzmann measure given in \eqref{Gibbs--Boltzmann}. We consider three test functions $g(y) = \cos(\|y\|^2)$, $\exp(-\|y\|^2/2)$, and $\sin(\|y\|^2)$ for $y \in \hR^5$. Figure \ref{fig:emporalaverages} displays the temporal averages $\frac{1}{N} \sum_{n=1}^N \hE [g(Y_n)]$ of the numerical solution generated by the splitting AVF method \eqref{numsolu}--\eqref{sdesolu} with step size $h = 2^{-3}$, starting from four different initial values $Y^{(1)} = (-10,2,3,4,1)^\top$, $Y^{(2)} = (2,1,1,1,-10)^\top$, $Y^{(3)} = (1,-1,-1,-1,3)^\top$, and $Y^{(4)} = (4,2,3,4,2)^\top$. Here, the expectations are approximated using $2{, }000$ Monte Carlo paths. All results indicate that the temporal averages of the numerical solutions effectively approximate the reference line, corresponding to the ergodic limit.

\begin{table}[htbp] 
\centering
\setlength{\tabcolsep}{5mm}
\caption{The strong error and strong order (left) as well as the weak error and weak order (right) for the splitting AVF method \eqref{numsolu}--\eqref{sdesolu}.}
\label{Table:strong_weak_order}
\begin{tabularx}{0.9\textwidth}{ccccc}
\toprule 
\ \ 
\makecell[c]{Stepsize}
& \makecell[c]{Strong error} &\makecell[c]{Strong order} & \makecell[c]{Weak error} & \makecell[c]{Weak order} \\
\midrule
\ \ $2^{-8}$ & $9.3149e$-$01$ & $0.99$ & $2.4671e$-$04$ & $0.97$ \\
\ \ $2^{-9}$ & $4.6842e$-$01$ & $1.00$ & $1.2572e$-$04$ & $1.03$ \\
\ \ $2^{-10}$ & $2.3475e$-$01$ & $0.96$ & $6.1386e$-$05$ & $1.32$ \\
\ \ $2^{-11}$ & $1.2057e$-$01$ & $*$ & $2.4599e$-$05$ & $*$ \\
\bottomrule
\end{tabularx}
\end{table}

\begin{adjustbox}{scale = 0.7}
\indent 
\begin{minipage}{1.2\linewidth}
\centering
\begin{minipage}{0.48\linewidth} 
\vspace{3pt}
\includegraphics[width = \textwidth]{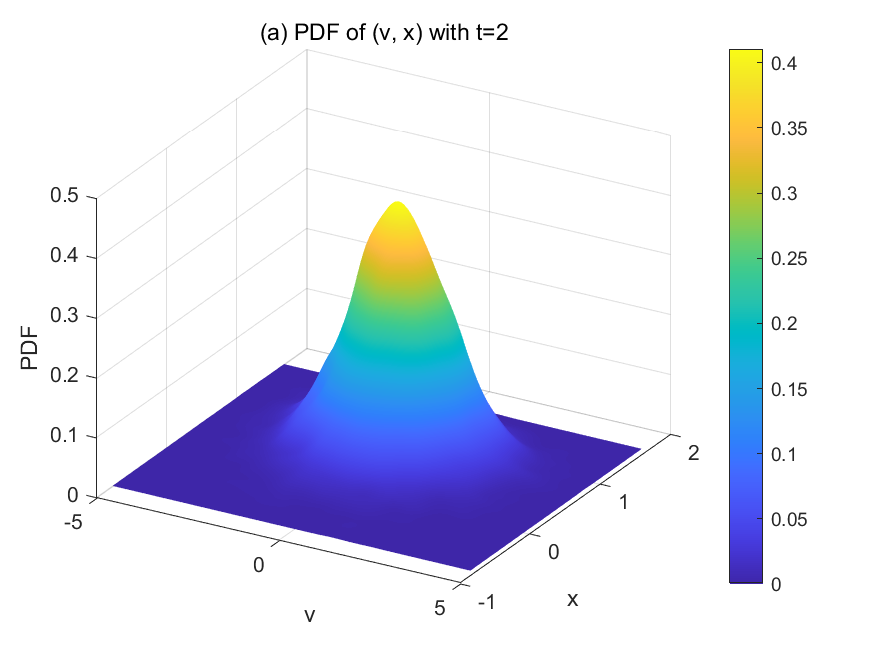}\\
\vspace{3pt}
\includegraphics[width = \textwidth]{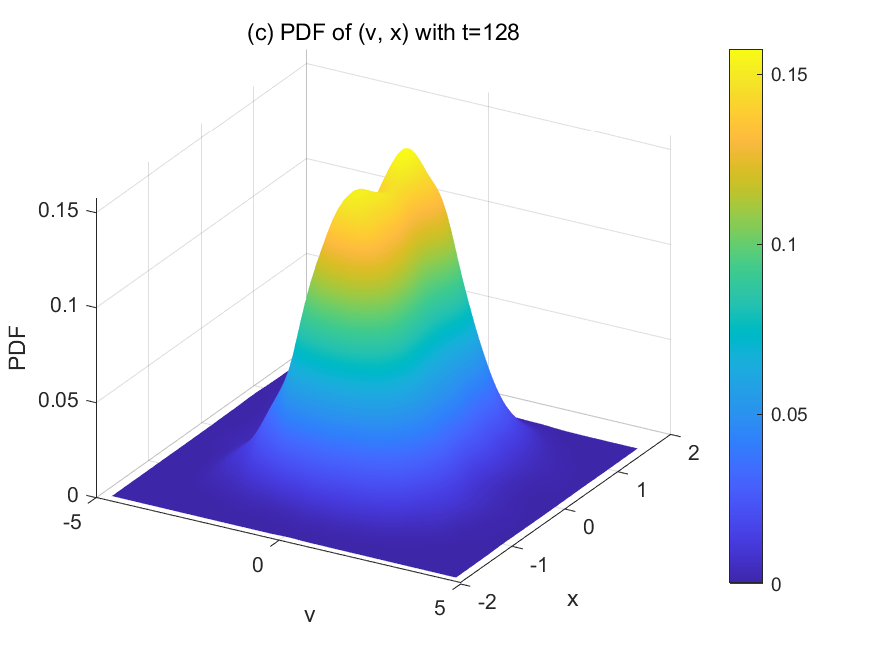}
\end{minipage}
\hfill
\begin{minipage}{0.48\linewidth}
\vspace{3pt}
\includegraphics[width = \textwidth]{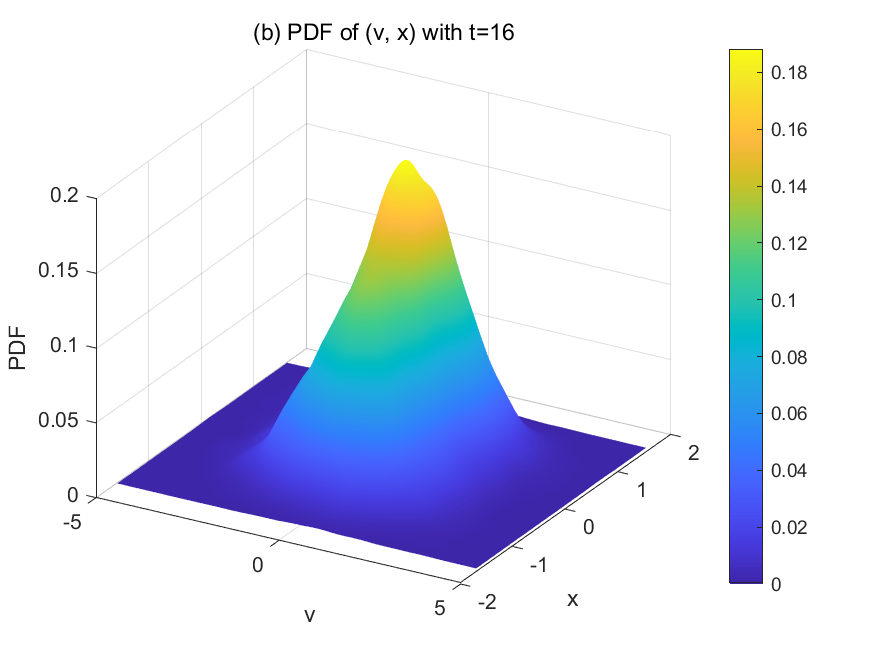}\\
\vspace{3pt}
\includegraphics[width = \textwidth]{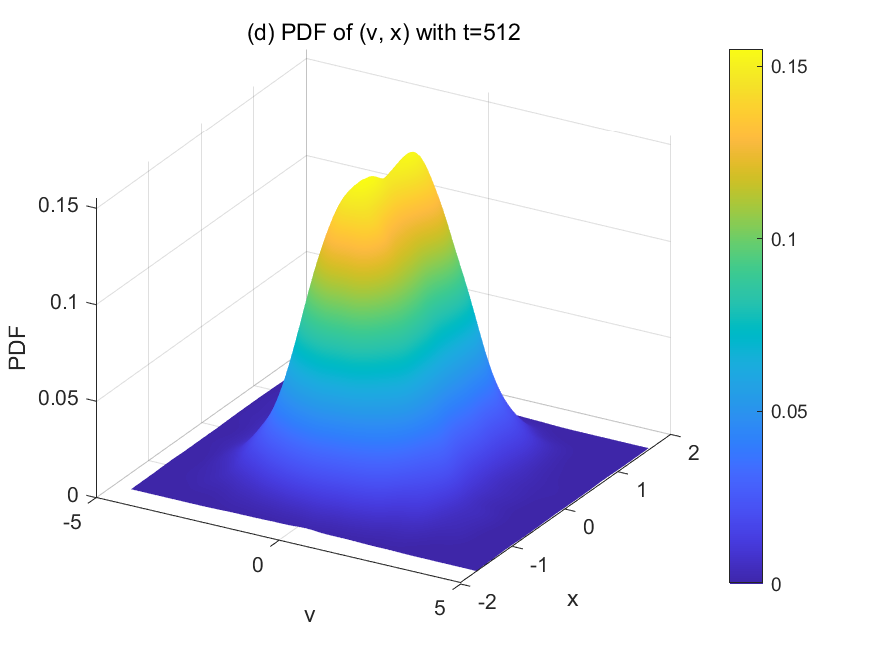}
\end{minipage}
\vskip -1.5em
\captionof{figure}{3D surface plots of the distribution at different times $t = 2$, $16$, $128$, $512$.}
\label{fig:ergodicity3D}
\end{minipage}
\end{adjustbox}

\begin{adjustbox}{scale = 0.7}
\indent 
\begin{minipage}{1.2\linewidth}
\centering
\begin{minipage}{0.48\linewidth} 
\vspace{3pt}
\includegraphics[width = \textwidth]{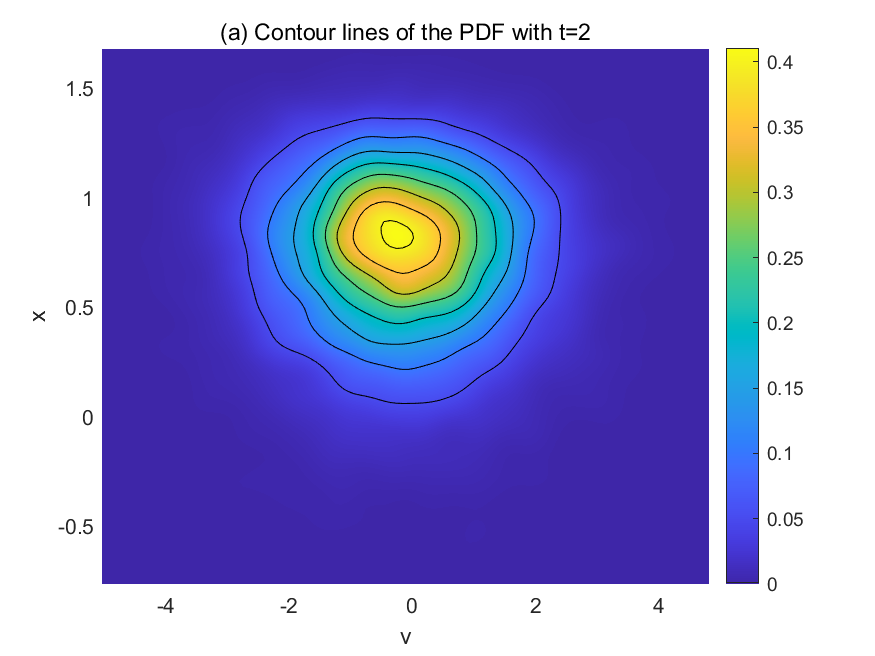}\\
\vspace{3pt}
\includegraphics[width = \textwidth]{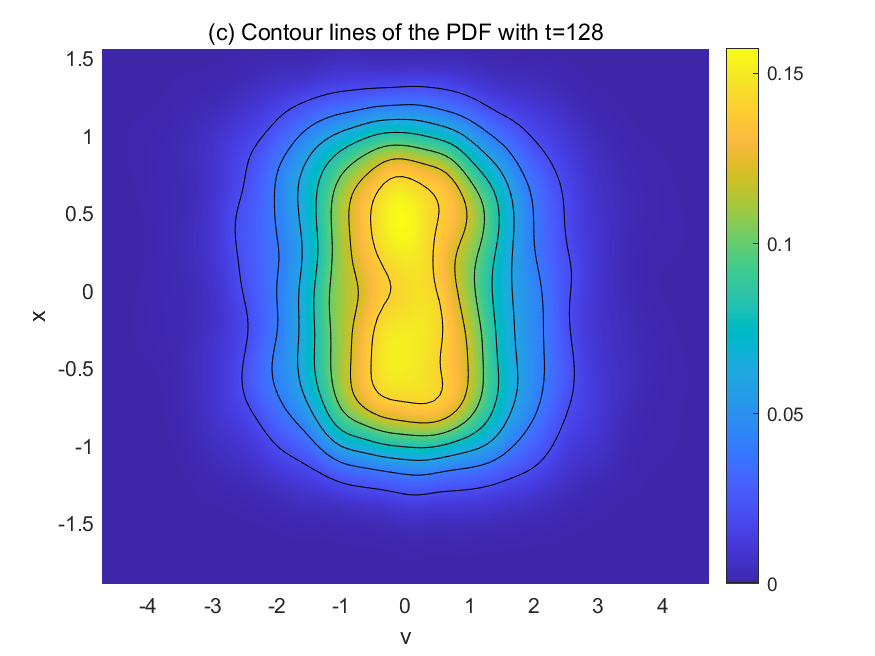}
\end{minipage}
\hfill
\begin{minipage}{0.48\linewidth}
\vspace{3pt}
\includegraphics[width = \textwidth]{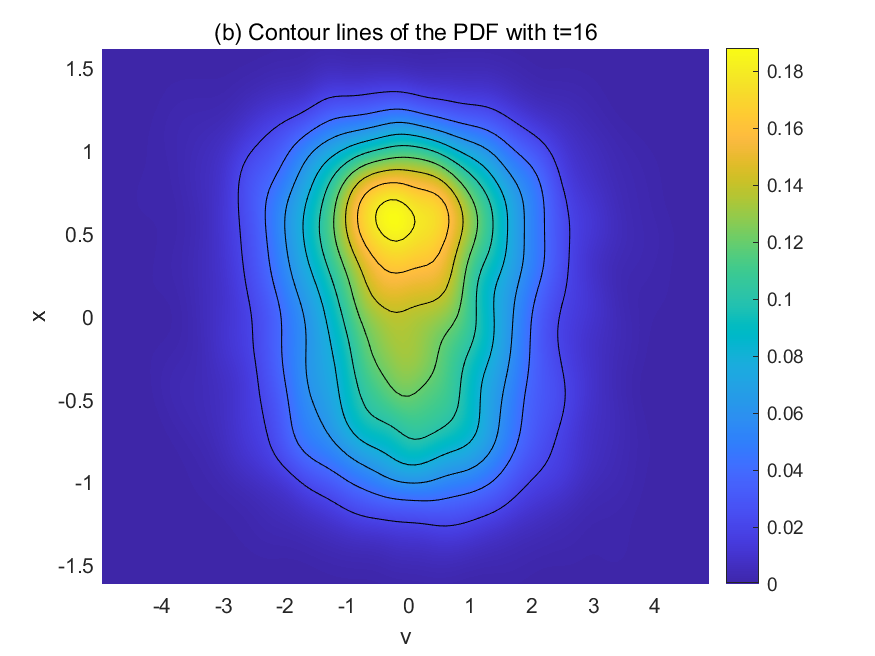}\\
\vspace{3pt}
\includegraphics[width = \textwidth]{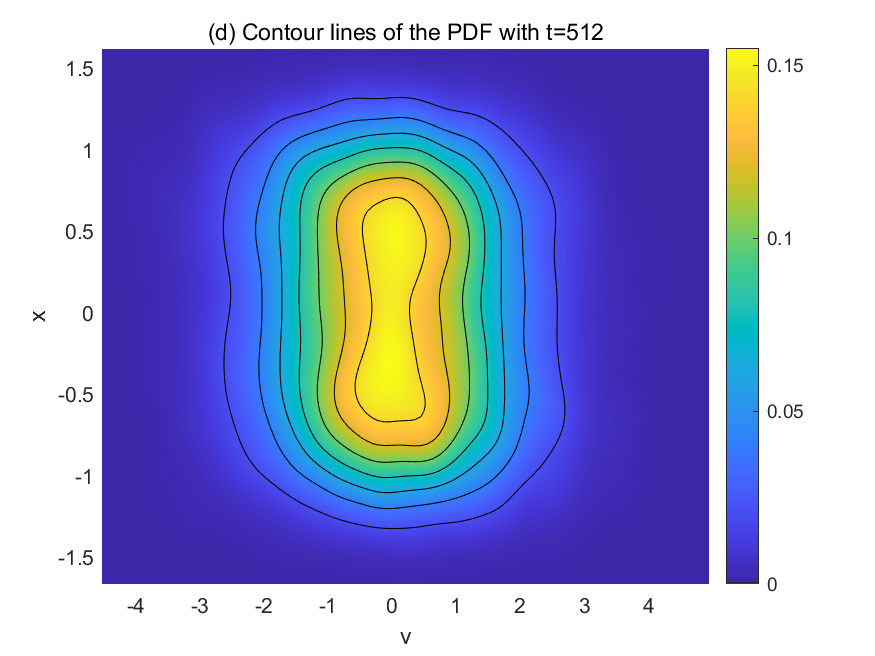}
\end{minipage}
\vskip -1.0em
\captionof{figure}{2D contour plots of the distribution at different times $t = 2$, $16$, $128$, $512$.}
\label{fig:ergodicity2D}
\end{minipage}
\end{adjustbox}

\begin{adjustbox}{scale=0.9}
\indent 
\begin{minipage}{1.0\linewidth}
\centering
\begin{minipage}{0.85\linewidth} 
\centering
(a) $g(y) = \cos(\|y\|^2)$
\includegraphics[width=\textwidth]{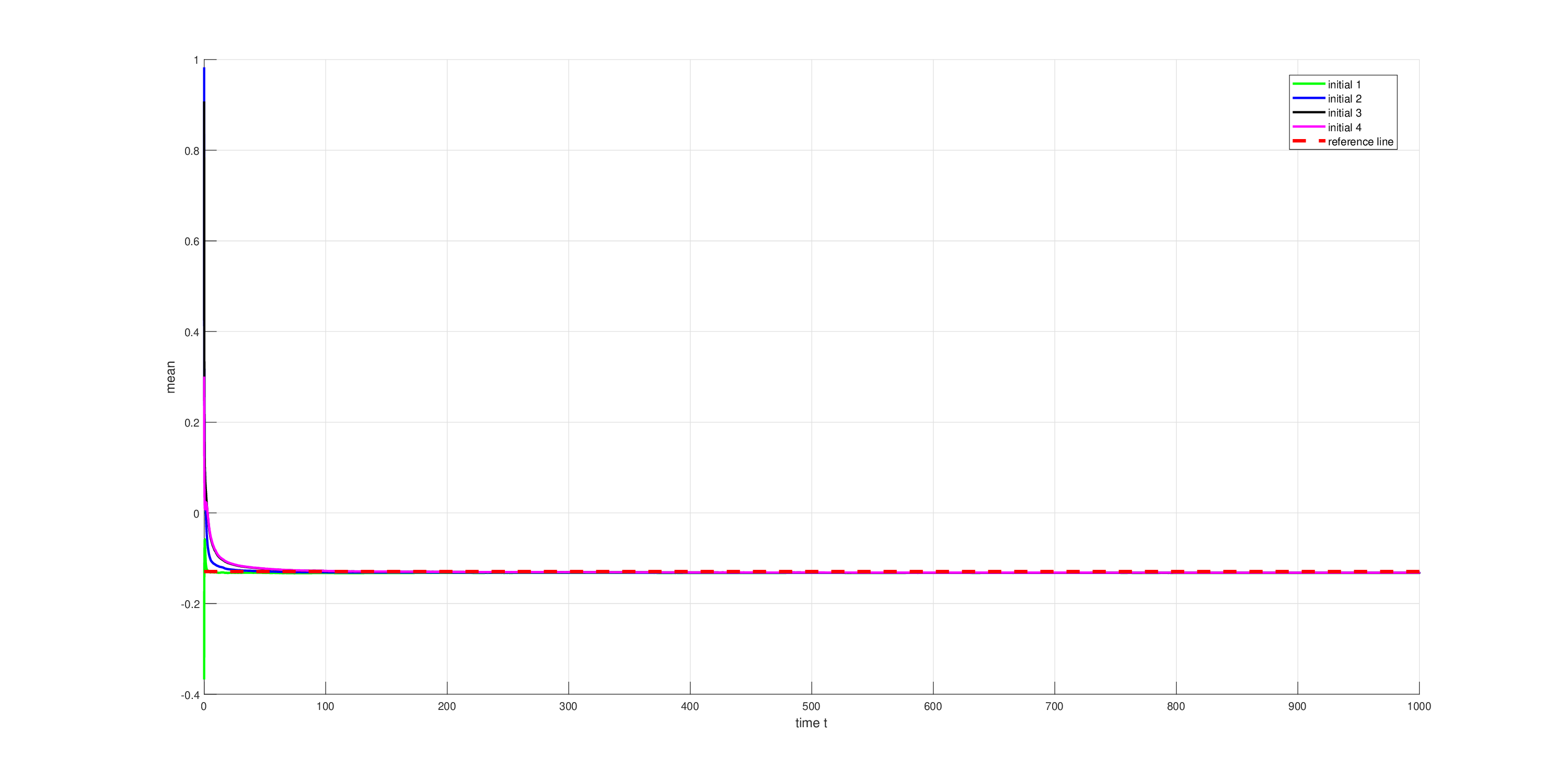}
\end{minipage}
\hfill
\begin{minipage}{0.85\linewidth}
\centering
(b) $g(y) = \exp(-\frac{\|y\|^2}{2})$
\includegraphics[width=\textwidth]{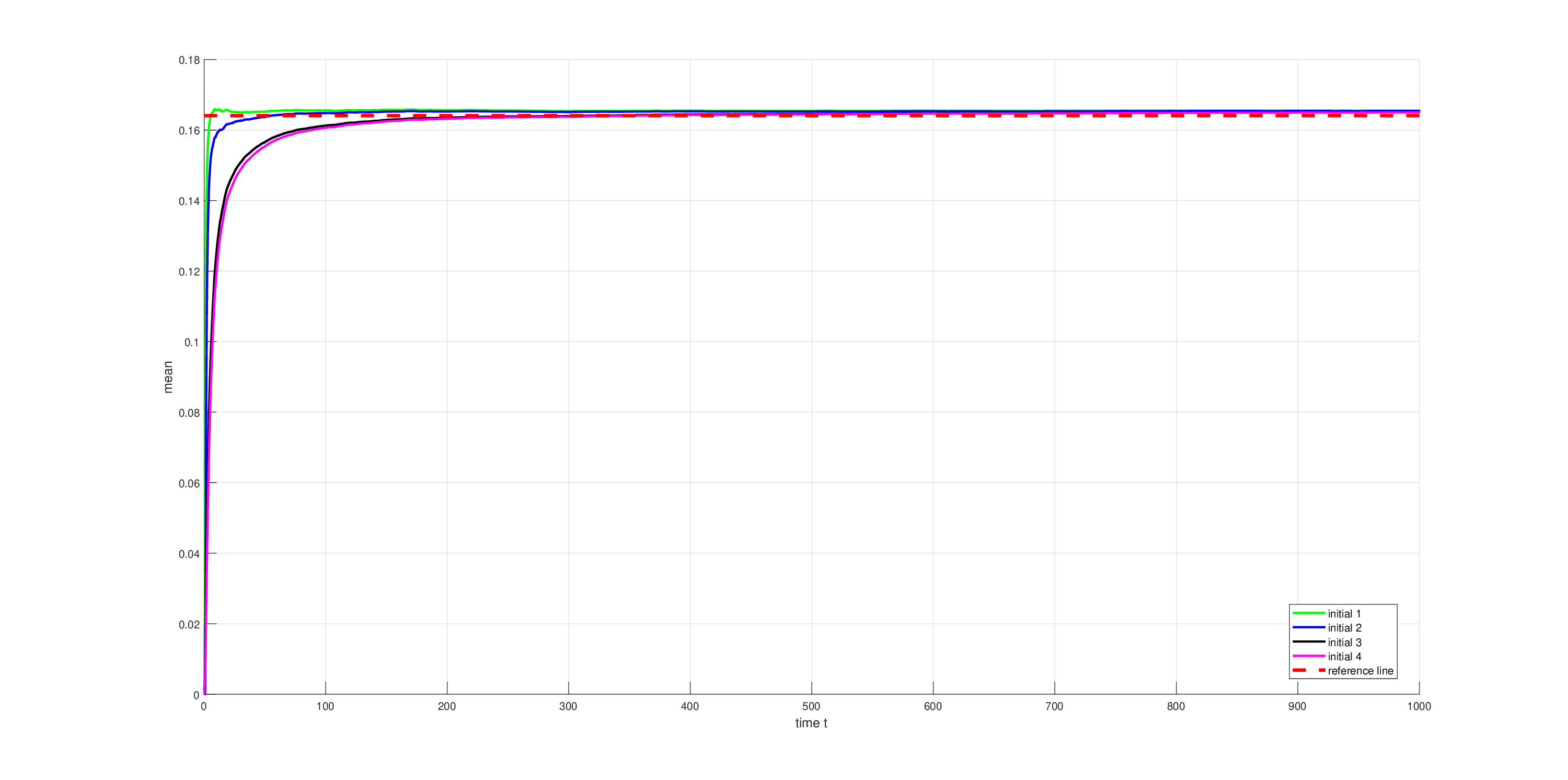}
\end{minipage}
\vspace{1em} 
\begin{minipage}{0.85\linewidth}
\centering
(c) $g(y) = \sin(\|y\|^2)$
\includegraphics[width=\textwidth]{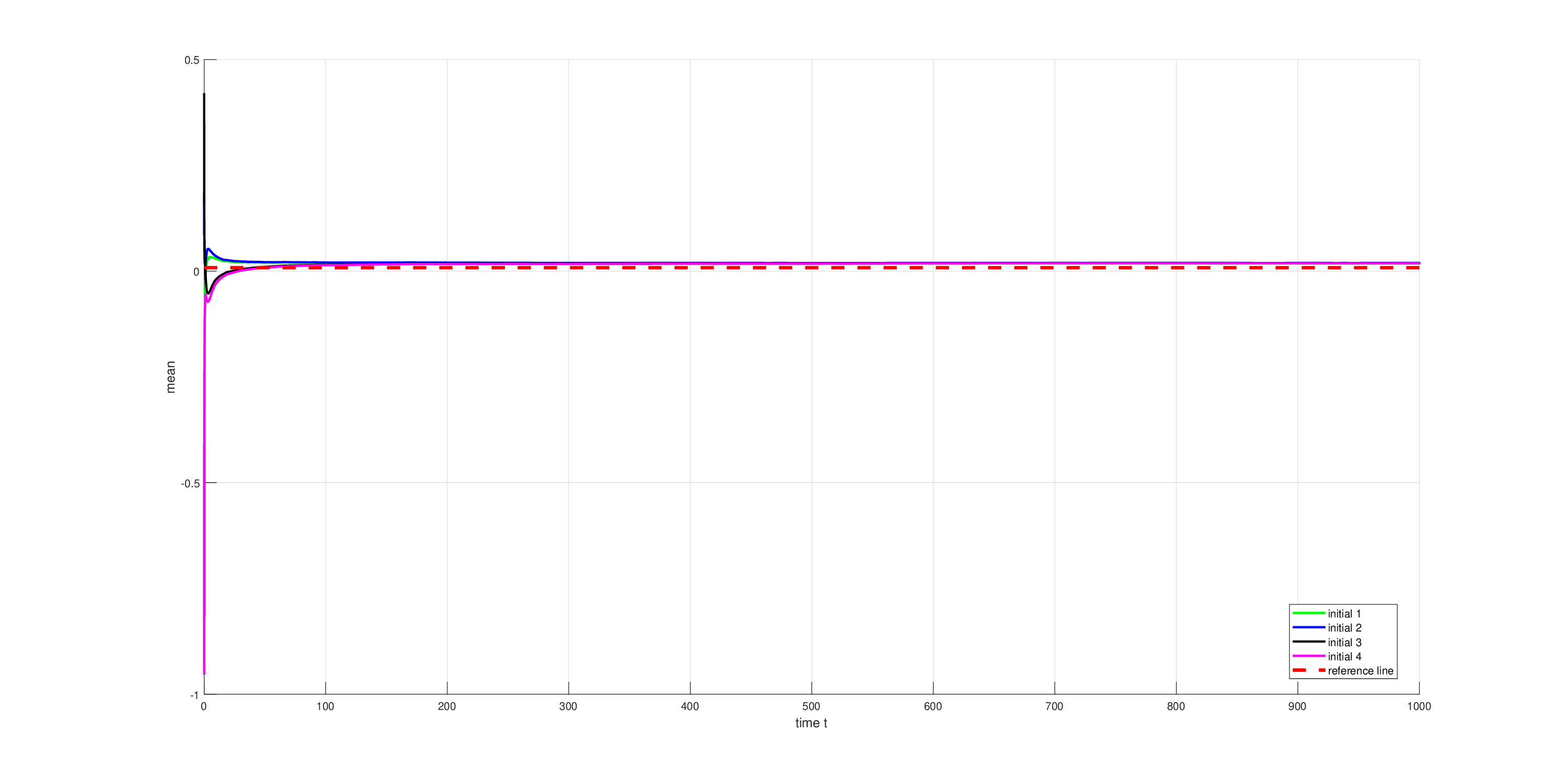}
\end{minipage}
\vspace{-0.5em}
\captionof{figure}{Temporal averages $\frac{1}{N} \sum_{n=1}^N \hE [g(Y_n)]$ for four different initial values.}
\label{fig:emporalaverages}
\end{minipage}
\end{adjustbox}



\end{document}